\documentclass{lms}
\usepackage{amssymb, amsmath, amsfonts}

\usepackage[colorlinks]{hyperref}
\hypersetup{linkcolor=blue, citecolor=red, urlcolor=green}
\usepackage[all]{xy}
\usepackage{mathabx}
\usepackage{color}
\usepackage{tikz}
\usepackage{tkz-graph}
\usepackage{vmargin}
\usepackage{multicol}
\usepackage{graphicx}

\newtheorem{theorem}{Theorem}[section] 
\newtheorem{lemma}[theorem]{Lemma}     
\newtheorem{corollary}[theorem]{Corollary}
\newtheorem{proposition}[theorem]{Proposition}

\newnumbered{assertion}{Assertion}    
\newnumbered{conjecture}{Conjecture}  
\newnumbered{definition}{Definition}
\newnumbered{hypothesis}{Hypothesis}
\newnumbered{remark}{Remark}
\newnumbered{note}{Note}
\newnumbered{observation}{Observation}
\newnumbered{problem}{Problem}
\newnumbered{question}{Question}
\newnumbered{algorithm}{Algorithm}
\newnumbered{example}{Example}
\newunnumbered{notation}{Notation} 

\numberwithin{equation}{section}

\DeclareMathOperator{\dimension}{dim}
\DeclareMathOperator{\modulo}{mod}

\DeclareMathOperator{\rep}{rep}

\DeclareMathOperator{\Hom}{Hom}

\DeclareMathOperator{\ident}{id}

\DeclareMathOperator{\nilp}{nil}
\DeclareMathOperator{\module}{-mod}
\DeclareMathOperator{\soc}{soc}
\DeclareMathOperator{\Matriz}{Mat}
\DeclareMathOperator{\diagonal}{diag}
\DeclareMathOperator{\flip}{flip}
\DeclareMathOperator{\Ext}{Ext}
\DeclareMathOperator{\ext}{ext}
\DeclareMathOperator{\Grass}{Gr}
\DeclareMathOperator{\Caminos}{Cam}
\DeclareMathOperator{\Supp}{Supp}

\DeclareMathOperator{\decrep}{decrep}
\DeclareMathOperator{\Endo}{End}
\DeclareMathOperator{\pd}{pd}
\DeclareMathOperator{\injd}{id}
\DeclareMathOperator{\firr}{decIrr^{s.r}}
\DeclareMathOperator{\Irr}{Irr}
\DeclareMathOperator{\DIrr}{decIrr}

\DeclareMathOperator{\Auto}{Aut}
\DeclareMathOperator{\rnk}{rank}
\DeclareMathOperator{\upp}{up}
\DeclareMathOperator{\bpp}{bp}

\setpapersize{A4}
\setmargins{2.5cm}              
{1.5cm}                         
{16.5cm}                        
{23.42cm}                       
{10pt}                          
{1cm}                           
{0pt}                           
{2cm}                           

\title[On a family of Caldero-Chapoton algebras]%
{On a family of Caldero-Chapoton algebras that have the Laurent phenomenon}

\author{Daniel Labardini-Fragoso and Diego Velasco}

\classno{16G20 (primary), 13F60 (secondary)}

\extraline{\textbf{Key words}: Generalized cluster algebras, surfaces with orbifold points, flip, triangulations, arc representations, Caldero-Chapoton algebras.\\
 \textbf{Affiliation}: Daniel Labardini-Fragoso and  Diego Velasco (corresponding author), Instituto de Matem\'aticas, Ciudad Universitaria,  04510 Ciudad de M\'exico, M\'exico. e-mail: labardini@matem.unam.mx and diego.velasco@matem.unam.mx}

\begin{document}
\maketitle
\begin{abstract}
We realize a family of generalized cluster algebras as Caldero-Chapoton algebras of quivers with relations. Each member of this family arises from an unpunctured polygon with one orbifold point of order 3, and is realized as a Caldero-Chapoton algebra of a quiver with relations naturally associated to any triangulation of the alluded polygon. The realization is done by defining for every arc $j$ on the polygon with orbifold point a representation $M(j)$ of the referred quiver with relations, and by proving that for every triangulation $\tau$ and every arc $j\in\tau$, the product of the Caldero-Chapoton functions of $M(j)$ and $M(j')$, where $j'$ is the arc that replaces $j$ when we flip $j$ in $\tau$, equals the corresponding exchange polynomial of Chekhov-Shapiro in the generalized cluster algebra. Furthermore, we show that there is a bijection between the set of generalized cluster variables and the isomorphism classes of $E$-rigid indecomposable decorated representations of $\Lambda$.
\end{abstract}
\section{Introduction}

The \emph{cluster algebras} of Sergey Fomin and Andrei Zelevinsky have pervaded several 
areas of Mathematics and even Physics in the past 15 years. They are defined through a 
recursive process by iterating an algebraic-combinatorial operation called 
\emph{cluster mutation}. Each cluster mutation produces a rational function by means of 
an \emph{exchange binomial} dictated by a skew-symmetrizable matrix. Cluster algebras 
satisfy the remarkable \emph{Laurent phenomenon}: all of the rational functions 
obtained during the process can be expressed as Laurent polynomials, cf 
\cite{FZ-02,FZ-07}. 

In \cite{CS-14}, Leonid Chekhov and Michael Shapiro have defined the notion of 
\emph{generalized cluster algebra}. The cluster mutation rule inside a generalized 
cluster algebra allows exchange polynomials to not be binomials, in contrast to cluster 
algebras, where all exchange polynomials are required to be binomials. Generalized 
cluster algebras, possess the remarkable Laurent phenomenon. 
Chekhov-Shapiro have shown that  the ``\emph{Lambda length} coordinate ring'' of a 
surface with marked points and arbitrary-order orbifold points carries a natural 
structure of generalized cluster algebra, thus generalizing previous works of Penner, 
Fomin-Shapiro-Thurston and Felikson-Shapiro-Tumarkin. Thomas Lam and Pavlo Pylyavskyy 
have shown that the generalized cluster algebras of Chekhov-Shapiro fit into a more 
general framework of \emph{Laurent phenomenon algebras} (\emph{LP-algebras} for short).

Works of Caldero-Chapoton, Caldero-Chapoton-Schiffler, Derksen-Weyman-Zelevinsky, Palu, 
Plamondon and Schiffler, among others, have established a very fruitful 
representation-theoretic approach to cluster algebras through representations of 
quivers. By now, it is very well known that every \emph{cluster monomial} in a 
skew-symmetric cluster 
algebra can be expressed as the \emph{Caldero-Chapoton function} of a representation of 
a suitable quiver, and that, for a non-degenerate quiver with potential $(Q,S)$, the 
\emph{Caldero-Chapoton algebra} of the Jacobian algebra of $(Q,S)$ sits between the 
cluster algebra and the \emph{upper cluster algebra} of $Q$, see \cite[Proposición 7.1]{CLS-15}. In this paper we prove 
that the generalized cluster algebra associated by Chekhov-Shapiro to a polygon with 
one orbifold point of order 3 is equal to the Caldero-Chapoton algebra of a quiver with 
relations naturally arising from the referred polygon.

Let us  give a more specific context for the paper.  Let $Q$ be a quiver of type ADE, 
and let $\Lambda=\mathbb{C}\langle Q \rangle$ be  the path algebra of $Q$ over $
\mathbb{C}$. In \cite{CCh-05}, Caldero-Chapoton defined a Laurent polynomial $
\mathcal{C}_{\Lambda}(M)$ (Caldero-Chapoton function for us) for every representation 
$M$ of $\Lambda$ and show that there is a bijection, given by $\mathcal{C}_{\Lambda}(-)
$, between the cluster variables of the cluster algebra  $\mathcal{A}_{Q}$ and the 
indecomposable representations of $\Lambda$. Recall that the cluster algebras are 
defined by an inductive procedure called cluster mutation, see \cite{FZ-02,FZ-07}, but 
the Caldero-Chapoton functions are defined without  an inductive procedure. The 
coefficients of $\mathcal{C}_{\Lambda}(M)$ are given by the complex Euler 
charateristics of Grassmannians of submodules of quiver representations. The class of 
cluster algebras which cluster variables have been described with this idea has been 
extended,  for example see \cite{DWZ-10,GLS-12,Pa-08,Pla-10}. 

In \cite{CLS-15}, the authors define the Caldero-Chapoton algebra associated to a 
\emph{basic algebra} $\Lambda$ and take account the \emph{strongly reduced irreducible 
components } $\firr(\Lambda)$ of decorated representations of $\Lambda$. With this data 
they define a  \emph{generic cluster structure} on $\firr(\Lambda)$, namely they define 
the  	\emph{component clusters} and $CC$-\emph{clusters} of $\Lambda$, as 
generalizations of clusters and cluster variables in cluster algebras, respectively.
The strongly reduced components were introduced in \cite{GLS-12}. In \cite{Pla-13} the 
strongly reduced components are  parametrized in terms of the $g$-vector of 
irreducible components for finite-dimensional algebras. In \cite{CLS-15}  the results 
of \cite{Pla-13} are generalized for \emph{basic algebras}. A comment deserve to be 
done about \emph{basic algebras}. The definition given in \cite{CLS-15} is not the 
usual one, we kindly ask to the reader to  see Definition \ref{DefBasicAlg} and be 
cautious.

These algebras with cluster structures have been related with surfaces in different 
contexts. Recently Charles Paquette and Ralf Schiffler introduced some notion of 
\emph{generalized cluster algebra} that is different from the one mentioned above.
In \cite{FeST-14}, cluster algebras are related with surfaces with marked points and 
orbifold points of order two. In \cite{W-16}, non-orientable surfaces are related with 
LP-algebras. In \cite{CS-14}, some surfaces with arbitrary many orbifold points of 
arbitrary order  are related with cluster generalized algebras. The interaction between 
the cluster algebraic structures and cluster  combinatorial structures have been 
extensively studied, for instance  \cite{PL-12,S-08}. In this paper we show one more of those interactions. 

In order to write our main result, see Theorem \ref{MainResult}, let us introduce some 
notation. Let  $\Sigma_n$ be a  $(n+1)$-polygon with one orbifold point of order three. 
For any triangulation $\tau$ of $\Sigma_n$ we define an algebra $\Lambda(\tau)$. This 
algebra is given by a quiver associated to $\tau$ and relations given by the internal 
triangles of the triangulation. Our main result is:

\begin{theorem}\label{Teo1.1}
For any triangulation $\tau$ of $\Sigma_n$  the Caldero-Chapoton algebra of $
{\Lambda(\tau)}$ is a generalized cluster algebra naturally associated to $\Sigma_n$.
\end{theorem}
 
The paper is organized as follows. In the first five sections we recall some facts 
about Caldero-Chapoton algebras, Galois $G$-covering functors, string algebras, cluster 
generalized algebras and surfaces with marked points and orbifold points of order 3 
that we need for stating and proving our results.  In Section \ref{section1} we recall 
the non-standard definition of \emph{basic} algebra introduced in \cite{CLS-15} and its 
representations. We also recall some facts about Galois $G$-coverings that will be 
crucial to deal with the $E$-invariant and $g$-vectors of arc representations with 
respect to an arbitrary triangulation of $\Sigma_n$. 
Section \ref{section2} is dedicated to recall the definition of Caldero-Chapoton 
algebra, see Example \ref{ExamMotiva}. In Section \ref{section3} we recall some results 
of \cite{BR-87} about string algebras, the concept of string is useful for us because 
the algebra $\Lambda(\tau)$ associated to a triangulation $\tau$ of $\Sigma_n$ is a 
string algebra. 

In Section \ref{Section4} we recall the definition of generalized cluster algebras 
introduced in \cite{CS-14}, the reader can find interesting relations of these algebras 
with cluster algebras in \cite{N-15,NR-15}.
In Section \ref{Section5} we recall some definitions and facts about surfaces with 
marked points and orbifold points. Also we define a \emph{natural Jacobian algebra} for 
any triangulation denoted by $\Lambda(\tau)$.

The goal of Section \ref{Sect14} is to present the $\Lambda(\tau)$ as an \emph{orbit 
Jacobian algebra}, see \cite{PS-17} and,   give the definition of the  arc 
representations of $\Lambda(\tau)$, which are (decorated) $\Lambda(\tau)$-modules that 
one can associate to arbitrary arcs on $\Sigma_n$.
In Section \ref{Section6} we study the arc representations of $\Lambda(\tau_0)$ for a 
specific choice of triangulation $\tau_0$; we show that on these representations the 
action of the Auslander-Reiten translation of $\Lambda(\tau_0)$ is given by 
rotation,\footnote{In sync with previous results of Br\"{u}stle-Qiu \cite{BQ} and 
Caldero-Chapoton-Schiffler \cite{CCS-04}.} compute their $g$-vectors, show that they 
are $E$-rigid, and prove that their orbits are dense in their respective irreducible 
components.

In Section \ref{Section13} we show that the results of Section \ref{Section6} hold for 
any triangulation $\tau$ of $\Sigma_n$ and not only for the specific triangulation $
\tau_0$. Here, Galois coverings, both of quivers and surfaces, come into play. 
 
In Section \ref{Sect15}  we state and prove our main result for any triangulation $\tau
$.  We close the paper with Section \ref{Section12}, which contains an example with 
explicit computations illustrating our main result; the example can be thought of as a 
complement to \cite[Example 9.4.2]{CLS-15}.
 
Sections \ref{Section6} and \ref{Section13} deserve a few words: all the results in 
Section \ref{Section6}, about the $E$-invariant, are particular cases of results from 
Section \ref{Section13}. We have decided to not omit  \ref{Section6}, and rather 
present a particular instance followed by the general treatment, because the Galois 
covering techniques used in Section \ref{Section13} are not needed when establishing 
the same results for the specific triangulation $\tau_0$ chosen in \ref{Section6}.

\section{Background}\label{section1}
In this section we fix notation and we recall some basic definitions  and facts about
algebras and quiver representations that  we will use throughout the work. The reader
can find more details in \cite{CLS-15}.

A \textit{quiver}  $Q=(Q_0, Q_1, t, h)$ consists of a finite set of vertices $Q_0$, a
finite set of arrows $Q_1$ and two maps $t,h: Q_1\rightarrow Q_0$ (tail and head). For
each $a\in Q_1$ we write $a:t(a)\rightarrow h(a)$. 
If $Q_0=\{1,\ldots, n\}$, we define the \emph{skew-symmetric matrix} $C_Q=(c_{i,j})\in
 \Matriz_{n\times n}(\mathbb{Z})$ from $Q$ as follows 
\[c_{i,j}=\left|\{a\in Q_1: h(a)=i, t(a)=j\}\right|-\left|\{a\in Q_1: h(a)=j, t(a)=i\}
\right|.\]

We say that a sequence of arrows  $\alpha=a_la_{l-1}\cdots a_2a_1$, is a \textit{path}
of $Q$ if  $t(a_{k+1})=h(a_{k})$, for $k=1, \ldots l-1$, in this case, we define  the 
\textit{length} of $\alpha$ as $l$. We say that $\alpha$ is a \textit{cycle} if 
$h(a_l)=t(a_1)$. In this work we deal with quivers with loops, that is, quivers where 
there is an arrow $a\in Q_1$ such that $h(a)=t(a)$.  

For $m\in \mathbb{N}$ let $C_m$ be the set of \textit{paths of length} $m$ and let $
\mathbb{C}C_m$ be the vector space with basis $C_m$.
The \textit{path algebra } of  a quiver $Q$ is denoted by $\mathbb{C}\langle Q\rangle$
and it is defined as $\mathbb{C}$-vector space as 
\[\mathbb{C}\langle Q\rangle=\bigoplus_{m\geq 0}{\mathbb{C}C_m},\]
where the product is given by the concatenation of paths. The \textit{completed path
algebra} of a quiver $Q$ is defined as vector space as 
\[\mathbb{C}\langle\langle Q\rangle\rangle=\prod_{m\geq 0}{\mathbb{C}C_m},\]
where the elements are written as infinite sums $\sum_{m\geq 0}{x_m}$ with $x_m\in
\mathbb{C}C_m$ and the product in $\mathbb{C}\langle\langle Q\rangle\rangle$ is
defined as
\[(\sum_{l\geq 0}{b_l})(\sum_{m\geq 0}{a_m})=\sum_{k\geq 0}{\sum_{l+m= k}{b_la_m}}.\]

Let $\mathfrak{M}=\prod_{m\geq 1}{\mathbb{C}C_m}$ be the two-sided ideal of $\mathbb{C}
\langle\langle Q \rangle \rangle$ generated by arrows of $Q$. Then $\mathbb{C} \langle
\langle Q \rangle \rangle$ can be viewed as a \textit{topological} 
$\mathbb{C}$-\textit{algebra} with the powers of $\mathfrak{M}$ as a basic system of 
openneighborhoods of $0$. This topology is known as $\mathfrak{M}$-\textit{adic 
topology}. Let  $I$ be a subset of  $\mathbb{C}\langle\langle Q \rangle \rangle$, we 
can calculate the closure of $I$ as 
$\overline{I}=\bigcap_{l\geq0}{(I+\mathfrak{M}^l)}$.

\begin{definition}\label{DefBasicAlg}
A two-sided ideal $I$ of $\mathbb{C}\langle\langle Q \rangle \rangle$ is 
\textit{semi-admissible} if $I\subseteq \mathfrak{M}^2$ and it is \emph{admissible} if 
some power of 
$\mathfrak{M}$ is a subset of $I$. Following \cite{CLS-15} we call an algebra $\Lambda$ 
\textit{basic} if $\Lambda=\mathbb{C}\langle\langle Q \rangle \rangle/I$ for some 
quiver $Q$ and some semi-admissible ideal $I$.
\end{definition}

A \textit{finite-dimensional representation} of $Q$ over $\mathbb{C}$ is a pair $
((M_i)_{i\in Q_0}, (M_a)_{a\in Q_1})$ where $M_i$ is a finite-dimensional 
$\mathbb{C}$-vector space for each $i\in Q_0$ and $M_a: M_{t(a)}\rightarrow M_{h(a)}$ 
is a $
\mathbb{C}$-linear map. Here the word \textit{representation} means finite-dimensional 
representation. 

The \textit{dimension vector} of a representation $M$ of $Q$ is given by $
\underline{\dimension}(M)=(\dimension(M_1),\ldots, \dimension(M_n)) $. We define $
\dimension(M)=\sum_{i=1}^{n}{\dimension(M_i)}$ as the \textit{dimension} of $M$. We say  
$M$ is a  \textit{nilpotent} representation if there is an $n>0$ such that for every  
path $a_na_{n-1}a\ldots a_1$ of length $n$ in $Q$  we have $M_{a_n}M_{a_{n-1}}\ldots 
M_{a_1}=0$. A \textit{subrepresentation} of $M$ is an $n$-tuple of $\mathbb{C}$-vector 
spaces $N=(N_i)_{i\in Q_0}$ such that $N_i\leq M_i$ for each $i\in Q_0$ and 
$M_a(N_{t(a)})\subseteq N_{h(a)}$ for every $a\in Q_0$.

We denote by $\nilp_\mathbb{C}(Q)$ the category of nilpotent representations of $Q$, 
and by $\mathbb{C}\langle\langle Q\rangle\rangle$-$\modulo$ the category of 
finite-dimensional left $\mathbb{C}\langle\langle Q\rangle\rangle$-modules. It is known 
that 
the category of representations of $Q$ and the category of $\mathbb{C}\langle Q \rangle
$-modules are equivalent. In \cite[Section 10]{DWZ-10} it was observed  that  $\nilp_
\mathbb{C}(Q)$ and $\mathbb{C}\langle\langle Q\rangle\rangle$-$\modulo$ are equivalent. 
 
Given a basic algebra $\Lambda=\mathbb{C}\langle\langle Q \rangle \rangle/I$ we define 
a \textit{representation} of $\Lambda$ as a nilpotent representation of $Q$ which is 
annihilated by $I$. We consider the category $\modulo (\Lambda)$ of finite-dimensional 
left modules as the category $\rep(\Lambda)$ of representations of $\Lambda$. 

Let $\Lambda=\mathbb{C}\langle\langle Q \rangle \rangle/I$ be a basic algebra.  We say 
$\mathcal{M}=(M,V)$ is a \textit{decorated representation} of $\Lambda$ if $M$ is a 
representation of $\Lambda$ and $V=(V_1,\ldots, V_n)$ is an n-tuple of 
finite-dimensional $\mathbb{C}$-vector spaces. We can think of $V$ as a representation 
of a 
quiver with $n$-vertices and no arrows. That is a representation of the semisimple 
ring $\mathbb{C}^{Q_0}$.
Let $\decrep(\Lambda)$ be the \emph{category of decorated representations} of 	$
\Lambda$. The objects of $\decrep(\Lambda)$ are  the decorated representations of $
\Lambda$ and its morphisms are given as follows. Let $(M,V)$ and $(N,W)$ be two 
decorated representations of $\Lambda$. We define the space of morhisms in $
\decrep(\Lambda)$ by
 $\Hom_{\decrep(\Lambda)}((M,V), (N,W))=\Hom_{\rep(\Lambda)}(M,N)\times
 \Hom_{\rep(\mathbb{C}^{Q_0})}(V,W)$.

Let $\mathcal{M}=(M,V)$ be a decorated representation of $\Lambda$. If $V=0$, we write 
$M$ instead $\mathcal{M}$. For $i\in\{1,\ldots, n\}$ we define the \textit{negative 
simple representation} of $\Lambda$ as $\mathcal{S}_i^-=(0,S_i)$ where $(S_i)_j$ is $
\mathbb{C}$ if $j=i$ and $(S_i)_j=0$ in other wise.

For a representation $M=((M_i)_{i\in Q_0}, (M_a)_{a\in Q_1})$ of $\Lambda$ and a vector 
$\textbf{e}\in \mathbb{N}^n$ let $\Grass_{\textbf{e}}(M)$ be the \textit{quiver 
Grassmannian} of subrepresentations $N$  of $M$ such that $\underline{\dimension}(N)=
\textbf{e}$. We denote the \textit{Euler characteristic} of $\Grass_{\textbf{e}}(M)$ by 
$\chi(\Grass_{\textbf{e}}(M))$. About Euler characteristic we are going to need the 
following result, see \cite{Bi-73},
 
\begin{lemma}[(Bia\'lynicki-Birula)]\label{LemmaBi}
Let $T$ be an algebraic torus acting on an algebraic variety $X$. If we denote by $X^{T}
$ the set of fixed points of the action,  then $\chi(X^{T})=\chi(X)$.
\end{lemma}

The following definition plays a crucial role in some computations that will be 
involved later. It was introduced in  \cite[Section 2.4]{CLS-15}. 

\begin{definition}\label{Def_Lambda_p}
Given a basic algebra $\Lambda=\mathbb{C}\langle\langle Q \rangle \rangle/I$ and $p\geq 
2$ we define the $p$-truncation of $\Lambda$  by
$\Lambda_p=\mathbb{C}\langle\langle Q \rangle \rangle/(I+\mathfrak{M}^p)$.
\end{definition}

We are going to need some basic definitions about \emph{quivers with potential}, for 
all details the reader can see \cite{DWZ-08}.
Let $Q$ be  a quiver, we say that $S\in \mathbb{C}\langle\langle Q\rangle\rangle$ is a 
\textit{potential} for $Q$ if $S$ is a, possibly infinite, $\mathbb{C}$-linear 
combination  of cycles in $Q$ . Given two potentials $S$ and $W$  we say that they are 
\textit{cyclically equivalent} and write $S\sim_{\operatorname{cyc}} W$, if $S-W$ is in 
the closure of the sub-vector space of $\mathbb{C}\langle\langle Q\rangle\rangle$   
generated by all elements of the form $a_1a_2\cdots a_{n-1}a_n -a_2\cdots a_{n-1}
a_na_1$, with  $a_1a_2\cdots a_{n-1}a_n$ a cycle on $Q$.

\begin{definition}
We say $(Q,S)$ is a \textit{quiver with potential} (QP) if $S$ is a potential for $Q$ 
and if any two different cycles appearing with non-zero coefficient   in $S$ are not  
cyclically equivalent.
\end{definition}

Given an arrow $a\in Q_1$ and a cycle $a_na_{n-1}\cdots a_1$ in $Q$, define the 
\textit{cyclic derivative} of $a_na_{n-1}\cdots a_1$ with respect to $a$ as follows:
\[\partial_{a}(a_na_{n-1}\cdots a_1)=\sum\limits_{k=1}^{n}{\delta_{a, a_k}a_{k-1}
a_{k-2}\cdots a_1a_na_{n-1}\cdots a_{k+2}a_{k+1}},\]
we extend this definition by $\mathbb{C}$-linearity and continuity to all potentials 
for $Q$.
\begin{definition}
Let $(Q,S)$ be a quiver with potential. We define the \textit{Jacobian ideal}  $
\mathcal{J}(Q,S)$ as the closure of the ideal on $\mathbb{C}\langle\langle Q\rangle
\rangle$ generated by  all cyclic derivatives $\partial_a(S)$ with $a\in Q_1$. The 
quotient $\mathbb{C}\langle\langle Q\rangle\rangle/J(Q,S)$ is called the 
\textit{Jacobian algebra} of $(Q,S)$ and is denoted as $\mathcal{P}(Q,S)$.
\end{definition}
 
 \subsection{Varieties of representations}\label{SectionVR}
  
Let $\Lambda=\mathbb{C}\langle\langle Q \rangle \rangle/I$ and $\textbf{d}=(d_1,
\ldots, d_n)\in \mathbb{N}^n$ be a basic algebra and a vector of non-negative integers. 
The representations $M$ of $Q$ with $\underline{\dimension}(M)=\textbf{d}$ can be seen 
as points of the affine space 
\[\rep_{\textbf{d}}(Q)=\prod_{a\in Q_1}{\Hom_{\mathbb{C}}}(\mathbb{C}^{d_{t(a)}},
\mathbb{C}^{d_{h(a)}} ).\]
Now, let $\rep_{\textbf{d}}(\Lambda)$ be the  Zariski closed subset of $
\rep_{\textbf{d}}(Q)$ given by the representations $N$ of $\Lambda$ with $
\underline{\dimension}(N)=\textbf{d}$.

In $\rep_{\textbf{d}}(\Lambda)$ we have the action of $G_{\textbf{d}}=\prod_{i=1}^{n}
{GL(\mathbb{C}^{d_i})}$ by conjugation. If $g=(g_1,\ldots, g_n)\in G_{\textbf{d}}$ and 
$M=((M_i)_{i\in Q_0}, (M_a)_{a\in Q_1})\in \rep_{\textbf{d}}(\Lambda)$, then 
\[g\cdot M=((M_i)_{i\in Q_0}, (g_{h(a)}M_ag_{t(a)^{-1}})_{a\in Q_1}).\]

From definitions follows that the isomorphism classes of representations of $\Lambda$ 
with dimension vector $\textbf{d}$  are in bijection with the $G_{\textbf{d}}$-orbits 
in $\rep_{\textbf{d}}(\Lambda)$. If $M\in \rep_{\textbf{d}}(\Lambda)$, its 
$G_{\textbf{d}}$-orbit is denoted by $\mathcal{O}(M)$.
For $(\textbf{d}, \textbf{v})\in \mathbb{N}^n\times \mathbb{N}^n$ let $
\decrep_{\textbf{d},\textbf{v}}(\Lambda)$ be the \textit{decorated representations 
variety} of $\Lambda$. If $\textbf{v}=(v_1,\ldots, v_n)$, then 
\[\decrep_{\textbf{d},\textbf{v}}(\Lambda)=\rep_{\textbf{d}}(\Lambda)\times \{( 
\mathbb{C}^{v_1}, \ldots, \mathbb{C}^{v_n})\}.\]

We have an action of $G_{\textbf{d}}$ on  $\decrep_{\textbf{d},\textbf{v}}(\Lambda)$ 
given by $g\cdot\mathcal{M}= (g\cdot M,V)$ where $\mathcal{M}=(M,V)\in
\decrep_{\textbf{d},\textbf{v}}(\Lambda)$ and $g\in G_{\textbf{d}}$.

\subsection{Galois G-covering}

In this section we are going to make a reminder of Galois $G$-covering. For a nice review 
of this theory the reader can see the introductions of \cite{As-11,BL-14}. For our 
convenience we are going to present some results from \cite{BL-14}.
 
In this section $G$ will denote a finite group (in the general theory this assumption 
is not required). A category  $\mathcal{A}$ is  called $\mathbb{C}$-\emph{linear} or  
$\mathbb{C}$-\emph{category} if its sets of morphisms are $\mathbb{C}$-modules 
and the composition of morphisms is $\mathbb{C}$-linear. We assume that  we have a 
morphism 
$\rho:G\rightarrow \Auto(\mathcal{A})$ from $G$ to the group of automorphisms of $
\mathcal{A}$, not the group of auto-equivalences. That means we have an \emph{action} 
of $G$ on $\mathcal{A}$. We will abuse of notation and we will write $g$ instead $
\rho(g):\mathcal{A}\rightarrow \mathcal{A}$ for every $g\in G$. The action of $G$ on $
\mathcal{A}$ is called \emph{free} provided $g\cdot X$ is not isomorphic to $X$ for 
every non-trivial element $g\in G$ and for any indecomposable object $X$ of $
\mathcal{A}$. 
The next definitions are due to Asashiba, see \cite[Definition 1.1, Definition 1.7]
{As-11}.

\begin{definition}
Let $\mathcal{A}$ and $\mathcal{B}$ be $\mathbb{C}$-categories with $G$ acting on $
\mathcal{A}$. A $\mathbb{C}$-linear functor  $\mathcal{F}:\mathcal{A}\rightarrow 
\mathcal{B}$ is called $G$-\emph{stable} if there exist  functorial isomorphisms 
$\delta_g:\mathcal{F}g\rightarrow 
\mathcal{F}$ such that $\delta_{h,X}\delta_{g, h\cdot X}=\delta_{gh, X}$ for any $g,h
\in G$ and any object  $X$ in $\mathcal{A}$. In this case $\delta=(\delta_g)_{g\in G}$ 
is called a $G$-\emph{stabilizer}. If $\delta_g=\ident_{\mathcal{F}}$ for every $g\in G
$, we say that $\mathcal{F}$ is $G$-\emph{invariant}, see the next diagram.
\[\xymatrix{\mathcal{F}gh \ar@{->}[dr]_{\delta_{gh}}\ar@{->}[r]^{\delta_g} & 
\mathcal{F}h \ar@{->}[d]^{\delta_{h}} \\
              \ \  &  \mathcal{F}}\]
\end{definition}

\begin{definition}
Let $\mathcal{A},\mathcal{B}$ be $\mathbb{C}$-categories with a group $G$ acting on $
\mathcal{A}$. Let $\mathcal{F}:\mathcal{A}\rightarrow \mathcal{B}$ be a $G$-stable 
functor with 
stabilizer $\delta$.
\begin{enumerate}
\item[(a)] We say that $\mathcal{F}$ is a $G$-\emph{precovering} if the following maps are 
isomorphisms for any $X, Y$ objects in $\mathcal{A}$:
\[
\mathcal{F}_{X,Y}:\bigoplus_{g\in G}\mathcal{A}(X, g\cdot Y)\rightarrow \mathcal{B}
(\mathcal{F}(X),\mathcal{F}(Y)); \ (u_g)_{g\in G}\mapsto \sum_{g\in G}{\delta_{g, Y}
\mathcal{F}(u_g)},
\]
\[
\mathcal{F}^{X,Y}:\bigoplus_{g\in G}\mathcal{A}(g\cdot X,Y)\rightarrow \mathcal{B}
(\mathcal{F}(X),\mathcal{F}(Y)); \ (v_g)_{g\in G}\mapsto \sum_{g\in G}{\mathcal{F}(v_g)
\delta^{-1}_{g, X}}.
\]
\item[(b)] A $G$-precovering $\mathcal{F}$ is called a \emph{Galois} $G$-\emph{covering} if 
$\mathcal{F}$ has the following three conditions:
\begin{itemize}
\item[(i)] The functor $\mathcal{F}$ is \emph{almost dense}. It means that any 
indecomposable object $Y$ of $\mathcal{B}$ is isomorphic to $\mathcal{F}(X)$ for some 
object $X$ in 
$\mathcal{A}$.
\item[(ii)] If $X$ is indecomposable in $\mathcal{A}$, then $\mathcal{F}(X)$ is 
indecomposable in $\mathcal{B}$.
\item[(iii)] For any indecomposable objects $X, Y$ in $\mathcal{A}$ such that $
\mathcal{F}(X)\cong \mathcal{F}(Y)$, there exist $g\in G$ such that $g\cdot X \cong Y$.
\end{itemize}
\end{enumerate} 
\end{definition}

\begin{remark}
\begin{enumerate}
\item In \cite[Proposition 1.6]{As-11} is proved that $\mathcal{F}^{X,Y}$ is an 
isomorphism if and only if $\mathcal{F}_{X,Y}$ is an isomorphism. Note that a 
$G$-precovering is a faithful functor, see \cite[Lemma 2.6]{BL-14}.
\item In Krull-Schmidt categories a functor is almost dense if and only if it is dense.
\end{enumerate}
\end{remark}

The following lemma allows us to find examples of a Galois $G$-covering from a 
$G$-precovering between  module categories, see \cite[Lemma 2.9]{BL-14}

\begin{lemma}\label{LemBL1}
Let $\mathcal{A}, \mathcal{B}$ be Krull-Schmidt $\mathbb{C}$-categories with a group $G
$ acting freely on $\mathcal{A}$ and let $\mathcal{F}:\mathcal{A}\rightarrow\mathcal{B}
$ be a $G$-precovering. Assume $X$ is an object in $\mathcal{A}$ such that $
\Endo_{\mathcal{A}}(X)$ is local and  has nilpotent radical. Then $
\Endo_{\mathcal{B}}(\mathcal{F}(X))$ is local with nilpotent radical and if $Y$ is an 
object in $\mathcal{A}$ such that $\mathcal{F}(X)\cong F(Y)$, then there exist $g\in G$ such 
that $g\cdot X \cong Y$.
\end{lemma}

The next theorem shows an interesting application of Galois $G$-covering in 
Auslander-Reiten theory, see \cite[Theorem 3.7]{BL-14}. 

\begin{theorem}[(Bautista-Liu, 2014)]\label{BLTheorem}
Let $\mathcal{A}, \mathcal{B}$ be Krull-Schmidt $\mathbb{C}$-categories with a group $G
$ acting freely on $\mathcal{A}$ and let $\mathcal{F}:\mathcal{A}\rightarrow\mathcal{B}
$ be a Galois $G$-covering. Then 
\begin{enumerate}
\item A short exact sequence $\eta$ in $\mathcal{A}$ is almost split if and only if $
\mathcal{F}(\eta)$ is almost split.
\item An object $X$ in $\mathcal{A}$ is the starting or ending term of a almost split 
sequence if and only if $\mathcal{F}(X)$ is the starting or ending term of a almost 
split sequence, respectively.
\end{enumerate}
\end{theorem}

Given a $\mathbb{C}$-algebra $\Lambda$ we consider it as a $\mathbb{C}$-category in the 
usual way, in other words, the set of objects of $\Lambda$ is a complete family of 
ortogonal and 
primitive  idempotents and the set of morphisms is given by $\Lambda(e_i,e_j)=e_j
\Lambda e_i$. In this context, the category of left $\Lambda$-modules $\Lambda\module$ 
is equivalent to the category of functors from $\Lambda$ to $\mathbb{C}\module$. An 
action of a group $G$ on $\Lambda$ induces an action of $G$ on $\Lambda\module$ in the 
following way. Given an $\Lambda$-module $M:\Lambda\rightarrow \mathbb{C}\module$ we 
define $g\cdot M:= Mg^{-1}$, remember that $g$ is thought as an automorphism of $
\Lambda$; for a morphism $u:M\rightarrow N$ of $\Lambda\module$, we define $g\cdot 
u(x)=u(g^{-1}x)$ for $x$ an object of $\Lambda$.  So, if we have a $G$-precovering $
\pi:\Lambda\rightarrow A$, Bongarzt and Gabriel defined the \emph{push-down functor} $
\pi_{*}:\Lambda\module \rightarrow A\module$, see \cite[Section 3.2]{BG-82}. In Remark 
\ref{DefPush} we define the push-down fuctor in our particular case.

We have a nice property for $\pi_{*}$, see \cite[Lemma 6.3]{BL-14}.

\begin{lemma}
Let $\Lambda$ and $A$ be finite dimensional $\mathbb{C}$-algebras with a group $G$ acting 
on $\Lambda$. Assume the action of $G$ is free. If $\pi: \Lambda\rightarrow A$ is a 
Galois $G$-precovering, then the push-down functor $\pi_{*}$ admits a $G$-stabilizer $
\delta$.
\end{lemma}

With the following lemma we can construct a $G$-precovering from a Galois $G$-covering, 
see \cite[Theorem 6.5]{BL-14}.

\begin{lemma}\label{G-precovering}
Let $\Lambda$ and $A$ be finite dimensional $\mathbb{C}$-algebras with a group $G$ 
acting on $\Lambda$. Assume the action of $G$ is free. If $\pi: \Lambda\rightarrow A$ 
is a Galois $G$-covering, then the push-down functor 
$\pi_{*}:\Lambda\module \rightarrow A\module$ is a $G$-precovering.
\end{lemma}

\section{E-invariant and Caldero-Chapoton functions}\label{section2}
In this section we recall some definitions that we work with. This definitions were 
introduced in \cite[Section 3.4]{CLS-15}. They were motivated by the theory of 
\textit{mutation of quivers with potential} developed in \cite{DWZ-10} and the 
Caldero-Chapoton \textit{functions} introduced in \cite{CCh-05}. In this section let 
$\Lambda=
\mathbb{C}\langle\langle Q \rangle \rangle/I$ a basic algebra.

\subsection{$g$-vectors}\label{section3.1}
For a decorated representation $\mathcal{M}=(M,V)$ of $\Lambda$ the $g$-\textit{vector} 
of $\mathcal{M}$ is given by $g_{\Lambda}(\mathcal{M})=(g_1,\ldots, g_n)$ where
\[g_i:=g_i(\mathcal{M})=-\dimension\Hom_{\Lambda}(S_i,M)+\dimension\Ext_{\Lambda}
^1(S_i,M)+\dimension(V_i).\]
It is clear that $g_{\Lambda}(\mathcal{M})\in \mathbb{Z}^{n}$. We denote by $I_i$ 
the \emph{injective envelope} of the simple representation $S_i$ in $\Lambda\module$. 
We recall  an interesting result to compute the $g$-vector of $\mathcal{M}$, see 
\cite[Lemma 3.4]{CLS-15} for a general version.   

\begin{lemma}\label{gVectorLemma}
Let $\mathcal{M}=(M,V)$ be a decorated representation of a finite dimensional algebra $
\Lambda$  and let $g_{\Lambda}(\mathcal{M})=(g_1,\ldots, g_n)$ be its $g$-vector. 
Assume we have a minimal injective presentation of $M$
\[
0\rightarrow M\rightarrow I_{0}(M)\rightarrow I_{1}(M),
\]
where $I_{0}(M)=\bigoplus_{i=1}^{n}{I_i^{a_i}}$ and $I_{1}(M)=\bigoplus_{i=1}^{n}
{I_i^{b_i}}$.
Then 
\[
g_i=-a_i+b_i+\dimension(V_i).
\]
\end{lemma}

\subsection{The $E$-invariant}
For decorated representations $\mathcal{M}=(M,V)$  and $\mathcal{N}=(N,W)$ of  $\Lambda
$ let 
\[E_{\Lambda}(\mathcal{M},\mathcal{N})=\dimension\Hom_{\Lambda}(M,N)+\sum_{i=1}^{n}
{\dimension(M_i)g_i(\mathcal{N})}.\]
The $E$-\textit{invariant} of $\mathcal{M}$   is defined as $E_{\Lambda}
(\mathcal{M})=E_{\Lambda}(\mathcal{M},\mathcal{M})$.

In \cite{CLS-15} it  was shown that the $E$-invariant has a homological interpretation 
in terms of the Auslander-Reiten translation of truncations of $\Lambda$, see 
Definition \ref{Def_Lambda_p}.

\bigskip
\begin{proposition}[{\cite[Proposition 3.5]{CLS-15}}]\label{PropoTrunc}
 Let $\mathcal{M}=(M,V)$  and $\mathcal{N}=(N,W)$ be decorated representations of $
\Lambda$. If $p>\dimension(M),\dimension(N)$, then
\[E_{\Lambda}(\mathcal{M},\mathcal{N})=E_{\Lambda_p}(\mathcal{M},\mathcal{N})=
\dimension\Hom_{\Lambda_p}(\tau_{\Lambda_p}^{-}(N),M)+\sum_{i=1}^{n}{\dimension(M_i)
\dimension(W_i)}.\]
\end{proposition}
This proposition is quite useful for us because the basic algebras we are considering 
satisfy $\Lambda_p=\Lambda$ for a sufficiently large  $p$.

\subsection{Caldero-Chapoton functions and algebras}

Let $\mathcal{M}=(M,V)$ be a decorated representation of $\Lambda$. For $\textbf{f}
=(f_1,\ldots, f_n)\in \mathbb{Z}^n$ by $\textbf{x}^\textbf{f}$ we mean $\prod_{i=1}^{n}
{x_i^{f_i}}$.
The \textit{Caldero-Chapoton function} (CC function for short) associated to $
\mathcal{M}$  of $\Lambda$ is the Laurent polynomial in $n$-variables $x_1,\ldots, x_n$ 
defined by
\[\mathcal{C}_{\Lambda}(\mathcal{M})=\textbf{x}^{g_{\Lambda}(\mathcal{M})}
\sum_{\textbf{e}\in\mathbb{N}}{\chi(\Grass_{\textbf{e}}(M))\textbf{x}^{C_Q\textbf{e}}},
\]
where $C_Q$ is defined  as in the second paragraph of Section \ref{section1}. From definitions we have $
\mathcal{C}_{\Lambda}(\mathcal{M})\in \mathbb{Z}[x_1^{\pm}, \ldots, x_n^{\pm}]$, Note $
\mathcal{C}_{\Lambda}(\mathcal{S}_i^{-})=x_i$.
The \textit{set of Caldero-Chapoton} functions associated to $\Lambda$ is 
\[
\mathcal{C}_{\Lambda}=\{\mathcal{C}_{\Lambda}(\mathcal{M})\colon \mathcal{M}\in 
\decrep(\Lambda)\}.
\]

The next lemma was shown in \cite{CLS-15}. It is convenient for computations of
 $g$-vectors and Caldero-Chapoton functions.

\begin{lemma}[{\cite[Lemma 4.1]{CLS-15}}]\label{Lemma4.1}
 If $\mathcal{M}=(M,V)$  and $\mathcal{N}=(N,W)$ are decorated representations of $
\Lambda$, then the following hold:
\begin{enumerate}
\item $g_{\Lambda}(\mathcal{M}\oplus\mathcal{N})=g_{\Lambda}(\mathcal{M})+ g_{\Lambda}
(\mathcal{N})$.
\item $\mathcal{C}_{\Lambda}(\mathcal{M})=\mathcal{C}_{\Lambda}(M,0)\mathcal{C}
_{\Lambda}(0,V)$.
\item $\mathcal{C}_{\Lambda}(\mathcal{M}\oplus\mathcal{N})=\mathcal{C}_{\Lambda}
(\mathcal{M})\mathcal{C}_{\Lambda}(\mathcal{N})$.
\end{enumerate}
\end{lemma}

\begin{definition}
The \textit{Caldero-Chapoton algebra} $\mathcal{A}_{\Lambda}$ associated to $\Lambda$ 
is the $\mathbb{C}$-subalgebra of $\mathbb{C}[x_1^{\pm}, \ldots, x_n^{\pm}]$ generated 
by $\mathcal{C}_{\Lambda}$.
\end{definition}

From Lemma \ref{Lemma4.1}(iii) follows that $\mathcal{C}_{\Lambda}$ generates  $
\mathcal{A}_{\Lambda}$ as $\mathbb{C}$-vector space, see \cite[Lemma 4.2]{CLS-15}.

\begin{example}\label{ExamMotiva}
Let $Q$ be the quiver 
\begin{equation*}
\xymatrix{  
1\ar@{->}[r]^{a_1} & 2\ar@{->}[r]^<(0.3){a_2}&  \cdots\ar@{->}[r]^<(0.3){a_{n-2}}  & 
n-1\ar@{->}[r]^<(0.4){a_{n-1}}&n   
}
\end{equation*}
and let  $\Lambda=\mathbb{C}\langle Q\rangle$. For each sub-interval $\textbf{e}=[i,j]$  
of $[1,n]$ with $i\leq j$ we define an indecomposable representation $M_{\textbf{e}}$ 
of $\Lambda$ the following way. Let $(M_{\textbf{e}})_k=\mathbb{C}$ if $k\in \textbf{e}
$ and $(M_{\textbf{e}})_k=0$ if $k\notin \textbf{e}$, for $k\in [1,n]$. For an arrow 
$a_l$  with $l\in [1, n-1]$ define $(M_{\textbf{e}})_{a_l}$ as $\ident_{\mathbb{C}}$ if 
$t(a_l), h(a_l)\in \textbf{e}$ and zero in other wise.
Note that the dimension vector of $M_{\textbf{e}}$ can be identified with the 
sub-interval $\textbf{e}$. If $1<i$, then we have

\begin{equation*}
g_{\Lambda}(M_{\textbf{e}})_k=
\left\{
\begin{array}{lll}
-1 \mbox{ if } k=j,\\
\quad 1 \mbox{ if } k=i-1,\\
\quad 0 \mbox{ in other wise}.
\end{array}
\right.
\end{equation*}

If $i=1$, then $g_{\Lambda}(M_{\textbf{e}})_j=-1$  and $g_{\Lambda}
(M_{\textbf{e}})_k=0$ for $k\neq j$.
We have $E_{\Lambda}(M_{\textbf{e}})=0$  for each sub-interval $\textbf{e}$  of $[1,n]
$.
From \cite[Theorem 3.4]{CCh-05} we have $\mathcal{A}_{\Lambda}$ can be identified with 
the \textit{cluster algebra} associated to $Q$.
\end{example}

\section{Strongly reduced components}\label{SRsection}

In this section we recall some facts about \emph{strongly 
reduced irreducible components} which were introduced in \cite[Section 1.5]{GLS-12}. 
For our convenience we follow the exposition of \cite{CLS-15},
the reader can see \cite[Sections 5 and 6]{CLS-15} for a complete 
treatment about strongly reduced components in Caldero-Chapoton algebras.

Let $\Lambda$ be a basic algebra and consider dimension vectors 
$(\textbf{d},\textbf{v})\in \mathbb{N}^n\times \mathbb{N}^n$.
We denote by $\Irr_{\textbf{d}}(\Lambda)$ and 
$\DIrr_{\textbf{d},\textbf{v}}(\Lambda)$ the set of \emph{irreducible} 
components of $\rep_{\textbf{d}}(\Lambda)$ and 
$\decrep_{\textbf{d},\textbf{v}}(\Lambda)$ respectively. For 
$Z\in \DIrr_{\textbf{d},\textbf{v}}(\Lambda)$ we write 
$\underline{\dimension}(Z)=(\textbf{d},\textbf{v})$. We define 
\[
\Irr(\Lambda)=\bigcup_{\textbf{d}}{\Irr_{\textbf{d}}(\Lambda)} 
\mbox{ and }
\DIrr(\Lambda)=\bigcup_{(\textbf{d},\textbf{v})}
{\DIrr_{\textbf{d},\textbf{v}}(\Lambda)}
\]
the corresponding  sets of irreducible components.
From Section \ref{SectionVR} we have that  
\[
\decrep_{\textbf{d},\textbf{v}}(\Lambda)=\rep_{\textbf{d}}(\Lambda)
\times \{( \mathbb{C}^{v_1}, \ldots, \mathbb{C}^{v_n})\}.
\]
It is clear that $T:\decrep_{\textbf{d},\textbf{v}}(\Lambda)\rightarrow 
\rep_{\textbf{d}}(\Lambda)$ with $(M, \mathbb{C}^{\textbf{v}})\mapsto M$ 
is an isomorphism of affin varieties. In this way results on the variety 
of representations can be transported to the variety of decorated ones.
We introduce further notation in order to define strongly reduced 
components.  

Let  $Z,Z_1, Z_2\in \DIrr(\Lambda)$ be irreducible components, define

\begin{align*} 
c_{\Lambda}(Z)&=\min \{\dimension(Z)-\dimension \mathcal{O}(\mathcal{M})\colon \mathcal{M}\in Z \},\\
e_{\Lambda}(Z)&=\min\{\dimension \Ext^{1}_{\Lambda}(M,M) \colon \mathcal{M}=(M,V)\in Z\},\\
\ext^{1}_{\Lambda}(Z_1,Z_2)&=\min \{\dimension \Ext^{1}_{\Lambda}(M_1,M_2): \mathcal{M}_i=(M_1,V_i)\in Z_1, \mbox{ for } i=1,2 \}.\\
\end{align*}
From the semi-continuity of the functions 
$\dimension \Hom_{\Lambda}(-,?)$ and $\dimension 
\Ext^1_{\Lambda}(-,?)$, see \cite[Lemma 4.3]{CBS-02}, there 
exist an open set $U$ of $Z$ such that 
$E_{\Lambda}(\mathcal{M})=E_{\Lambda}(\mathcal{N})$ for all 
$\mathcal{M}, \mathcal{N}\in U$. Then we define 
$E_{\Lambda}(Z)=E_{\Lambda}(\mathcal{M})$ for 
$\mathcal{M}\in U$. In a similar way we define $E_{\Lambda}(Z_1,Z_2)$. 

The following lemma is proved in \cite[Lemma 5.2]{CLS-15}.

\begin{lemma}
Let  $Z,Z_1, Z_2\in \DIrr(\Lambda)$ be irreducible components. The following inequalities hold
\[
c_{\Lambda}(Z)\leq e_{\Lambda}(Z)\leq E_{\Lambda}(Z) \ \mbox{ and } \ 
\ext^1_{\Lambda}(Z_1,Z_2)\leq E_{\Lambda}(Z_1,Z_2).
\]
\end{lemma}

The following definition comes from \cite{GLS-12}.
\begin{definition}
Let  $Z\in \DIrr(\Lambda)$ be an irreducible component. We say that
$Z$ is \emph{strongly reduced} if $c_{\Lambda}(Z)= E_{\Lambda}(Z)$.
\end{definition}

We say that an irreducible component $Z\in \Irr(\Lambda)$ (resp. $Z\in 
\DIrr(\Lambda)$) is \emph{indecomposable} if there exist a dense  open  
$U \subseteq Z$ which contains  only indecomposable representations 
(resp. indecomposable decorated representations). 

Crawley-Boevey and Schr\"oer gave a \emph{canonical decomposition}
at the level  of irreducible components. This seems something as \emph{the
Krull-Schmidt property} at that level, see \cite[Theorems 1.1 and 1.2]{CBS-02}.

\begin{theorem}[(Crawley-Boevey-Schr\"oer)]
Let  $Z_1, \ldots, Z_t$ be irreducible components in $\Irr(\Lambda)$. The 
following two statements  are equivalent:
\begin{itemize}
    \item $\overline{Z_1\oplus\cdots \oplus Z_t}$ is an irreducible 
    component.
    \item $\ext^1_{\Lambda}(Z_i,Z_j)=0$, for $i\neq j$ with $i,j\in \{1, 
    \ldots, t\}$.
\end{itemize}
Moreover, the following hold
\begin{itemize}
    \item If $W\in \Irr(\Lambda)$ is an irreducible component, then there
    exist indecomposable irreducible components $W_1,\ldots, W_t$ in $ 
    \Irr(\Lambda)$ such that $W=\overline{W_1\oplus\cdots \oplus W_t}$ 
    and this decomposition is unique up to a permutation on $\{1,\ldots ,t\}$.
\end{itemize}
\end{theorem}

We wrote down the Crawley-Boevey-Schr\"oer theorem in its original 
version for $\rep(\Lambda)$, but it is true for $\decrep(\Lambda)$ as is stated in 
\cite[Theorem 5.3]{CLS-15}.

Now, we get something similar for strongly reduced components, see 
\cite[Theorem 5.11]{CLS-15}.

\begin{theorem}[(CI-LF-S)]\label{CDIsr}
Let $Z_1, \ldots Z_t$ be irreducible components in $\DIrr(\Lambda)$. The 
following two statements are equivalent:
\begin{itemize}
    \item $\overline{Z_1\oplus\cdots \oplus Z_t}$ is a strongly 
    reduced irreducible component.
    \item For any $Z_i$ we have that it is strongly reduced and $ 
    E_{\Lambda}(Z_i, Z_j)=0$ for all $i\neq j$ with $i,j\in \{1, 
    \ldots, t\}$.
\end{itemize}
\end{theorem}

We want generic Caldero-Chapoton functions. For each $(\textbf{d},\textbf{v})\in 
\mathbb{N}^n\times \mathbb{N}^n$ we consider the function
\[
C_{\textbf{d, v}}\colon \decrep_{\textbf{d, v}}(\Lambda)\rightarrow 
\mathbb{Z}[x_1^{\pm}, \ldots, x_n^{\pm}],
\]
defined by $\mathcal{M}\mapsto C_{\Lambda}(\mathcal{M})$. Since this function is 
constructible, then its image is finite. It turns out that for any irreducible component 
$Z\in\DIrr_{\textbf{d, v}}(\Lambda)$ there exist a dense open subset $U\subseteq Z$ 
where $C_{\textbf{d, v}}$ is constant in $U$. We define 
$C_{\Lambda}(Z)=C_{\Lambda}(\mathcal{M})$, for any $\mathcal{M}\in U$.

The set of irreducible components is denoted by $\firr(\Lambda)$. We define the 
\emph{graph} $\Gamma(\firr(\Lambda))$ of strongly reduced irreducible components as 
follows: it  has a vertex for any 
indecomposable  strongly reduced component and there is an edge between vertices $Y$ and 
$Z$ if $E_{\Lambda}(Y,Z)=0=E_{\Lambda}(Z,Y)$ . Note that $Y$ can be equal $Z$. The graph 
of irreducible components $\Gamma(\Irr(\Lambda))$ was defined in \cite[Section 
12.3]{CBS-02}

Let $\Gamma$ be a graph. We consider  graphs with single edges and loops. We denote by 
$\Gamma_0$ the set of vertices of $\Gamma$. By $\Gamma_{\mathcal{U}}$ we denote the full 
subgraph of $\Gamma$, whose set of vertices is $\mathcal{U}$. 

We call $\Gamma_{\mathcal{U}}$ \emph{complete} if for any vertices $i\neq j$ in 
$\mathcal{U}$ there is an edge between them.
A complete $\Gamma_{\mathcal{U}}$ subgraph is called \emph{maximal} if for any other 
complete subgraph $\Gamma_{\mathcal{U}'}$ with $\mathcal{U}\subseteq 
\mathcal{U}'$ we have $\mathcal{U}=\mathcal{U}'$.

We define a \emph{component cluster} of $\Lambda$ as the set of vertices of a maximal 
complete subgraph of $\Gamma(\firr(\Lambda))$. A component cluster $\mathcal{U}$ is 
$E$-rigid whenever $E_{\Lambda}(Z)=0$ for all $Z\in \mathcal{U}$.

The following notions were introduced in \cite[Section 6.5]{CLS-15}.

\begin{definition}
Let $\mathcal{U}=\{Z_1, Z_2,\ldots, Z_t\}$ be a  component cluster. We define  the 
$CC$-\emph{cluster} of $\Lambda$ associated to $\mathcal{U}$ by 
$\mathcal{C}_{\mathcal{U}}=\{\mathcal{C}_{\Lambda}(Z_1),\mathcal{C}_{\Lambda}(Z_2), 
\ldots, \mathcal{C}_{\Lambda}(Z_t)\}$. 
\end{definition}

The notation $CC$ comes from sets of Caldero-Chapoton functions. Cerulli Irelli, 
Labarini-Fragoso and Schr\"oer introduced the notion of laurent phenomenon for 
Caldero-Chapoton algebras, see \cite[Section 6.5]{CLS-15}.

\begin{definition}
The Caldero-Chapoton algebra $\mathcal{A}_{\Lambda}$ has the \emph{Laurent 
phenomenon property} provided for any  $E$-rigid component cluster 
$\{Z_1, \ldots, Z_t\}$ of $\Lambda$,  we have
\[
\mathcal{A}_{\Lambda}\subseteq \mathbb{C}[\mathcal{C}_{\Lambda}(Z_1)^{\pm}, \ldots 
\mathcal{C}_{\Lambda}(Z_t)^{\pm}].
\]
\end{definition}

\section{String algebras}\label{section3}
In this section we recall  some definitions and results about \textit{string algebras} 
exposed in \cite{BR-87}. Let $Q$ be a quiver.

Let $P$ be a subset of paths in $\mathbb{C}\langle Q\rangle$ and denote by $\langle P
\rangle $ the ideal generated by $P$. The algebra $\Lambda=\mathbb{C}\langle Q\rangle/ 
\langle P\rangle $ is called a \emph{string algebra} if the following conditions hold:

\begin{enumerate}
\item Any vertex $i\in Q_0$ is the tail or head point of at most two arrows of $Q$, 
that is, $|\{a\in Q\colon t(a)=i\}|\leq 2$ and $|\{a\in Q\colon h(a)=i\}|\leq 2$.
\item For any arrow $a\in Q_1$ we have $|\{b\in Q_1\colon t(a)=h(b) \mbox{ and } ab
\notin P\}|\leq 1$ and $|\{c\in Q_1\colon t(c)=h(a) \mbox{ and } ca\notin P\}|\leq 1$.
\item The ideal $\langle P\rangle$ is admissible on $\mathbb{C}\langle Q\rangle$. 
\end{enumerate}

To describe the finite-dimensional indecomposable $\Lambda$-modules we need the concept 
of \emph{string}.  We introduce an \textit{alphabet} consisting of  \textit{direct 
letters} given by each arrow $a\in Q_1$ and \emph{inverse letters}  given by $a^{-1}$ 
for each arrow $a\in Q$. The head and tail functions extend to this alphabet in the 
obvious way, that is, $h(a^{-1})=t(a)$ and $t(a^{-1})=h(a)$ for every arrow $a\in Q_1$. 
For a letter $l$ in this alphabet we denote its inverse letter with $l^{-1}$ and we 
write $l$ instead  $(l^{-1})^{-1}$. A \textit{word} in this alphabet of \emph{length} 
$r\geq 1$ is a sequence of letters $l_r\cdots l_1$ such that $t(l_{i+1})=h(l_{i})$ for 
$i=1,\ldots , r-1$.
For a word $W=l_r\cdots l_1$ we denote its \emph{inverse word} by $W^{-1}=l_1^{-1}
\cdots l_r^{-1}$. It is clear we can extend the head and tail functions to words.
 A \textit{string} of length $r\geq 1$ is a word $W=l_r\cdots l_1 $ such that  $W$ and  
$W^{-1}$ do not contain sub-words of the form $ll^{-1}$ for a letter $l$ and no 
sub-words of $W$ belongs to $P$. 

We introduce strings of length $0$ in the following way. For each vertex $i\in Q_0$  we 
have two strings of length $0$  denoted by $1_{(i,u)}$ with $u\in \{1,-1\}$. In this 
case $h(1_{(i,u)})=i=t(1_{(i,u)})$. By definition $1_{(i,u)}^{-1}=1_{(i,-u)}$.

We recall the definition of two functions to deal with strings. In \cite{BR-87} it is 
shown we can choose two functions $\sigma,\epsilon\colon Q_1\rightarrow \{1,-1\}$ such 
that the following conditions are satisfied

\begin{enumerate}
\item If $a_1\neq a_2$ are arrows with $t(a_1)=t(a_2)$, then $\sigma(a_1)=-\sigma(a_2)
$.\\
\item If $b_1\neq b_2$ are arrows with $h(b_1)=h(b_2)$, then  $\epsilon(b_1)=-
\epsilon(b_2)$.\\
\item If $a,b \in Q$ are arrows with $t(b)=h(a)$ and $ba\notin P$, then $\sigma(b)=-
\epsilon(a)$.
\end{enumerate}

For an arrow $a\in Q_1$ we have $\sigma(a^{-1})=\epsilon(a)$ and $\epsilon(a^{-1})=
\sigma(a)$. For a string $W=l_r\cdots l_1$ we define $\sigma(W)=\sigma(l_1)$ and $
\epsilon(W)=\epsilon(l_r)$. Besides we have  $\sigma (1_{(i,u)})=-u$ and $
\epsilon(1_{(i,u)})=u$. Note that if $W_1$ and $W_2$ are strings such that $W_2W_1$ is 
a string, then $\sigma(W_2)=-\epsilon(W_1)$. For $(i,u)\in Q_0\times \{1,-1\}$ let $
\mathcal{W}_{(i,u)}$ be the set of all strings $W$ with $h(W)=i$ and $\epsilon(W)=u$. 
Let $\mathcal{W}$ be the set of all strings and define on $\mathcal{W}$ an equivalence 
relation given by $W_1\sim W_2$ if and only if $W_2\in\{W_1, W_1^{-1}\}$. Let $
\underline{\mathcal{W}}$ be a complete set of representatives of the corresponding 
equivalence classes. 

\begin{remark}
In this article we are not going to use this functions $\epsilon$ and $\sigma$, but it 
is useful to remember that for string algebras the strings can be thought of as 
sequence of signs.
\end{remark}

In \cite{BR-87}, it was also defined the set $\mathcal{B}$ of \textit{bands}. A string $W\in 
\mathcal{W}$ belongs to $\mathcal{B}$ if length of $W$ is positive, $W^n\in \mathcal{W}
$ for all $n\in \mathbb{N}$ and $W$ is not the power of some string of smaller length. 

\subsection{Indecomposable string modules}\label{DefinitionNW}
For a string $W$, in \cite{BR-87}, it was defined a $\Lambda$-module $N(W)$, for 
convenience we  repeat this definition.
For the string $1_{(i,u)}$ we define $N(1_{(i,u)})$ as the simple representation $S_i$ 
at the  vertex $i\in Q_0$. If $W=l_r\cdots l_1$, then $N(W)$ is a representation of $
\mathbb{C}$-dimension $r+1$. For describe the structure of $\Lambda$-module let 
$p_0=t(l_1)$ and $p_k=h(l_k)$ for $k=1, \ldots, r$ vertices of $Q$. By definition  $
\dimension(N(W)_i)$ is $|\{k\in[1,r+1]\colon p_k=i\}|$. If $\{z_0, \cdots, z_r\}$ is a 
basis of $N(W)$ with $z_k\in N(W)_{p_k}$ for $k=0,\ldots, r$, then the action of the 
arrows is given by the following way
\[\xymatrix{  
z_0\ar@{|->}[r]^{l_1} & z_1\ar@{|->}[r]^<(0.3){l_2}&  \cdots\ar@{|->}[r]^<(0.3)
{l_{n-1}}  & z_{n-1}\ar@{|->}[r]^<(0.4){l_{n}}&z_n.  
}\]
If $l_k$ is a direct letter, then $N(W)_{l_k}(z_{k-1})  = z_{k}$; if $l_k$ is a inverse 
letter, then $N(W)_{l_k}(z_{k-1})  = z_{k}$ with $k=1,\ldots, n$; if $a\in Q_1$ and 
$N(W)_{a}(z_{k})$ is not defined yet, then $N(W)_{a}(z_{k})=0$.

In \cite{BR-87}, it was observed that $N(W)$ is isomorphic to $N(W^{-1})$. The modules 
$N(W)$ are called \emph{string modules}. The next result is a special case of the 
more general  result of Butler and Ringel proved in \cite[Section 3]{BR-87}.

\begin{theorem}[(Butler-Ringel)]\label{BRTheorem}
Let $\Lambda$ be a string algebra. If $\mathcal{B}=\emptyset$, then   the 
$\Lambda$-modules $N(W)$ with $W\in \underline{\mathcal{W}}$ form a complete list of 
indecomposable,  pairwise non-isomorphic $\Lambda$-modules.
\end{theorem}

\subsection{Sub-strings}\label{SubSectionCam}
For a string of positive length $W=l_r\cdots l_1$ we define its \textit{support} as $
\Supp(W)=\{t(l_1)\}\cup\{h(l_k)\colon k=1,\cdots, r\}$. If $W=1_{(i,u)}$, then $
\Supp(W)=\{i\}$.
Given a string of positive length $W=l_r\cdots l_1$, we say that a string  $V$ is a 
\emph{sub-string} of $W$ if $V=l_{t+j}\cdots l_t$ is a subword of $W$ and there are no 
arrows $a,b\in Q_1$ such that $aV$ and $Vb^{-1}$ are subwords of $W$. 
For technical reasons we introduce the \emph{zero string} $0$ which is sub-string of 
any string.  Now, given a string $W$ we denote by $\Caminos(W)$ the set  of all 
sub-strings of $W$.
\begin{remark}
We use $\Caminos$ from the  word  in spanish \emph{caminata} for 
walk. On one hand  the elements of $\Caminos(W)$ can be thought as walks in 
the quiver. On the other  hand, there are a lot of W's in our notation.
\end{remark}
\section{Generalized cluster algebras}\label{Section4}
For convenience we recall the definition of \emph{generalized cluster algebras} 
introduced by Chekhov and Shapiro in \cite{CS-14}.

We say a matrix $B\in \Matriz_{n\times n}(\mathbb{Z})$ is \emph{skew-symmetrizable} if 
there exist positive integers $d_1,\ldots, d_n$ such that  $DB$ is skew-symmetric with 
$D=\diagonal(d_1,\ldots, d_n)$ a diagonal matrix. In this case we call $D$ a 
\emph{skew-symmetrizer} of $B$.

Given a skew-simmetrizable matrix $B$ and an integer $k\in\{1,\ldots, n\}$, the 
\emph{mutation of $B$ with respect to} $k$ is the matrix $\mu_{k}(B)$ with entries 
$b'_{ij}$ defined as follows

\begin{equation*}
b'_{ij}=
\left\{
\begin{array}{ll}
\quad -b_{ij} \mbox{ \ \ \ \ \ \ \ \ \ \ \ \ \ \ \ \ \  if } k=i \mbox{ or } k=j,\\
b_{ij} +\frac{|b_{ik}|b_{kj}+ b_{ik}|b_{k,j}|}{2} \mbox{ if } i\neq k\neq j.
\end{array}
\right.
\end{equation*}

For us it is enough  work without coefficients. For a study of generalized  cluster 
algebras with principal  coefficients, in parallel with the one for cluster algebras, 
the reader can see \cite{N-15}.

Let $\mathcal{F}$ be the field of the rational functions in $n$ algebraic independent 
variables with coefficients in $\mathbb{Q}$.

\begin{definition}
A \emph{seed} in $\mathcal{F}$ is a pair $(B,\textbf{x})$ where $B$ is a 
skew-symmetrizable matrix and $\textbf{x}=(x_1,\ldots, x_n)$ is  an $n$-tuple of algebraic 
independent elements of $\mathcal{F}$ which generate it.
\end{definition}

Now we assume that $B$ is skew-symmetric with skew-symmetrizer  $D$ and $b_{ik}/d_k$ is 
an integer for all $i\in\{1,\ldots, n\}$.

For $k\in \{1,\ldots, n\}$ we define the polynomials
\[ v_k^+=\prod_{b_{l,k}>0}{x_l^{b_{l,k}/d_k}} , \  \  \ v_k^-= \prod_{b_{l,k}<0}{x_l^{-
b_{l,k}/d_k}} \ \ \ \mbox{ and } \ \ \ \theta_k(u,v)=\sum_{l=0}^{d_k}{u^lv^{d_k - l}}.
\]

\begin{definition}[(Generalized cluster mutation)]
For a seed $(B,\textbf{x})$ and $k\in \{1,\ldots, n\}$, the \emph{mutation of $(B,
\textbf{x})$ with respect to} $k$ is the pair $\mu_k(B,\textbf{x})=(\mu_k(B), 
\mu_k(\textbf{x})$ where $\mu_k(\textbf{x})=(x_1',\ldots, x_n')$ is the $n$-tuple of 
elements of $\mathcal{F}$ given by

\begin{equation*}
x_i'=
\left\{
\begin{array}{ll}
x_i, \mbox{ \ \ \ \ \ \ \ \ \ \ \ if } k\neq i,\\
\frac{\theta_k(v_k^+,v_k^-)}{x_i}, \ \ \mbox{  if } k=i.
\end{array}
\right.
\end{equation*}
\end{definition}
In this article the polynomials $\theta_k(v_k^+,v_k^-)$ are often called 
\emph{polynomials of Chekhov-Shapiro}.
For a seed $(B,\textbf{x})$, let
\[\mathcal{X}=\{x\in \mu_{k_r}\cdots\mu_{k_1}(\textbf{x})\colon k_r\in\{1,\ldots, n\} \mbox{ 
and } r\geq 0 \}.\]

\begin{definition}
The \emph{generalized cluster algebra associated to} $(B, \textbf{x})$, denoted by $
\mathcal{A}(B)=\mathcal{A}(B,\textbf{x})$, is the subring of $\mathcal{F}$ generated by 
$\mathcal{X}$.
\end{definition} 

Note that exchange polynomials do not have to be binomials.  In  \cite[Theorem 2.5]
{CS-14}, it was shown a desired property.

\begin{theorem}[(The Laurent phenomenon)]
 Any generalized cluster variable can be expressed as a Laurent polynomial in the 
initial variables $x_i$.
\end{theorem}
 
\begin{remark}
The reader should be cautious by comparing the definitions of \cite{CS-14} with this 
ones  because there they take skew-symmetrizer by the right and here we did by the 
left.
\end{remark}

\begin{remark}\label{ExRel1}
In \cite[Lemma 3.1]{CS-14} the authors obtain exchange relations of the form 
$a^2+2\cos(\frac{\pi}{p})ab+b^2$ in the case of orbifold points of order $p$, with $p$ 
greater than one. We are going to obtain exchange relations of the form 
$a^2+ab+b^2$. Of course it is more natural to consider orbifold points of order three 
than order two as we made in the first version of this article.  
\end{remark}
\section{Polygons with one orbifold point}\label{Section5}
We will work with \emph{polygons with one orbifold point of order three} but for 
convenience  we recall some definitions of surfaces with orbifold points.
 For  more details about surfaces with orbifold points of order two or three and 
relations with generalized cluster algebras the reader can see \cite{CS-14} and 
references therein, for example \cite{Ch-09_1,Ch-09_2}. For an interesting and 
beautiful application of surfaces with orbifold points and group actions  in some 
cluster structures the reader is kindly asked to look at \cite{PS-17}.  

\subsection{Basic definitions}

Here we review some combinatorial aspects about orbifolds, however orbifolds, in 
general, are spaces generalizing manifolds. Orbifolds are locally modeled 
by an Euclidean space modulo a finite group. The \emph{order} of an orbifold point 
is the order of the non-trivial stabilizer  at that point. In our setting is not 
necessary to recall more geometric aspects about orbifolds. In the next section, we 
will concentrate in our specific surface. In particular, we shall see that the orbifold 
point of order three we are talking about is the fixed point under a suitable action of 
$\mathbb{Z}_3$ on a polygon.

Let $\Sigma$ be a compact connected oriented 2-dimensional real surface with possible 
empty boundary.  The pair $(\Sigma, \mathbb{M})$ where $\mathbb{M}$ is a finite subset 
of $\Sigma$ with at least one point from each connected component
of the boundary of $\Sigma$ is called a \emph{bordered surface} or just a surface. The 
points of $\mathbb{M}$ are called \emph{marked points} and the points of $\mathbb{M}$ 
that lie in the interior of $\mathbb{M}$ are called \emph{punctures}.
A triple $(\Sigma, \mathbb{M}, \mathbb{O})$ where $(\Sigma, \mathbb{M})$ is a bordered 
surface and $\mathbb{O}$ is a finite subset of $\Sigma \setminus (\mathbb{M}\cup 
\partial\Sigma)$ is called a \emph{marked surface with orbifold points}. The points of 
$\mathbb{O}$ are called \emph{orbifold points} and they will be denoted by a cross $
\times$ in the surface.

\subsection{Triangulations and flips}
 Let $(\Sigma, \mathbb{M}, \mathbb{O})$ be a marked surface with orbifold points of 
order three, without punctures, with boundary and we assume $\mathbb{O}$ is not empty.  
 An arc $i$ on $(\Sigma, \mathbb{M}, \mathbb{O})$ is a curve $i\colon [0,1]\rightarrow 
 \Sigma$ satisfying the following conditions
\begin{itemize}
\item the  endpoints of $i$ are both contained in $\mathbb{M}$. 
\item $i$ does not intersect itself, except that its endpoints may coincide.
\item $i$ does no intersect to $\mathbb{O}$ and $i$ does not intersect $\mathbb{M}$ 
except in its endpoints.
\item if $i$ cuts out an  monogon, then such monogon contains just one orbifold point. 
In this case $i$ is called a \emph{pending arc} or \emph{the loop of a orbifold point}.
\end{itemize}

Two arcs $i$ and $j$ are \emph{isotopic relative }to  $\mathbb{M}\cup\mathbb{O}$ if 
there exist a continuous function $H:[0,1]\times\Sigma \rightarrow \Sigma$ such that 
\begin{itemize}
\item $H(0, x)=x$, for all $x\in \Sigma$;
\item $H(1,i)=j$;
\item $H(t, p)=p$ for all $p\in \mathbb{M}\cup \mathbb{O}$;
\item  For every $t\in [0,1]$ the function $H_t\colon \Sigma \rightarrow \Sigma$ with 
$x\mapsto H(t,x)$ is a homeomorphism.
\end{itemize}
 We will consider arcs up to isotopy relative to $\mathbb{M}\cup \mathbb{O}$, 
 parametrization and orientation.

Given an arc $i$, we denote by $\tilde{i}$ its isotopy class. Let $i$ and $j$ be two 
arcs. We say   $i$ and $j$ are \emph{compatible} if either $\tilde{i}=\tilde{j}$ or $
\tilde{i}\neq\tilde{j}$ and there are arcs   $i_1\in\tilde{i}$ and $j_1\in  \tilde{j}$, 
such that  $i_1$ a $j_1$ do not intersect in $\Sigma \setminus \mathbb{M}$.

\begin{definition}
A \emph{triangulation} of $(\Sigma, \mathbb{M}, \mathbb{O})$  is a maximal collection 
of pairwise compatible arcs.
\end{definition}

Given a triangulation $\tau$ of the surface we define a \emph{triangle} of $\tau$ as 
the closure of a connected component of $(\Sigma\setminus \tau) \cup \{ \mbox{ the 
pending arcs of } \tau\}$. An \emph{orbifold triangle} is a triangle containing an 
orbifold point. A triangle without an orbifold point in its interior is called an 
\emph{ordinary triangle}. If a triangle intersects to the boundary 
of the surface at most in a  three points it is called an \emph{internal triangle}.

Let $\tau$ be a triangulation of $(\Sigma, \mathbb{M}, \mathbb{O})$. If $i$ is an arc 
of $\tau$, the  \emph{flip} of  $i$ with respect to $\tau$ is the unique arc $i'$ such 
that $\sigma=(\tau\setminus \{i\})\cup\{i'\}$ is a triangulation of $(\Sigma, 
\mathbb{M}, \mathbb{O})$. In this case we denote $i'=\flip_{\tau}(i)$ and we say that $
\sigma$ is obtained from $\tau$ by a flip of $i\in \tau$.
In our case, flips act transitively on triangulations of $(\Sigma, \mathbb{M}, 
\mathbb{O})$, see \cite[Theorem 4.2]{FeST-14}.

\subsection{Our surfaces}\label{SectPoly}
Let  $\Sigma_n$ be the surface, with $n\geq 2$,  given by a disk with boundary, $n+1$ 
marked points in its boundary,   without punctures and one orbifold point of order 
three. In this work we often refer to $\Sigma_n$ as  the $(n+1)$-\emph{polygon with one 
orbifold point}, the marked points are called \emph{vertices} and they are denoted by $
\{v_0,v_1,\ldots, v_n\}$. We orient the vertices in counterclockwise order. In  
pictures the orbifold point  is drawn with the symbol $\times$.

Let $\tau$ be  a triangulation of $\Sigma_n$. We have that $|\tau|=n$, see \cite[Lemma 
4.1]{PS-17}. In this case we have two types of admissible  triangles for $\tau$, see 
Figure \ref{Fig1}

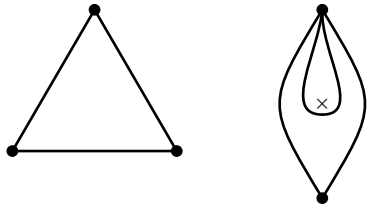
\begin{figure}[h]
\centering
\begin{tikzpicture}[scale=0.5,cap=round,>=latex]
\begin{scope}
  \newdimen\R
  \R=2.5cm
  \coordinate (center) at (0,0);
	\filldraw [black] (90:\R)  circle (4pt)
			(210:\R)  circle (4pt)
			(330:\R) circle (4pt);
 \draw (0:\R);
			\draw[line width=1pt] (90:\R)--(210:\R);
			\draw[line width=1pt] (210:\R)--(330:\R);
			\draw[line width=1pt] (330:\R)--(90:\R);
\end{scope}
\begin{scope}[xshift=6cm]
\newdimen\R
  \R=2.5cm
  \coordinate (center) at (0,0);
	\filldraw [black] (90:\R)  circle (4pt)
			(270:\R)  circle (4pt);
 \draw (0:\R);
			\draw[line width=1pt] (90:\R) .. controls (-1.5, 0)  ..(270:\R);
			\draw[line width=1pt] (270:\R) .. controls (1.5, 0)  .. (90:\R);
	\node at (0,0) {$\times$};
	\draw[line width=1pt, rotate=45] (45:\R) to[out=220,in=135] (-0.2,-0.2);
			\draw[line width=1pt, rotate=45] (-0.2,-0.2) to[out=315,in=222] (45:2.5);	
\end{scope}
	\end{tikzpicture}
\caption{ An ordinary triangle (left) and an orbifold  triangle, i.e. a triangle 
containing the orbifold point (right) .}	
\label{Fig1}
\end{figure}

 We associate a quiver $Q(\tau)$ to a triangulation $\tau$ of the \emph{orbifold} $
\Sigma_n$ in the following way:  the set of vertices is given by the arcs of $\tau$ and 
the set of arrows is described as follows. For each  triangle $\Delta$ of $\tau$ and 
arcs $i$ and $j$ in $\Delta$  we draw an arrow from $j$ to $i$  if $i$ succeeds  $j$  
in the clockwise orientation,  with the understanding that no arrow incident to a 
boundary segment is drawn. Finally  we draw a loop at the pending arc of $\tau$. 
 
 \begin{remark}\label{ComenPoligo}
In the classical context of marked Riemann surfaces without orbifold points this loop  
is not drawn. For instance the quiver of a triangulation $T$ of a polygon $P$ without 
punctures and without orbifold points will be denoted by $Q(T)$ and it is constructed 
as above but, as we said, it does not have loops. 
\end{remark}

Denote by $\mathcal{H}(\tau)$ the collection of all internal triangles $\Delta$ of a 
given triangulation $\tau$. Any element $\Delta$ of $\mathcal{H}(\tau)$ defines a 
3-cycle $c_{\Delta}b_{\Delta}a_{\Delta}$ on $Q(\tau)$ up to \emph{cyclical 
equivalence}. If we denote 
by $\varepsilon$ the loop of $\tau$, then the \emph{potential} associated to $\tau$
 is $S(\tau)=\sum_{\Delta\in \mathcal{H}(\tau)}{c_{\Delta}b_{\Delta}a_{\Delta}}+ 
 \varepsilon^3$. 
\begin{definition}
For any triangulation $\tau$ of $\Sigma_n$ we define the \emph{basic algebra} $
\Lambda(\tau)$ associated to $\tau$ as the Jacobian algebra $\Lambda(\tau)=\mathcal{P}
(Q(\tau), S(\tau))$.  
\end{definition}

\begin{example}\label{FirExamLam}
Consider the triangulation $\tau$  of Figure \ref{ExFinSect}. We see that the algebra $
\Lambda(\tau)$  is $\mathbb{C}\langle Q(\tau)\rangle/I$, where $I$ is the ideal 
generated by $ba$, $cb$, $ac$ and $\varepsilon^2$ (compare to  \cite[Example 2.3]
{PS-17}).
\begin{figure}
\includegraphics[width=3cm, height=3cm]{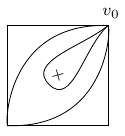}\hspace{3cm} 
\includegraphics[width=3cm, height=3cm]{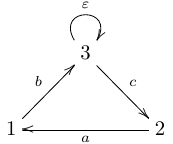}
\caption{ One triangulation $\tau$ of $\Sigma_3$ and its quiver $Q(\tau)$.}
\label{ExFinSect}
\end{figure}
\end{example}

\begin{definition}
A \emph{weighted quiver} is a pair $(Q, \textbf{d})$ where $Q$ is a quiver without 
loops and $\textbf{d}=(d_i)_{i\in Q_0}$ is an $n$-tuple of positive integers.
\end{definition}

Let $Q(\tau)^*$ be the quiver obtained from $Q(\tau)$ by deleting the loop.   Now we 
denote by $(Q(\tau)^*,\textbf{d}_{\tau})$  the \emph{weighted quiver associated to} $
\tau$ where $(\textbf{d}_{\tau})_j=2$ if $j$ is the pending arc and $(\textbf{d}
_{\tau})_j=1$ in other wise.

Fix an $n$-tuple $\textbf{d}=(d_1\ldots, d_n)$, in \cite[Lemma 2.3]{LZ-13} it was 
proved that there is a bijection between the set of 2-aclycic weighted quivers $(Q,
\textbf{d})$ and the collection of skew-symmetrizable matrices $B$ with skew-
symmetrizer given by $D=\diagonal(d_1,\ldots, d_n)$. Indeed, given a quiver $Q$, if 
$c_{ij}$ is  as in Section \ref{section2}, then $b_{ij}=d_jc_{ij}/\mbox{gcd}(d_i,d_j)$ 
define a matrix $B_{Q}$ skew-symmetrized by $D$.

Following \cite[Lemma 2.3]{LZ-13}, we denoted by $B(\tau)$ the skew-symmetrizable 
matrix associated to $(Q(\tau),\textbf{d}_{\tau})$ and we call it the \emph{adjacency 
matrix } associated to $\tau$.
\section{$\Lambda(\tau)$ as an orbit algebra}\label{Sect14}
In this section we shall note that  $\Lambda(\tau)$ can be seen as an \emph{orbit 
Jacobian algebra}. With this observation we are going to obtain some results about 
Galois coverings.  For details and missing definitions about orbit Jacobian algebras 
the reader can see \cite{PS-17}. At the end of the section we will define the arcs 
representations of $\Lambda(\tau)$.

Let $\widetilde{\Sigma}_n$ be the regular $(3n+3)$-gon with $u_1, u_2, \ldots, u_{3n+3}
$ vertices in counterclockwise orientation and let $\theta$ be the rotation by 
$120^{\circ}$ on $\widetilde{\Sigma}_n$ which sends a vertex $v_i$ to $v_{i+(n+1)}$ 
modulo $(3n+3)$.  In the terminology of \cite{PS-17}, $\Sigma_n$ is the 
$\mathbb{Z}_3$-\emph{orbit space} of $\widetilde{\Sigma}_n$. We consider $\theta$ as a generator of  
$G=\mathbb{Z}_3$. We can see that $G$ acts freely on $\{u_1, u_2, \ldots, u_{3n+3}\}$, 
that is, if $g\in G\setminus \{e\}$, then $g\cdot u_i\neq u_i$ for $i\in [1,3n+3]$.

We say that an arc $\widetilde{j}$ of $\widetilde{\Sigma}_n$ is $G$-\emph{admissible} 
or just \emph{admissible} if $\widetilde{j}$ belongs to some $G$-invariant 
triangulation $T$ of $\widetilde{\Sigma}_n$.

Let $T$ be a triangulation of $\widetilde{\Sigma}_n$ and suppose that $T$ is $G$-
invariant. Consider $Q(T)$ the quiver associated to $T$, see Remark \ref{ComenPoligo}. 
We can define a potential  for $Q(T)$ as  $S(T)=\sum_{\Delta\in \mathcal{H}(T)}
{\gamma_{\Delta}\beta_{\Delta}\alpha_{\Delta}}$. Note that  $G$ acts freely on $Q(T)_0$ 
and for any $\alpha_{\Delta}\beta_{\Delta}\gamma_{\Delta}$ we have that $g\cdot( 
\alpha_{\Delta}\beta_{\Delta}\gamma_{\Delta})$ is again a summand of $S(T)$ for all $g
\in G$. We can define, \cite[Section 2.1]{PS-17}, the \emph{orbit quiver} $Q(T)_G$ of 
$Q(T)$   in the obvious way. We define the potential $S(T)_G$ for $Q(T)_G$ as the image 
of $S(T)$ under the canonical  morphism $\pi:\mathbb{C}\langle Q(T)\rangle\rightarrow 
\mathbb{C}\langle Q(T)_G\rangle$ induces by $\pi(i)=G\cdot i$ for $i\in Q(T)_0$ and $
\pi(a)=G\cdot a$ for $a\in Q(T)_1$, note that $\pi$ is a Galois $G$-covering. We define 
the \emph{orbit Jacobian algebra} of the orbit quiver with potential as $\mathcal{P}
(Q(T),S(T))_G= \mathcal{P}(Q(T)_G,S(T)_G)$. We make the following convention $
\Lambda(T)=\mathcal{P}(Q(T),S(T))$ and $\Lambda(T)_G=\mathcal{P}(Q(T),S(T))_G$, see 
Example \ref{ExplaPS}.
The following result shows that we get a Galois covering, see \cite[Proposition 3.1]
{PS-17}. 
\begin{lemma}[(Paquette-Schiffler)]\label{LambdaGcovering}
The Galois $G$-covering $\pi:\mathbb{C}\langle Q(T)\rangle\rightarrow \mathbb{C}\langle 
Q(T)_G\rangle$ induces a Galois $G$-covering $\pi:\Lambda(T)\rightarrow  \Lambda(T)_G$.
\end{lemma}

\begin{remark}
Let $\tau$ be a triangulation of $\Sigma_n$ and let $T$ be the triangulation of $
\widetilde{\Sigma}_n$ such that $G\cdot T=\tau$. In $T$ there exist an unique triangle 
$\Delta_T$ such that it is $G$-invariant and  the other triangles in $\mathcal{H}(T)$ 
have a trivial stabilizer. The triangle $\Delta_T$ corresponds to the pending arc of $
\tau$ and the $G$-orbit of any triangle $\Delta$ different to $\Delta_T$  corresponds 
with a triangle of $\mathcal{H}(\tau)$.   We conclude that  $Q(\tau)=Q(T)_{G}$ and with 
the above observation we get that $\Lambda(\tau)=\Lambda(T)_G$.
\end{remark}

\begin{example}\label{ExplaPS}
Let $\tau$ be the triangulation of $\Sigma_3$ depicted on the right of Figure 
\ref{FigExamPS}. Let $T$ be the corresponding triangulation on $\widetilde{\Sigma}_3$ 
depicted on the left of Figure \ref{FigExamPS}. The quiver $Q(T)$ is drawn below

\begin{equation*} 
\xymatrix{& & & i_3\ar@{->}[d]_{\beta_3}& & &\\
          & & & j_3\ar@{->}[d]_{\alpha_3}& & &\\  
          & & & k_3 \ar@{->}[dr]^{\varepsilon_3} & & &  \\
          i_1\ar@{->}[r]_{\beta_1}&j_1\ar@{->}[r]_{\alpha_1} &  k_1 \ar@{->}
          [ru]^{\varepsilon_1}  & \ \  &  k_2 \ar@{->}[ll]^{\varepsilon_2}& j_2\ar@{->}
          [l]^{\alpha_2} &  i_2\ar@{->}[l]^{\beta_2}}
\end{equation*}

Consider $S(T)=\varepsilon_1\varepsilon_2\varepsilon_3$ the \emph{potential} associated 
to $T$. Let $G=<\theta>$ be the cyclic group of order 3 with generator $\theta$.  Then 
$G$ acts freely on $(Q,S)$ by increasing by one, module 3, the indices of the symbols. 
Passing to the orbit space of this action we get
\[
\xymatrix{  
Q(T)_G: i\ar@{->}[r]^{\beta} & j\ar@{->}[r]^<(0.3){\alpha}&  k\ar@(ur,dr)^{\varepsilon}   
} \  \ \mbox{ and the potential } S(T)_{G}=\varepsilon^3,
\]
where $i=G\cdot i_1$, $j=G\cdot j_1$, $k=G\cdot k_1$, $\alpha=G\cdot \alpha_1$, $
\beta=G\cdot \beta_1$ and $\varepsilon=G\cdot \varepsilon_1$. The orbit Jacobian 
algebra $\mathcal{P}(Q(T)_G, S(T)_G)$  is not but $\Lambda(\tau)=\mathbb{C}\langle 
Q(\tau)\rangle /\langle \varepsilon^2 \rangle$.  
\end{example}

\begin{proposition}\label{FinitedimPro}
Let $\tau$ be a  triangulation of $\Sigma_n$. Then $\Lambda(\tau)$ is finite 
dimensional.
\end{proposition}
\begin{proof}
From Lemma \ref{LambdaGcovering} we have that $\pi:\Lambda(T)\rightarrow  \Lambda(T)_G$ 
is a Galois $G$-covering. In particular we have isomorphism $\pi_{i,j}: \bigoplus_{g\in 
G}\Lambda(T)(e_i,g\cdot e_j)\rightarrow\Lambda(\tau)(\pi(e_i),\pi(e_j))$ for any 
idempotent $e_i$ and $e_j$ of $\Lambda(T)$. We know $\Lambda(T)$ is finite-dimensional 
(the reader can see  the finite dimension of the Jacobian algebra associated to a 
triangulation in a more general context of surfaces with non-empty boundary in 
\cite[Theorem 36]{L-09}), so $\Lambda(\tau)$ is finite dimensional. The proof of the 
lemma is completed.
\end{proof}

A string algebra $B=\mathbb{C}\langle Q\rangle/ \langle P\rangle$ is a \emph{gentle 
algebra} if the following conditions are satisfied
\begin{itemize}
\item[(Gt1)] $P$ is generated by paths of length 2.
\item[(Gt2)] For any arrow $a\in Q_1$ we have $|\{b\in Q_1\colon t(a)=h(b) \mbox{ and } 
ab \in P\}|\leq 1$ and $|\{c\in Q_1\colon t(c)=h(a) \mbox{ and } ca\in P\}|\leq 1$.
\end{itemize}

\begin{proposition}\label{PropGentAlg}
For any triangulation $\tau$ of $\Sigma_n$ we have that $\Lambda(\tau)$ is a gentle 
algebra.
\end{proposition}
\begin{proof}
Let $T$ be the triangulation of $\widetilde{\Sigma}_n$ such that $G\cdot T=\tau$. The 
proof is an \emph{adaptation} of proof \cite[Lemma 2.5]{ABCP-09}, from that lemma we 
have that $\Lambda(T)$ is gentle.  By definition  $\Lambda(\tau)=\mathbb{C}\langle Q_G
\rangle/J(Q_G, S_G)$ and it is clear that $J(Q_G, S_G)$ is generated by paths of length 
two. Since we have Proposition \ref{FinitedimPro},  only remains to prove (Gt2), (S1) 
and(S2). 

(S1). First, let $j$ be the pending arc of $\tau$. We consider $\tilde{j}$ an element 
in $\pi^{-1}(j)$. We have that  $\tilde{j}$ is contained in two triangles of $T$. One 
of those triangles has  the other preimages of $j$ as sides,  say $\Delta(j)$, in other 
words, $\Delta(j)$ is invariant under $G$. By the definition of $S(T)_G$ we can 
conclude that there is a loop based at $j$ and there is at most one arrow starting at 
$j$ and one arrow ending at $j$. Now, one connected component of $\tilde{\Sigma}_n
\setminus\{\tilde{j}\}$, precisely which no contain the other preimages of $j$, it is a 
fundamental region for the action of $G$ on $T$. If $k$ is not a pending arc of $\tau$, 
we can consider a preimage of $k$  in the fundamental region above, recall $\Lambda(T)$ 
is gentle, this implies (S1) for $k$. For this reason we just have to prove (S2) and 
(Gt2) in the orbifold triangle. 

(S2) and (Gt2). This two properties follow from the fact that there is  a loop based at 
the pending arc $j$ and at most one arrow with starting at $j$ and at most one arrow 
with ending at $j$.
This conclude the proof. 
\end{proof}

By Lemma \ref{G-precovering} and  Lemma \ref{LambdaGcovering}  we know the following.

\begin{lemma}\label{PushG}
Let $\tau$ be a triangulation of $\Sigma_n$. Then the push-down functor $\pi_{*}:
\Lambda(T)\module \rightarrow \Lambda(\tau)\module$ is a $G$-precovering.
\end{lemma}

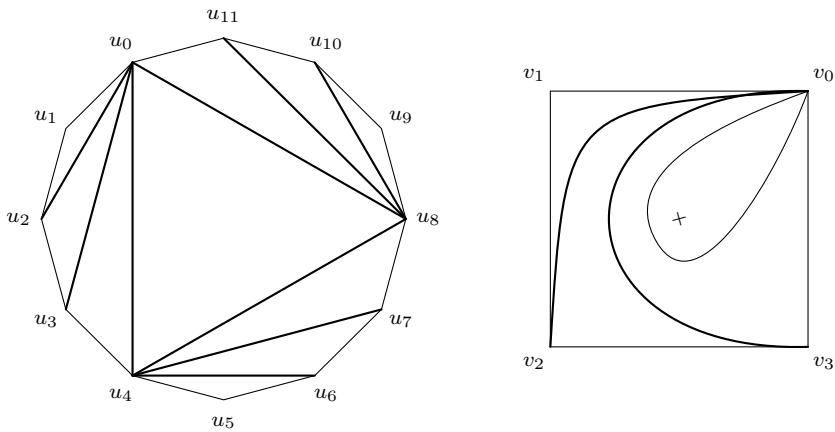
\begin{figure}[ht]
\centering
\begin{tikzpicture}[scale=3,cap=round,>=latex]
\begin{scope}
  \newdimen\R
  \R=0.8cm
  \coordinate (center) at (0,0);
 \draw (0:\R)
     \foreach \x in {30,60,...,360} {  -- (\x:\R) }
              -- cycle (360:\R)
              -- cycle (330:\R) 
              -- cycle (300:\R) 
              -- cycle (270:\R)
              -- cycle (240:\R) 
              -- cycle (210:\R)              
              -- cycle (180:\R)
					    -- cycle (150:\R)
							-- cycle (120:\R)
              -- cycle (90:\R) 
              -- cycle (60:\R) 
              -- cycle (30:\R)
              -- cycle (0:\R);
							\draw[line width=0.8pt] (90:\R)-- (0:\R);
              \draw[line width=0.8pt] (60:\R)-- (0:\R);
							\draw[line width=0.8pt] (120:\R)-- (0:\R);
							\draw[line width=0.8pt] (240:\R)-- (0:\R);
							\draw[line width=0.8pt] (240:\R)-- (330:\R);
							\draw[line width=0.8pt] (240:\R)-- (300:\R);
							\draw[line width=0.8pt] (120:\R)-- (240:\R);
							\draw[line width=0.8pt] (120:\R)-- (210:\R);
							\draw[line width=0.8pt] (120:\R)-- (180:\R);
				\node at (120:0.9) {$u_0$};		
				\node at (150:0.9) {$u_1$};
				\node at (180:0.9) {$u_2$};
				\node at (210:0.9) {$u_3$};
				\node at (240:0.9) {$u_4$};
				\node at (270:0.9) {$u_5$};
				\node at (300:0.9) {$u_6$};
				\node at (330:0.9) {$u_7$};
				\node at (360:0.9) {$u_8$};
				\node at (30:0.9) {$u_9$};
				\node at (60:0.9) {$u_{10}$};
				\node at (90:0.9) {$u_{11}$};
\end{scope}
\begin{scope}[xshift=2cm]	
\node at (0,0) {\rotatebox{60}{$\times$}};
  \newdimen\R
  \R=0.8cm
  \coordinate (center) at (0,0);
 \draw (45:\R)
     \foreach \x in {45,135,225,315} {  -- (\x:\R) }
              -- cycle (315:\R)
              -- cycle (225:\R) 
              -- cycle (135:\R) 
              -- cycle (45:\R);
	    \draw (45:\R) to[out=200,in=120] (-0.1,-0.1);
			\draw (-0.1,-0.1) to[out=300,in=250] (45:0.8);
	    \draw[line width=0.8pt](45:\R)   .. controls (-0.5, 0.5)  .. (225:\R);	
	    \draw[line width=0.8pt](45:\R) ..controls (-0.6,0.6) and (-0.6,-0.6) .. (315:
	    \R);
			\node at (45:0.9) {$v_0$};	
			\node at (135:0.9) {$v_1$};
			\node at (225:0.9) {$v_2$};
			\node at (315:0.9) {$v_3$};
\end{scope}
\end{tikzpicture}
\caption{We obtain a  triangulation of a square with one orbifold point of order 3  as 
the $G$-orbit space of a triangulation of a 12-gon.}
\label{FigExamPS}
\end{figure}			

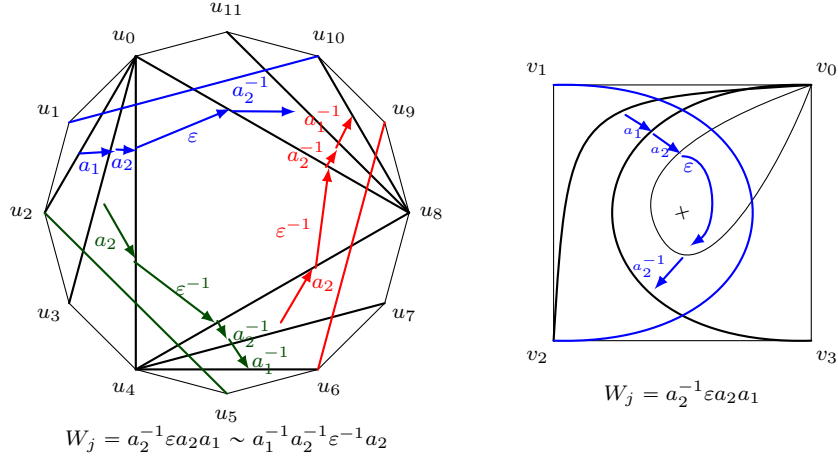
\begin{figure}[ht]
\centering
\begin{tikzpicture}[scale=3,cap=round,>=latex]
\begin{scope}
  \newdimen\R
  \R=0.8cm
  \coordinate (center) at (0,0);
 \draw (0:\R)
     \foreach \x in {30,60,...,360} {  -- (\x:\R) }
              -- cycle (360:\R)
              -- cycle (330:\R) 
              -- cycle (300:\R) 
              -- cycle (270:\R)
              -- cycle (240:\R) 
              -- cycle (210:\R)              
              -- cycle (180:\R)
					    -- cycle (150:\R)
							-- cycle (120:\R)
              -- cycle (90:\R) 
              -- cycle (60:\R) 
              -- cycle (30:\R)
              -- cycle (0:\R);
							\draw[line width=0.8pt] (90:\R)-- (0:\R);
              \draw[line width=0.8pt] (60:\R)-- (0:\R);
							\draw[line width=0.8pt] (120:\R)-- (0:\R);
							\draw[line width=0.8pt] (240:\R)-- (0:\R);
							\draw[line width=0.8pt] (240:\R)-- (330:\R);
							\draw[line width=0.8pt] (240:\R)-- (300:\R);
							\draw[line width=0.8pt] (120:\R)-- (240:\R);
							\draw[line width=0.8pt] (120:\R)-- (210:\R);
							\draw[line width=0.8pt] (120:\R)-- (180:\R);
				\node at (120:0.9) {$u_0$};		
				\node at (150:0.9) {$u_1$};
				\node at (180:0.9) {$u_2$};
				\node at (210:0.9) {$u_3$};
				\node at (240:0.9) {$u_4$};
				\node at (270:0.9) {$u_5$};
				\node at (300:0.9) {$u_6$};
				\node at (330:0.9) {$u_7$};
				\node at (360:0.9) {$u_8$};
				\node at (30:0.9) {$u_9$};
				\node at (60:0.9) {$u_{10}$};
				\node at (90:0.9) {$u_{11}$};
		\draw[blue, line width=0.8pt] (150:\R)-- (60:\R);	
		\draw[black!70!green, line width=0.8pt] (180:\R)-- (270:\R);	
		\draw[red, line width=0.8pt] (300:\R)-- (30:\R);
		\draw[blue, line width=0.8pt, ->] (158:0.7)-- node[pos=0.3, below]{ $a_1$}
		(151:0.56);
		\draw[blue, line width=0.8pt, ->] (150:0.56)-- node[pos=0.3, below]{ $a_2$}
		(145:0.48);
		\draw[blue, line width=0.8pt, ->] (144:0.49)-- node[pos=0.6, below]{$
		\varepsilon$}(88:0.46);
		\draw[blue, line width=0.8pt, ->] (87:0.45)-- node[pos=0.3, above]{$a_2^{-1}$}
		(56:0.54);
		\draw[red, line width=0.8pt, rotate=-120, <-](158:0.7)-- node[pos=0.1, left]{ 
		$a_1^{-1}$}(151:0.56);
		\draw[red, line width=0.8pt, <-, rotate=-120] (150:0.56)-- node[pos=0.3, left]{ 
		$a_2^{-1}$}(145:0.48);
		\draw[red, line width=0.8pt, <-, rotate=-120] (144:0.49)-- node[pos=0.6, left]
		{$\varepsilon^{-1}$}(88:0.46);
		\draw[red, line width=0.8pt, <-, rotate=-120] (87:0.45)-- node[pos=0.3, right]
		{$a_2$}(56:0.54);
		\draw[black!70!green, line width=0.8pt, <-, rotate=120] (158:0.7)-- 
		node[pos=0.3, right]{ $a_1^{-1}$}(151:0.56);
		\draw[black!70!green,, line width=0.8pt, <-, rotate=120] (150:0.56)-- 
		node[pos=0.3, right]{ $a_2^{-1}$}(145:0.48);
		\draw[black!70!green,, line width=0.8pt, <-, rotate=120] (144:0.49)-- 
		node[pos=0.6, right]{$\varepsilon^{-1}$}(88:0.46);
		\draw[black!70!green,, line width=0.8pt, <-, rotate=120] (87:0.45)-- 
		node[pos=0.3, left]{$a_2$}(56:0.54);
		 \node at (0,-1) {$W_j=a_2^{-1}\varepsilon a_2a_1\sim a_1^{-1}a_2^{-1}
		 \varepsilon^{-1}a_2$};
\end{scope}		
\begin{scope}[xshift=2cm]
\node at (0,0) {\rotatebox{60}{$\times$}};
  \newdimen\R
  \R=0.8cm
  \coordinate (center) at (0,0);
 \draw (45:\R)
     \foreach \x in {45,135,225,315} {  -- (\x:\R) }
              -- cycle (315:\R)
              -- cycle (225:\R) 
              -- cycle (135:\R) 
              -- cycle (45:\R);
	    \draw (45:\R) to[out=200,in=120] (-0.1,-0.1);
			\draw (-0.1,-0.1) to[out=300,in=250] (45:0.8);
	    \draw[line width=0.8pt](45:\R)   .. controls (-0.5, 0.5)  .. (225:\R);	
	    \draw[line width=0.8pt](45:\R) ..controls (-0.6,0.6) and (-0.6,-0.6) .. (315:
	    \R);
			\draw[blue, line width=0.8pt, rotate=180](45:\R) ..controls (-0.6,0.6) and 
			(-0.6,-0.6) .. (315:\R);
			\node at (45:0.9) {$v_0$};	
			\node at (135:0.9) {$v_1$};
			\node at (225:0.9) {$v_2$};
			\node at (315:0.9) {$v_3$};
			\draw[blue, line width=0.8pt, ->] (120:0.5)-- node[pos=0.3, below]{\tiny 
			$a_1$}(110:0.38);
			\draw[blue, line width=0.8pt, ->] (110:0.37)-- node[pos=0.25, below]{\tiny 
			$a_2$}(91:0.26);
			\draw[blue, line width=0.8pt, ->] (90:0.25) to[out=0,in=35] node[pos=0.25, 
			left]{\footnotesize $\varepsilon$}(283:0.15);
			\draw[blue, line width=0.8pt, ->] (270:0.2)-- node[pos=0.25, left]{\tiny 
			$a_2^{-1}$}(250:0.36);
      \node at (0,-0.8) {$W_j=a_2^{-1}\varepsilon a_2a_1$};
\end{scope}
\end{tikzpicture}
\caption{Let $j$ be the blue arc (right). We define $W_j$ from the left. In this case $
\alpha$ can be the blue, red or green arc.  Note that $W_j$ can be  read  
\emph{directly} from the right.}
\label{FigExampleT}
\end{figure}	

 We are going to define a string $W_j(\tau)$ of $\Lambda(\tau)$ for every arc $j\notin 
 \tau$ of $\Sigma_n$. We denote by $\pi\colon \widetilde{\Sigma}_n\rightarrow \Sigma_n$ the 
 canonical projection. We know that $\pi^{-1}(j)=\{ \widetilde{j}, \widetilde{j}
 _{\theta}, \widetilde{j}_{\theta^2}\}$. Recall that $G=\langle \theta\rangle =
 \mathbb{Z}_3$.  Let $T$ be the triangulation in $\widetilde{\Sigma}_n$ corresponding 
 to $\tau$,  see Figure \ref{FigExampleT}. 

Let $j$ be an arc of $\Sigma_n$ such that $j\notin \tau$. Choice $\alpha\in \pi^{-1}(j)
$, by definition $\alpha$ is an arc of $\widetilde{\Sigma}_n$ and $\alpha$ joints two 
vertices  $u_{l}$ and $u_{l+r}$  of $\widetilde{\Sigma}_n$. Every time  $\alpha$ 
crosses two adjacent initial arcs $\gamma:\widetilde{i}_{s_{1}}\rightarrow\widetilde{i}
_{s_{2}}$  of $\widetilde{\tau}$, we write the letter $G\cdot \gamma$ (a letter on 
$Q(\tau)$) if $\alpha$ crosses $\widetilde{i}_{s_{1}}$ first from $u_l$ to $u_{l+r}$ or 
we write the letter $ (G\cdot \gamma)^{-1}$ in other wise, see Example \ref{ExplaPS} . 
This construction does not depend of the choice of $\alpha$ up to string equivalency, 
see Section \ref{section3}. Denote by $W_j(\tau)$ the string of $\Lambda(\tau)$ 
obtained in this way. In Figure \ref{FigExampleT} we show an example of this 
construction.

\begin{definition}\label{DefMCuerda}
Let $\tau$ be a triangulation of $\Sigma_n$. For any arc $j\notin \tau$ we define 
\emph{the arc representation}  $M(j, \tau)$ of $j$ with respect to $\tau$ as the string 
module associated to $W_j(\tau)$, i.e $M(j,\tau)=N(W_j(\tau))$, see Section 
\ref{DefinitionNW}. Since a string and its inverse generate isomorphic string modules 
we have that $M(j,\tau)$ is well defined up to isomorphism. Now, for any arc $j\in \tau
$ we define $M(j,\tau):=\mathcal{S}_{j}^{-}$ as the corresponding negative simple 
representation of $\Lambda(\tau)$.
\end{definition}

As long as there is no confusion   we ease the notation and  write $W_j:=W_j(\tau)$ and  
$M(j):=M(j,\tau)$.
We finish this section with some remark about the push-down fuctor.

\begin{remark}\label{DefPush}
For  convenience we  are going to define explicitly  the push-down functor in our 
situation. Let $T$ be atriangulation of $\widetilde{\Sigma}_n$. Set      
$\Lambda=\Lambda(T)$  and consider   $\pi:\Lambda \rightarrow 
\Lambda_G $ the canonical projection of the action, where $\Lambda_G=\Lambda(T)_G$. 
We define the push-down functor  $\pi_{*}:\Lambda\module \rightarrow \Lambda_G\module$ 
as follows. 

\emph{For objects}: let $M\in \Lambda\module$ be a $\Lambda$-representation. For $i\in 
Q_0$ we define 
$\pi_*(M)_{G\cdot i}=\bigoplus_{g\in G}{M_{g\cdot i}}$. Let $\alpha:i\rightarrow j$ be 
an arrow of $Q$. We are going to define $\pi_*(M)_{G\cdot \alpha}: \bigoplus_{g\in G}
{M_{g\cdot i}}\rightarrow \bigoplus_{h\in G}{M_{h\cdot j}}$. Now, by definition, for 
any $h\in G$ we have an isomorphism $\pi_{j, h\cdot i}: \bigoplus_{g\in G}\Lambda(g
\cdot i, h\cdot j)\rightarrow \Lambda_G(G\cdot i,G\cdot j)$. So, $G\cdot \alpha=\sum_{g
\in G}{\pi(\alpha_{h,g})}$ for any $h\in G$ and we define $\pi_{*}(M)_{G\cdot \alpha}
=(\alpha_{h,g})_{g,h\in G}$.

\emph{For morphims:} let $f: M\rightarrow N$ be a morphism in $\Lambda\module$. For any 
$i\in Q_0$ we need to define $\pi_*(f)_{G\cdot i}: \bigoplus_{g\in G}{M_{g\cdot i}}
\rightarrow \bigoplus_{h\in G}{N_{h\cdot i}}$. We set $\pi_*(f)_{G\cdot i}=
\diagonal(f_{g\cdot i}: g\in G)$ as a \emph{diagonal map}.
\end{remark}

\section{The Caldero-Chapoton algebra for a specific triangulation}\label{Section6}
In this section we will study the Caldero-Chapoton algebra of a specific triangulation 
$\tau_0$ of $\Sigma_n$. We will see that the arc representations of $\Lambda(\tau_0)$ 
play a central role.
Fix  the vertices of $\Sigma_n$ in  counter clockwise order $\{v_0\ldots, v_n\}$. Let 
$i_n$ be the pending arc at $v_0$. We denote the pending at $v_k$ as $i'_k$ for $k=0,
\ldots, n$. With this notation we see that $i'_0=i_n$. Let $i_k$ be the arc from $v_0$ 
to $v_{k+1}$ going in counterclockwise for $k=1,\ldots, n-1$. We define the 
\emph{special  triangulation} $\tau_0$ of $\Sigma_n$ as the collection of arcs $\{i_1,
\ldots,i_n\}$, see on the right hand of Figure \ref{FigExamPS}.

For $\tau_0$ we have a nice description of the concepts introduced in Section 
\ref{Section5}, for instance, the weighted quiver associated to $\tau_0$ looks like
\begin{equation*}\label{QA_n}
\xymatrix{  
Q(\tau_0)^*\colon1\ar@{->}[r]^{a_1} & 2\ar@{->}[r]^<(0.3){a_2}&  \cdots\ar@{->}
[r]^<(0.3){a_{n-2}}  & n-1\ar@{->}[r]^<(0.4){a_{n-1}}&n,  
}
\mbox{ and } \textbf{d}_{\tau_0}=(1,1,\ldots, 1,2).
\end{equation*}

The matrix $B(\tau_0)$ is going to be our input to obtain the polynomials of 
Chekhov-Shapiro and we are going to describe a basic algebra associated to $\tau_0$.

Let $\Lambda:=\Lambda(\tau_0)$ be the basic algebra associated to $\tau_0$, it is clear 
that $\Lambda$ is given by $\mathbb{C}\langle Q(\tau_0)\rangle/I$ where $Q(\tau_0)$ is 
the quiver 
\begin{equation}\label{SQA_n}
\xymatrix{  
1\ar@{->}[r]^{a_1} & 2\ar@{->}[r]^<(0.3){a_2}&  \cdots\ar@{->}[r]^<(0.3){a_{n-2}}  & 
n-1\ar@{->}[r]^<(0.4){a_{n-1}}&n\ar@(ur,dr)^{\varepsilon}   
}
\end{equation}
and $I$ is the ideal generated by $\varepsilon^2$.
 
For every arc $j$ of $\Sigma_n$ we defined a \emph{decorated indecomposable 
representation} $M(j)$ of $\Lambda$ with respect to $\tau_0$, see Definition 
\ref{DefMCuerda}. 

\begin{remark}\label{Obs1}
For any arc $j\notin \tau_0$ Definition \ref{DefMCuerda} can be rewritten up to 
isomorphism by counting intersection numbers  directly in $\Sigma_n$. We give this 
approach for convenience,

\begin{equation*}
\dimension(M(j))_l= |i_l \cap j| \mbox{ in the interior of } \Sigma_n.
\end{equation*}
Given an arrow $a_l$ with $l=1,\ldots, n-1$, we define $M(j)_{a_l}$ as follows: 
\begin{itemize}
\item if $0< \dimension(M(j)_{t(a_l)}) < \dimension(M(j)_{h(a_l)})$, then  $M(j)_{a_l}=
\left( \begin{smallmatrix}
0 \\
1 \end{smallmatrix} \right)$;
\item if $0<\dimension(M(j)_{t(a_l)}) = \dimension(M(j)_{h(a_l)})$, then $M(j)_{a_l}$ 
acts as the corresponding identity;
\item $M(j)_{a_l}=0$ in otherwise.
\end{itemize}
If $\dimension(M(j)_{n})\neq 0$, then $M(j)_{\varepsilon}=\left( \begin{smallmatrix}
0 & 1 \\
0 & 0 \end{smallmatrix} \right)$.
\end{remark}

\begin{remark}
From Definition \ref{DefMCuerda} and Theorem \ref{BRTheorem} we know $M(j)$ is 
indecomposable in $\decrep(\Lambda)$ for every arc $j$. Remark \ref{Obs1} allow us to 
compute $M(j)$ without  $\widetilde{\Sigma}_n$.
\end{remark}

For any arc $j\notin \tau_0$ we define the \emph{support} of $M(j)$ as $\Supp M(j)=\{l
\colon M(j)_l\neq 0\}$. The same argument of \cite[Lemma 2.2]{CCS-04} can be applied 
here to conclude that $\Supp M(j)$ is connected as a subset of $[1,n]$. So, we are 
going to think that $\Supp M(j)$ is an interval.

\begin{example}
For $n=5$, we compute $M(j_l)$ with $l=1,2$ and $3$, see Figure \ref{FigExample1}. 
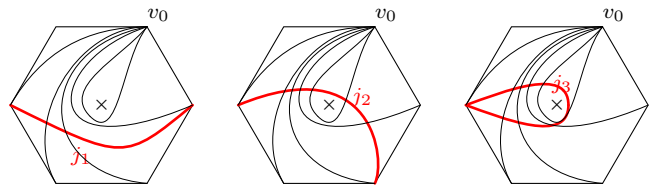
\begin{figure}[ht]
\centering
\begin{tikzpicture}[scale=1.5,cap=round,>=latex]
\begin{scope}
\node at (0,0) {$\times$};
  \newdimen\R
  \R=0.8cm
  \coordinate (center) at (0,0);
 \draw (0:\R)
     \foreach \x in {60,120,...,360} {  -- (\x:\R) }
              -- cycle (300:\R) 
              -- cycle (240:\R) 
              -- cycle (180:\R) 
              -- cycle (120:\R) 
              -- cycle (60:\R)
              -- cycle (0:\R);
	    \draw[red][line width=1pt](180:\R) .. controls (0.2, -0.5) .. node[pos=0.27, 
	    below]{$j_1$} (0:\R);
			\draw[line width=0.3pt](60:\R) to[out=230,in=135] (-0.1,-0.1);
			\draw[line width=0.3pt](-0.1,-0.1) to[out=315,in=232] (60:0.8);
			\draw[line width=0.3pt](60:\R) to[out=180,in=60] (180:0.8);
			\draw[line width=0.3pt](60:\R) to[out=180,in=120] (240:0.8);
			\draw[line width=0.3pt](60:\R)   .. controls (-0.6, 0.5) and (-0.6, -0.6) 
			.. (300:\R);	
			\draw[line width=0.3pt](60:\R)   .. controls (-0.5, 0.4) and (-0.6, -0.6) 
			.. (360:\R);	
		\node at (0.5,0.8) {$v_0$};
	\end{scope}
	\begin{scope}[xshift=2cm]
	\node at (0,0) {$\times$};
  \newdimen\R
  \R=0.8cm
  \coordinate (center) at (0,0);
 \draw (0:\R)
     \foreach \x in {60,120,...,360} {  -- (\x:\R) }
              -- cycle (300:\R) 
              -- cycle (240:\R) 
              -- cycle (180:\R) 
              -- cycle (120:\R)
              -- cycle (60:\R) 
              -- cycle (0:\R);				
	    \draw[red][line width=1pt](180:\R)   .. controls (0.4, 0.5) and (0.5, -0.5) ..  
	    node[pos=0.5, above]{$j_2$}(300:\R);
	    \draw[line width=0.3pt](60:\R) to[out=230,in=135] (-0.1,-0.1);
			\draw[line width=0.3pt](-0.1,-0.1) to[out=315,in=232] (60:0.8);
			\draw[line width=0.3pt](60:\R) to[out=180,in=60] (180:0.8);
			\draw[line width=0.3pt](60:\R) to[out=180,in=120] (240:0.8);
			\draw[line width=0.3pt](60:\R)   .. controls (-0.6, 0.5) and (-0.6, -0.6) 
			.. (300:\R);	
			\draw[line width=0.3pt](60:\R)   .. controls (-0.5, 0.4) and (-0.6, -0.6) 
			.. (360:\R);		
 \node at (0.5,0.8) {$v_0$};
	\end{scope}
	 
\begin{scope}[xshift=4cm]
	\node at (0,0) {$\times$};
  \newdimen\R
  \R=0.8cm
  \coordinate (center) at (0,0);
 \draw (0:\R)
     \foreach \x in {60,120,...,360} {  -- (\x:\R) }
              -- cycle (300:\R) 
              -- cycle (240:\R)
              -- cycle (180:\R)
              -- cycle (120:\R)
              -- cycle (60:\R)
              -- cycle (0:\R);						
	    \draw[red][line width=1pt] (180:\R) to[out=20,in=90] node[pos=0.6, right]
	    {$j_3$}(0.1,0);
			\draw[red][line width=1pt] (0.1,0) to[out=270,in=340](180:0.8);
	    \draw[line width=0.3pt](60:\R) to[out=230,in=135] (-0.1,-0.1);
			\draw[line width=0.3pt](-0.1,-0.1) to[out=315,in=232] (60:0.8);
			\draw[line width=0.3pt](60:\R) to[out=180,in=60] (180:0.8);
			\draw[line width=0.3pt](60:\R) to[out=180,in=120] (240:0.8);
			\draw[line width=0.3pt](60:\R)   .. controls (-0.6, 0.5) and (-0.6, -0.6) 
			.. (300:\R);	
			\draw[line width=0.3pt](60:\R)   .. controls (-0.5, 0.4) and (-0.6, -0.6) 
			.. (360:\R);		
			 \node at (0.5,0.8) {$v_0$};
	\end{scope}
\end{tikzpicture}
\caption{ Some arcs for $n=5$.}
\label{FigExample1}
\end{figure} 

We  have 
\begin{equation*}
\begin{array}{l}
\xymatrix{  
M(j_1):0\ar@{->}[r] &\mathbb{C}\ar@{->}[r]^<(0.4){\ident}& \mathbb{C} \ar@{->}[r]&
 0\ar@{->}[r]&0\ar@(ur,dr)^{ 0}
},\\
\xymatrix{  
M(j_2):0\ar@{->}[r] &\mathbb{C}\ar@{->}[r]^<(0.4){\ident}& \mathbb{C} \ar@{->}
[r]^<(0.3){ {\scriptstyle \left( \begin{smallmatrix}
0 \\
1 \end{smallmatrix} \right)}}&\mathbb{C}^2\ar@{->}[r]^<(0.4){\ident}&\mathbb{C}
^2\ar@(ur,dr)^<(0.3){ \left( \begin{smallmatrix}
0 & 1 \\
0 & 0 \end{smallmatrix} \right)}},\\
\xymatrix{  
M(j_3):0\ar@{->}[r] & \mathbb{C}^2\ar@{->}[r]^<(0.4){\ident}&\mathbb{C}^2\ar@{->}
[r]^<(0.4){\ident}& \mathbb{C}^{2} \ar@{->}[r]^<(0.4){\ident}&\mathbb{C}
^2\ar@(ur,dr)^<(0.3){ \left( \begin{smallmatrix}
0 & 1 \\
0 & 0 \end{smallmatrix} \right)}
}.
\end{array}
\end{equation*}

As illustration we have the Caldero-Chapoton function  $\mathcal{C}_{\Lambda}(M(j_2))$
{\footnotesize
\[
\frac{x_1x_2x_3^2x_4^2+ x_1x_2x_3^2x_4+ x_1x_2x_3x_4x_5 +x_1x_2x_3^2+2x_1x_2x_3x_5+ 
x_1x_2x_5^2+x_1x_4x_5^2+ x_3x_4x_5^2 +x_1x_3x_4x_5+x_3^2x_4x_5}{x_2x_3x_4^2x_5}.
\]
}
\end{example}

\subsection{AR translations, E-invariant and g-vectors of arc representations}
Let $j$ be an arc of $\Sigma_n$. We introduce some notation. Given two vertices $v_r$ 
and $v_l$ of $\Sigma_n$ with $r+1<l$ and $r\in\{1,\ldots,n-2\}$ we have two arcs from 
$v_r$ to $v_l$ denoted by $[v_r, v_l]^+$ and $[v_r,v_l]^-$. Indeed,if $0<r$, then $
[v_r, v_l]^+$ does not intersect to $i_n$ in the interior of $\Sigma_n$ while $[v_r, 
v_l]^-$ does. For example, in the Figure \ref{FigExample1} we have $j_1=[v_2,v_5]^+$ 
and $j_2=[v_2,v_4]^-$. 
For $r=0$ we say that $i_k=[v_0,v_{k+1}]^-$ and $[v_0,v_{k+1}]^+$ is the another arc 
from $v_0$ to $v_{k+1}$ with $k=1, \ldots, n-1$. In the case $l=r+1$, we have $[v_r, 
v_{r+1}]^-$ is not a boundary segment.

\begin{remark}\label{Obs2}
If $n\geq 4$, with the above notation we can describe $W_j$ explicitly for any $j\notin
\tau_0$. 
\begin{itemize}
\item $W_{[v_1, v_n]^+}=a_{n-3}\cdots a_1$ and $W_{[v_1, v_n]^-}=\varepsilon\cdots a_1$;
\item  $W_{[v_i, v_l]^+}=a_{l-3}\cdots a_i$ and $W_{[v_i, v_l]^-}=a_{l}^{-1}\cdots 
a_{n-1}^{-1}\varepsilon\cdots a_i$ for $0<i$ and $i+2<l<n$;
\item  $W_{[v_i, v_{l}]^+}=1_{(i,+)}$ and $W_{[v_i, v_l]^-}=a_{l}^{-1}\cdots a_{n-1}
^{-1}\varepsilon\cdots a_i$ for $l=i+2$ and $i\leq n-2$;
\item   $W_{[v_i, v_l]^-}=a_{l}^{-1}\cdots a_{n-1}^{-1}\varepsilon\cdots a_i$ for $l=i
+1$ and $0<i<n-2$;
\item $W_{[v_{n-1}, v_n]^-}= \varepsilon a_{n-1}$;
\item  $W_{[v_0, v_l]^+}=a_{n-2}\cdots a_l$ for $1\leq l<n-1$;
\item  $W_{i'_k}=a_{k}^{-1}\cdots a_{n-1}^{-1}\varepsilon\cdots a_k$ for $1\leq k\leq 
n-1$; 
\item $W_{i'_n}=\varepsilon$.
\end{itemize}
\end{remark}

The reader can compare the following lemma with the $A_n$ case,  \cite[Theorem 2.13]
{CCS-04}. Given an arc $j$ we denote by $r^+(j)$ ( in  \cite{CCS-04} it would be $r^-$)  
the arc of $\Sigma_n$ that we obtain by rotating $j$ in counterclockwise for an angle 
of $2\pi/(n+1)$. By $r^-(j)$ ( in  \cite{CCS-04} it would be $r^+$) we denote the arc 
obtained from $j$ by rotating $j$ in clockwise for an angle of $2\pi/(n+1)$. 

\begin{lemma}\label{LemmaTau}
Assume $j$ is not an initial arc of $\Sigma_n$.
\begin{itemize}
\item[(a)] If $M(j)$ is not projective, then $\tau(M(j))=M(r^+(j))$, where $\tau$ 
denotes the Auslander-Reiten translation.
\item[(b)] If $M(j)$ is not injective, then $\tau^-(M(j))=M(r^-(j))$.
\end{itemize}
\end{lemma}
\begin{proof}
The lemma can be proved by cases using Remark \ref{Obs2} and 
the classification of the Auslander-Reiten sequences containing string modules from 
\cite[p.p 170-172]{BR-87}. However, we proved this result in Corollary \ref{CoroTau+}.
\end{proof}

\begin{lemma}\label{LemmaE0}
Assume that $j$ is an arc of $\Sigma_n$. Then $E_{\Lambda}(M(j))=0$.
\end{lemma}
\begin{proof}
We postpone this proof up to Corollary \ref{Lem13.1}. It is worth mention that this 
lemma can be proved with the alluded results of \cite{BR-87}.
\end{proof}

\begin{lemma}\label{LemmaA}
Assume $j$ is an arc of $\Sigma_n$ such that $j\notin \tau_0$. Then $\pd M(j)\leq 1$ 
and $\injd M(j)\leq 1$. Here $\pd M(j)$ (resp. $\injd M(j)$) denotes the 
\emph{projective dimension} of $M(j)$ (resp. the \emph{injective dimension} of M(j)).
\end{lemma}
\begin{proof}
The lemma follows from  \cite[Proposition 3.5]{GLS-15} and the definition of $M(j)$. 
Indeed, in the language of \cite{GLS-15}, if we take 
\begin{equation*}
C=\begin{pmatrix}
    2 & -1 & 0 & \dots  & 0 & 0 \\
    -1 & 2 & -1 & \dots  & 0 & 0 \\
    0 & -1 & 2  & \ldots & 0 & 0 \\
    \vdots & \vdots & \vdots & \ddots & \vdots & \vdots \\
    0 & 0 & 0 & \ldots & 2 & -2 \\
    0 & 0 & 0 & \dots  & -1 & 2
\end{pmatrix}, \ 
D=\begin{pmatrix}
    1 & 0 & \dots  & 0 & 0 \\
    0 & 1 & \dots  & 0 & 0 \\
    \vdots  & \vdots & \ddots & \vdots & \vdots \\
    0  & 0 & \ldots & 1 & 0 \\
    0 & 0 & \dots  & 0 & 2
\end{pmatrix},
\end{equation*} 
and $\Omega=\{(i+1,i)\colon 1\leq i\leq n-1\}$, then we get $\Lambda=H(C,D,\Omega)$. By 
definition $M(j)$ is \emph{locally free}  $\Lambda$-module for every arc $j\notin 
\tau_0$ (see  \cite[Section 1.5]{GLS-15}).
\end{proof}

\begin{remark}
In \cite{GLS-18}, Gei\ss-Leclerc-Schr\"oer have proved that, in particular, for this 
algebra $\Lambda=H(C,D,\Omega)$ of the previous lemma, we can recover a classic cluster 
algebra  by means of  Caldero-Chapoton functions. However, they consider the 
quasiprojective variety $\Grass_{\mbox{l.f.}}(\textbf{r},M)$ of \emph{locally free 
submodules} $N$ of $M$ with rank vector $\textbf{r}$  instead  $\Grass_{\textbf{e}}(M)$ 
as we made her, see Example \cite[Section 13.1]{GLS-18}.
\end{remark}

\begin{remark}\label{Obs3}
Lemma \ref{LemmaA} ensures that $\injd M(j)\leq 1$, now it can be seen that we have the 
following minimal injective presentation of $M(j)$ for each arc $j\notin\tau_0$, this 
is a consequence of \cite{BR-87}.
\begin{enumerate}
\item  $j$ crosses to $i_n\colon$ in this case $W_j=a_{n_j}^{-1}\cdots\varepsilon 
\cdots a_{m_j}$ with $m_j\leq n_j$. Then the following exact sequence is a minimal 
injective presentation of $M(j)$,
\[0\rightarrow M(j)\rightarrow   N(a_1^{-1}\cdots \varepsilon\cdots a_1)\rightarrow 
N(a_{n_j-2}\cdots a_1)\oplus N(a_{m_j-2}\cdots a_1).\]
Here we define  $N(a_{r-2}\cdots a_1)$ as zero if $r=1$   and it is the simple 
representation at $1$ if $r=2$.
\item $j$ does not cross to $i_n\colon$ in this case $W_j=a_{n_j}\cdots a_{m_j}$ with 
$n-2\geq n_{j}\geq m_j$. Then the following exact sequence is a minimal injective 
presentation of $M(j)$,
\[0\rightarrow M(j)\rightarrow   N( a_{n_j}\cdots a_1)\rightarrow N(a_{m_j-2}\cdots 
a_1).\]
\end{enumerate}
\end{remark}

\begin{proposition}\label{Prop2}
If $j_1$ and $j_2$ are not arcs of $\tau_0$, then the following hold:
\begin{itemize}
\item There exists a $\mathbb{C}$-linear isomorphism 
\[\Ext^1_{\Lambda}(M(j_1), M(j_2))\cong \Hom_{\Lambda}(\tau^-(M(j_2)), M(j_1)).\]
\item  There exists a $\mathbb{C}$-linear isomorphism 
\[\Hom_{\Lambda}(M(j_1),\tau(M(j_2)))\cong \Hom_{\Lambda}(\tau^-(M(j_1)),M(j_2)).\]
\end{itemize}
\end{proposition}
\begin{proof}
The proposition follows from Lemma \ref{LemmaA},  \cite[Corollary (IV) 2.14(b)]{ASS} 
and \cite[Corollary (IV) 2.15(a)]{ASS} respectively.
\end{proof}

Albeit the previous proposition is true for any modules with projective and injective 
dimension at most 1 we stated it in that fashion for convenience.

\begin{lemma}\label{LemmaCompati}
Let $\tau$ be a triangulation of $\Sigma_n$. If $j_1$ and $j_2$ are arcs of $\tau$, 
then $E_{\Lambda}(M(j_1),M(j_2))=0$.
\end{lemma}
\begin{proof}
If  $j_2\in \tau_0$ or $M(j_2)$ is injective, then $E_{\Lambda}(M(j_1),M(j_2))=0$ for 
all arc $j_1\in \tau$ by definitions and Proposition \ref{PropoTrunc}. So we can 
suppose $j_2$ is not in $\tau_0$ and $M(j_2)$ is not injective.
\begin{itemize}
\item[Case 1.] $j_1 =i_k$ for some $1\leq k\leq n\colon$ in this case $M(j_1)$ is the 
negative simple representation of $\Lambda$ at $k$. It is clear that 
$E_{\Lambda}(M(j_2),M(i_k))=\dimension M(j_2)_k$ by Proposition \ref{PropoTrunc}, but 
$i_k, j_2\in \tau$, then $\dimension M(j_2)_k=0$. 
\item[Case 2.] $j_1\notin \tau_0$ and $M(j_1)$ is injective: then $E_{\Lambda}
(M(j_2),M(j_1))=0$ for all $j_2\in \tau$. We have to prove $E_{\Lambda}
(M(j_1),M(j_2))=0$.
By Proposition \ref{PropoTrunc} we get $E_{\Lambda}(M(j_1),M(j_2))=\dimension 
\Hom_{\Lambda}(\tau^-(M(j_2), M(j_1))$. If  $M(j_1)$ is injective, then $j_1=[v_1,v_l]^
+$ with $2<l\leq n$, $j_1=i'_1$ or $j=[v_0,v_1]^+$. Since $j_1, j_2\in\tau$ and $M(j_2)
$ is not injective, $\Supp(M(j_1))\cap \Supp(\tau^-M(j_2))=\emptyset$ and  $
\Hom_{\Lambda}(\tau^-(M(j_2)), M(j_1))=0$. 
 
\item[Case 3.] $j_1\notin \tau_0$ and $M(j_1)$ is not injective$\colon$ we have to 
prove $E_{\Lambda}(M(j_1),M(j_2))=0$ and $E_{\Lambda}(M(j_2),M(j_1))=0$. For $l=1,2$, 
let $m_l$ be the  minimum positive  integer such that $M((r^{-})^{m_l}(j_l))$ is 
injective.

If $m_1\leq m_2$ , then by \cite[Corollary (IV) 2.15 (c)]{ASS} and Proposition 
\ref{PropoTrunc} we have $$E_{\Lambda}(M(j_1),M(j_2))=\dimension \Hom_{\Lambda}
((\tau^-)^{m_1+1}(M(j_2),(\tau^-)^{m_1} M(j_1)).$$
 \cite[Corollary (IV) 2.14 (b)]{ASS} implies $\dimension \Ext^1((\tau^-)^{m_1} M(j_1), 
 (\tau^-)^{m_1}(M(j_2))) = E_{\Lambda}(M(j_1),M(j_2))$. Since $(\tau^-)^{m_1}( M(j_1))$ 
 is injective, we get $E_{\Lambda}(M(j_1),M(j_2))=0$.
Now, 
\[E_{\Lambda}(M(j_2),M(j_1))=\dimension \Hom_{\Lambda}(\tau^-(M(j_1), M(j_2)).\]
Since $m_1\leq m_2$, we apply  \cite[Corollary (IV) 2.15 (c)]{ASS} to obtain 
$E_{\Lambda}(M(j_2),M(j_1))=0$. The case $m_2< m_1$ is similar.
\end{itemize}
 This proves the lemma.
\end{proof}

\begin{lemma}\label{LemmaE1}
Given an arc $j\notin \tau_0$ we have $\Ext^1(M(j),M(j))=0$.
\end{lemma}
\begin{proof}
By Proposition \ref{PropoTrunc} we have $E_{\Lambda}(M(j))=\dimension \Hom_{\Lambda}( 
\tau^-(M(j)), M(j))$. Proposition \ref{Prop2} implies $E_{\Lambda}(M(j))=\dimension
\Ext^1(M(j),M(j))$. The lemma follows from Lemma \ref{LemmaE0}.
\end{proof}
The following result is a consequence of Voigt's Lemma, see 
\cite[Sections 1.6 and 1.8]{DP-95}.
\begin{lemma}\label{OrbitOpen}
Assume $j\notin \tau_0$. Then the $G_{\textbf{d}}$-orbit $\mathcal{O}(M(j))$ is open in 
$\rep_{\textbf{d}}(\Lambda)$. 
\end{lemma}
\begin{proof}
By Lemma \ref{LemmaE1} we have $\Ext^1(M(j),M(j))=0$, this implies that $\mathcal{O}(M)
$ is open, for example  see \cite[1.7 Corollary 3]{DP-95}.
\end{proof}

By Lemma \ref{LemmaE0} we have examples of $E$-\emph{rigid} indecomposable $\Lambda$-
modules  i.e.$E_{\Lambda}(M(j))=0$. The next result shows that we already know all  
$E$-rigid $\Lambda$-modules, this result can be seen as a consequence of Lemma 
\ref{Lem13.1}, 
but we have been mentioned that for $\tau_0$ the proof of some result may be made with 
explicit computations, here we give an example.

\begin{proposition}\label{ClasE}
If $N$ is a indecomposable  $\Lambda$-module and $N$ is not of the form $M(j)$ for some 
arc $j$ of $\Sigma_n$, then $E_{\Lambda}(N)>0$. 
\end{proposition} 
\begin{proof}
Let $W$ be the string associated to $N$ and assume that $N$ is not $E$-rigid. Given a 
non-initial arc $j$ of $\Sigma_n$ we have the string $W_j$ is one of the following

\begin{equation*}
\begin{split} 
 &1_{(i,+)}, \mbox{ with  }i\in [1,n-1], \\
  & \varepsilon \cdots a_{m_j} \mbox{ with } m_j\in [1,n-1],\\
  & a_{n_j}^{-1}\cdots\varepsilon \cdots a_{m_j} \mbox{ with } m_j\leq n_j \mbox{ and } 
  m_j\in [1,n-1],\\
  & a_{n_j}\cdots a_{m_j} \mbox{ with } n-2\geq n_{j}\geq m_j.
\end{split}
\end{equation*}
Therefore  $W$ is  different to $W_j$ for any arc $j$ of $\Sigma_n$, here we use Remark 
\ref{Obs2}.
If $W=a_{n-1}\cdots a_l$ with $l>1$ we have $N(W)$ looks like
\[0\rightarrow \cdots 0\rightarrow\mathbb{C}\rightarrow \cdots \rightarrow \mathbb{C}.
\]
By \cite{BR-87} we get $\tau^{-1}(N(W))=N(\varepsilon^{-1}a_{n-1}\cdots a_{l-1})$. 
Since 
\[\xymatrix @-0.8pc {  
N(\varepsilon^{-1}a_{n-1}\cdots a_{l-1}):0\ar@{->}[r] & \cdots\ar@{->}[r] & 0\ar@{->}
[r] & \mathbb{C}\ar@{->}[r]^<(0.4){\ident}& \cdots\ar@{->}[r]^<(0.4){\ident} & 
\mathbb{C} \ar@{->}[r]^<(0.3){ {\scriptstyle \left( \begin{smallmatrix}
1 \\
0 \end{smallmatrix} \right)}}&\mathbb{C}^2\ar@(ur,dr)^<(0.3){ \left( 
\begin{smallmatrix}
0 & 1 \\
0 & 0 \end{smallmatrix} \right)}}\]
we have $\Hom_{\Lambda}(\tau^{-1}(N(W)), N(W))\neq 0$. Proposition \ref{PropoTrunc} 
implies $E_{\Lambda}(N(W))>0$.
The case when $l=1$ is similar and follows from \cite{BR-87}.

If $W=a_{n_W}^{-1}\cdots\varepsilon^{-1} \cdots a_{m_W}$ with $1<m_W<n_W$, then 
\[\xymatrix @-0.8pc {  
N(W):0\ar@{->}[r] & \cdots\ar@{->}[r] & 0\ar@{->}[r] & \mathbb{C}\ar@{->}[r]^<(0.4)
{\ident}& \cdots\ar@{->}[r]^<(0.4){\ident} & \mathbb{C} \ar@{->}[r]^<(0.3){ 
{\scriptstyle \left( \begin{smallmatrix}
1 \\
0 \end{smallmatrix} \right)}}&\mathbb{C}^2\ar@{->}[r]^<(0.4){\ident}&\cdots\ar@{->}
[r]^<(0.4){\ident}&  \mathbb{C}^2\ar@(ur,dr)^<(0.3){ \left( \begin{smallmatrix}
0 & 1 \\
0 & 0 \end{smallmatrix} \right)}}\]
and by \cite{BR-87} we get $\tau^{-1}(N(W))=N(a_{n_W-1}Wa_{m_W-1})$. From definitions 
we get $\Hom_{\Lambda}(\tau^{-1}(N(W)), N(W))\neq 0$, so $E_{\Lambda}(N(W))>0$. The 
case when $m_W=1$ is  similar and follows from \cite{BR-87}. Since the indecomposable $
\Lambda$-modules are parametrized by strings,  the proposition follows from Theorem 
\ref{BRTheorem}.
\end{proof}

Now we interpret the $g$-vector of a representation $M(j)$ in terms of intersection 
numbers. The three lemmas below follow from Remark \ref{Obs3} and Lemma 
\ref{gVectorLemma}. For instance
\begin{lemma}
For a pending arc $i_k'$ with $k\in\{1, 2,\ldots, n\}$ we have
\begin{equation*}
g_{\Lambda}(M(i'_k))_l=
\left\{
\begin{array}{lll}
 \ \ 2 \mbox{ if } l=k-1,\\
-1 \mbox{  if } l=n,\\
 \ \ 0 \mbox{ in otherwise}.
\end{array}
\right.
\end{equation*}
\end{lemma}
\begin{proof}
We start with  the pending arc $i'_k$, from Remark \ref{Obs3} we obtain $I_n=N(a_1^{-1}
\cdots \varepsilon\cdots a_1)$ and  $N(a_{n_j-2}\cdots a_1)= N(a_{m_j-2}\cdots a_1)$ 
because $n_j=k=m_j$, remember that the string associated to $i'_k$ is $a_k^{-1}\cdots 
\varepsilon\cdots a_k$. By the other hand  $I_{k-1}=N(a_{k-2}\cdots a_1)$. The result 
follow from Lemma \ref{gVectorLemma}.
\end{proof}

\begin{lemma}
Let $j$ be an arc. If $j$ is not a initial arc, it is not a pending arc and it 
intersects to $i_n$ in the interior of $\Sigma_n$, then we have
\begin{equation*}
g_{\Lambda}(M(j))_l=
\left\{
\begin{array}{lll}
 \ \ 1 \mbox{ if } l+1  \mbox{ is the minimum of } k \mbox{ such that }
 \dimension(M(j))_k=1,\\
  \ \ 1 \mbox{ if } l+1 \mbox{ is the minimum of } k \mbox{ such that }
  \dimension(M(j))_k=2,\\
-1 \mbox{  if } l=n,\\
 \ \ 0 \mbox{ in otherwise}.
\end{array}
\right.
\end{equation*}
\end{lemma}

\begin{lemma}
Let $j$ be an arc. If $j$ is not an initial arc, it is not a pending arc and  it does 
not intersect to $i_n$, then we have
\begin{equation*}
g_{\Lambda}(M(j))_l=
\left\{
\begin{array}{lll}
  \ \ 1 \mbox{ if } l+1 \mbox{ is the minimum of } k \mbox{ such that }
  \dimension(M(j))_k=1,\\
-1 \mbox{  if } l \mbox{ is the maximum of } k \mbox{ such that }\dimension(M(j))_k\neq 
0,\\
 \ \ 0 \mbox{ in otherwise}.
\end{array}
\right.
\end{equation*}
\end{lemma}

\begin{proposition}\label{PropIndLin}
The set 
\[
\{\mathcal{C}_{\Lambda}(M(j))\colon j \mbox{ is an arc of   } \Sigma_n\} 
\]
is linearly independent over $\mathbb{C}$.
\end{proposition}
\begin{proof}
From the three lemmas above we have that  the $g$-vectors $g_{\Lambda}(M(j))$ are 
pairwise different. For $n$ even the result  follows directly from \cite[Proposition 
4.3]{CLS-15}  since $\ker(C_Q)=0$. For $n$ arbitrary we can adapt the argument in proof 
of \cite[Proposition 4.3]{CLS-15} as follows.
Define 
\[
\begin{split}
&\mathbb{Q}^n_{\geq 0}=\{(x_1,x_2,\ldots,x_n)\in \mathbb{Q}^n\colon x_i\geq 0 \mbox{ 
for all } i\},\\ 
&\mathbb{Q}^n_{\mbox{succ}}=\{(x_1,x_2,\ldots,x_n)\in \mathbb{Q}^n_{\geq 0}\colon x_i
\neq 0 \mbox{ implies } x_{i+1}\neq 0 \mbox{   } \},\\
&\mathbb{Q}^{n}_{0}=\{ (x_1,x_2,\ldots,x_n)\in \mathbb{Q}^n_{\geq 0}\colon x_n=0 \}.
\end{split}
\]
We can define two partial orders in $\mathbb{Z}^n$. Let $\textbf{a},\textbf{b}\in 
\mathbb{Z}^n$ be vectors. We say $\textbf{a}\leq \textbf{b}$ if there exist some $
\textbf{e}\in \mathbb{Q}^n_{\mbox{succ}}$ such that $\textbf{a}=\textbf{b} +C_Q
\textbf{e}$ and $\textbf{a}\preceq \textbf{b}$ if there exist some $\textbf{f}\in 
\mathbb{Q}^{n}_{0}$ such that $\textbf{a}=\textbf{b} +C_Q\textbf{f}$. These two 
orders induce two partial orders on the set of Laurent  monomials in $n$ variables 
$x_1, x_2,\ldots, x_n$. We say  $\textbf{x}^{\textbf{a}}\leq \textbf{x}^{\textbf{b}}$ 
if $\textbf{a}\leq  \textbf{b}$ and  $\textbf{x}^{\textbf{a}}\preceq \textbf{x}
^{\textbf{b}}$ if $\textbf{a}\preceq  \textbf{b}$. We define the \emph{degree} of $
\textbf{x}^{\textbf{a}}$ as $\deg(\textbf{x}^{\textbf{a}})=\textbf{a}$.

If $\soc(M(j))=S_n$ (the \emph{socle} of $M(j)$) , then  $\mathcal{C}_{\Lambda}(M(j))$ 
has an unique monomial of maximal degree with respect to $\leq$, namely $g_{\Lambda}
(M(j))$. If $\soc(M(j))=S_i$ with $i\neq n$, then  $\mathcal{C}_{\Lambda}(M(j))$ has an 
unique monomial of maximal degree with respect to $\preceq$ given by $g_{\Lambda}(M(j))
$. Since the $g$-vectors are pairwise different we have that the Caldero-Chapoton 
functions are pairwise different. 
Now assume $\lambda_1\mathcal{C}_{\Lambda}(M(j_1))+\cdots +\lambda_t\mathcal{C}
_{\Lambda}(M(j_t))=0$ for some $\lambda_l\in\mathbb{C}$. We can assume that  $\lambda_l
\neq 0$ for all $l$. 

It can be seen that if $\soc (M(j))=S_n$, then $\textbf{x}^{g_{\Lambda}(M(j))}$ does 
not occur as a summand of any $\mathcal{C}_{\Lambda}(M(k))$ with $\soc (M(k))=S_i$ and 
$i<n$. If there exist an index $s_0$ such that $M(j_{s_0})$ has socle $S_n$, then there 
exist an index $s$ such that $\textbf{x}^{g_{\Lambda}(M(j_s))}$ is $\leq$-maximal  in 
the set of $\{\textbf{x}^{g_{\Lambda}(M(j_l))}\colon \soc(M(j_l))=S_n \}$. Since the $g
$-vectors are pairwise different we can conclude that $\lambda_s=0$. Indeed, $
\textbf{x}^{g_{\Lambda}(M(j_s))}$ does not occur as a summand of any $\mathcal{C}
_{\Lambda}(M(j_l))$ with $l\neq s$, which is a contradiction.

If $\soc(M(j_l))\neq S_n$ for all $l$, then there exist an index $r$ such that $
\textbf{x}^{g_{\Lambda}(M(j_r))}$ is $\preceq$-maximal in the set of $\{\textbf{x}
^{g_{\Lambda}(M(j_l))}\colon 1\leq l\leq t \}$. Since the $g$-vectors are pairwise 
different we have that $\textbf{x}^{g_{\Lambda}(M(j_r))}$ does not occur as a summand 
of any of the $\mathcal{C}_{\Lambda}(M(j_l))$ with $l\neq r$. Thus $\lambda_r=0$, a 
contradiction. Therefore  $\mathcal{C}_{\Lambda}(M(j_1)), \ldots, \mathcal{C}_{\Lambda}
(M(j_t))$ are linearly independent.
\end{proof}

\subsection{Generic version}\label{GenericSect}
 In this section we obtain \emph{generic} version of the results of the 
last section. Given a  triangulation $\tau$ of $\Sigma_n$ we construct an 
\emph{strongly reduced irreducible component} of $\decrep(\Lambda)$, see Section 
\ref{SRsection}.

Denote by $Z_j$ the irreducible component of $\decrep(\Lambda)$ that contains  $
\mathcal{O}(M(j))$. We know $\mathcal{O}(M(j))$ is open, so it is dense in $Z_j$. Then 
$E_{\Lambda}(Z_j)=E_{\Lambda}(M(j))=0$. In the notation of Section \ref{SRsection} 
this means, in particular, that $Z_j$ is a strongly reduced  irreducible component  of $
\decrep(\Lambda)$. We can think that some generic homological data of $Z_j$ is encoded 
in the homological data of $M(j)$.  
\begin{proposition}\label{PropE1}
Given a triangulation $\tau$ of $\Sigma_n$ and two arcs $j_1,j_2\in \tau$ we have 
$E_{\Lambda}(Z_{j_1},Z_{j_2})=0$.
\end{proposition}
\begin{proof}
By Lemma \ref{LemmaCompati} we know $E_{\Lambda}(M(j_1),M(j_2))=0$. It can be seen that 
the set 
$\mathcal{O}(M(j_1))\times \mathcal{O}(M(j_2))$ is open in $Z_{j_1}\times Z_{j_2}$. 
Indeed, from Lemma \ref{OrbitOpen} we know that $\mathcal{O}(M(j_l))$ is open in 
$Z_{j_l}$ for $l=1,2$.
The claim follows from the  equality of sets
\[
A\times B\setminus (C \times D)=[(A\setminus C)\times B] \cup [A\times (B\setminus D)], 
\]
because 
$Z_{j_1}\times Z_{j_2}\setminus [\mathcal{O}(M(j_1))\times \mathcal{O}(M(j_2))]$
would be the union of two closed sets.

If $(M,N)\in \mathcal{O}(M(j_1))\times \mathcal{O}(M(j_2))$, then $E_{\Lambda}(M,N)=0$. 
Since $Z_{j_1}$ and $Z_{j_2}$ are irreducible, we have $\mathcal{O}(M(j_1))\times 
\mathcal{O}(M(j_2))$ is dense in $Z_{j_1}\times Z_{j_2}$. Then $E_{\Lambda}(Z_{j_1}, 
Z_{j_2})=0$.
\end{proof}

The next result is a consequence of Theorem \ref{CDIsr} and Proposition \ref{PropE1}.

\begin{proposition}
Given a triangulation $\tau=\{j_1,\ldots ,j_n\}$ of $\Sigma_n$, the closed set
\[
Z_{\tau}=\overline{Z_{j_1}\oplus \cdots \oplus Z_{j_n}}
\]
is a strongly reduced irreducible component of $\decrep(\Lambda)$.
\end{proposition}

We introduce some notation that, albeit  \emph{non-standard},  shall
be useful to us. Given a vector $v=(v_1,\ldots, 
,v_n)^t\in \mathbb{Z}^n$ we write $v=v_1[1]+\cdots 
+v_n[n]$. 
Moreover, we write $[n_1, n_2]\subseteq [1,n]$ to 
simplify $1[n_1]+\cdots +1[n_2]$.  For example, with 
this notation,  $2[n_1,n_2]$ means $2[n_1]+\cdots 
+2[n_2]$.

The next proposition generalizes   \cite[Proposition 9.4]{CLS-15}.

\begin{proposition}\label{ProGen}
The set 
\[
\{\mathcal{C}_{\Lambda}(Z)\colon Z\in \firr(\Lambda), E_{\Lambda}(Z)=0\}
\]
generates the Caldero-Chapoton algebra $\mathcal{A}_{\Lambda}$ as $\mathbb{C}$-algebra, 
where $\firr(\Lambda)$ denotes the strongly reduced irreducible  components of $
\decrep(\Lambda)$.
\end{proposition}
\begin{proof}
Only remains to prove that the Caldero-Chapoton functions of the non-$E$-rigid 
representations can be expressed in terms of the Caldero-Chapoton functions of the 
$E$-rigid representations. Let $L_1=a_{n_1}^{-1}\cdots\varepsilon \cdots a_{m_1}$ with 
$m_1<n_1$ be a string and  let $m_2\leq n$ be an integer. A direct calculation yields 
the following equations
\[
\begin{split}
&\mathcal{C}_{\Lambda}(N(L_1'))=\mathcal{C}_{\Lambda}(N(L_1))+ \mathcal{C}_{\Lambda}
(N(W_{[m_1,n_1-2]}))\\
&\mathcal{C}_{\Lambda}(N(W_{[m_2,n]}))=\mathcal{C}_{\Lambda}(N(W_{[m_2,n-1]}))+
\mathcal{C}_{\Lambda}(\mathcal{S}_{m_2-1}^-)
\end{split}
\]
Here we set $W_{\emptyset}=0$ and $\mathcal{S}_{0}^-:=0$. The proposition  follows 
from Proposition \ref{ClasE}.
Indeed, let us verify the first equality:
set $N_1=N(L_1)$, $N_1'=N(L_1')$ and $M=N(W_{[m_1,n_1-2]})$. With that convention we 
obtain that 
$g_{\Lambda}(N_1)=[n_1-1]+[m_1-1]-[n]=g_{\Lambda}(N_1')$ and $g_{\Lambda}(M)=-[n_1-2]+
[m_1-1]$.

For $M$: the dimension vector of sub-representations are given by $\textbf{h}_0=
\textbf{0}$ and 
$\textbf{h}_i=[n_1-i-1, n_1-2]$, where $i\in[1, n_1-m_1-1]$.  From this we get the 
vectors 
$C_Q\textbf{h}_0=\textbf{h}_0$ and $C_Q\textbf{h}_{i}= [n_1-1, n_1-2]-[n_1-2-i, n_1-1-
i]$, where 
$i\in\{1,\ldots, n_1-m_1-1\}$. Therefore,

\[
\mathcal{C}_{\Lambda}(M)=\textbf{x}^{-[n_1-2]+[m_1-1]}\sum_{i=0}^{n_1-m_1-1}{\textbf{x}
^{-[n_1-i-2,
n_1-i-1]+[n_1-1]+[n_1-2]}}.
\]

Note that all the dimension vectors of $N_1$ are dimension vector of $N_1'$ and the 
only vectors 
that are dimension vectors of $N_1'$ but not of $N_1$ are $\textbf{e}_j=[n-j+1, n]$, 
where $j\in 
[ n-n_1+2, n-m_1+1 ]$. Besides, if we consider $\textbf{a}$ as dimension vector of 
$N_1$, then 
$\chi(\Grass_{\textbf{a}}(N_1))=\chi(\Grass_{\textbf{a}}(N_1'))$ and   
$\chi(\Grass_{\textbf{e}_i}(N_1'))=1$, for all $i\in\{1,\ldots, n_1-m_1-1\}$.

In the following sum, $\textbf{a}$ runs over all the dimension vectors of $N_1$:

\[
\begin{split}
\mathcal{C}_{\Lambda}(N_1')&=\textbf{x}^{[n_1-1]+[m_1-1]-[n]}(\sum_{\textbf{a}}
{\chi(\Grass_{\textbf{a}}(N'_1))\textbf{x}^{C_Q\textbf{a}}}+\sum_{j=n-n_1+2}^{n-m_1+1}
{\textbf{x}^{C_Q\textbf{e}_j}})\\
&= \textbf{x}^{[n_1-1]+[m_1-1]-[n]}\sum_{\textbf{a}}
{\chi(\Grass_{\textbf{a}}(N_1))\textbf{x}^{C_Q\textbf{a}}} + 
\textbf{x}^{[n_1-1]+[m_1-1]-[n]}\sum_{j=n-n_1+2}^{n-m_1+1}
{\textbf{x}^{C_Q\textbf{e}_j}}\\
&=\mathcal{C}_{\Lambda}(N_1)+\textbf{x}^{[n_1-1]+[m_1-1]-[n]}\sum_{j=n-n_1+2}^{n-m_1+1}
{\textbf{x}^{-[n-i,n-i+1]+[n]}}\\
&=\mathcal{C}_{\Lambda}(N_1)+ \textbf{x}^{[n_1-1]+[m_1-1]-[n]}\sum_{i=0}^{n_1-m_1-1}
{\textbf{x}^{-[n_1-i-2, n_1-i-1]+[n]}}\\
&=\mathcal{C}_{\Lambda}(N_1)+ \textbf{x}^{[n_1-1]+[m_1-1]}\sum_{i=0}^{n_1-m_1-1}
{\textbf{x}^{-[n_1-i-2, n_1-i-1]}}\\
&=\mathcal{C}_{\Lambda}(N_1)+ 
\textbf{x}^{-[n_1-2]+[m_1-1]}\sum_{i=0}^{n_1-m_1-1}{\textbf{x}^{-[n_1-i-2,
n_1-i-1]+[n_1-1]+[n_1-2]}}\\
&=\mathcal{C}_{\Lambda}(N_1)+ \mathcal{C}_{\Lambda}(M).
\end{split}
\]

The first equality is completed.
Now we verify the second equality, recall the notation before this proposition. Set $N_2=N(W_{[m_2,n-1]})$ and $N_2'=N(W_{[m_2,n]})$. From 
definitions we get 
 that  $g_{\Lambda}(N_2)=[m_2-1]-[n-1]$  and  $g_{\Lambda}(N_2')=[m_2-1]$.

For $N_2$: the dimension vector are given by $\textbf{f}_0=\textbf{0}$ and 
$\textbf{f}_{j}=[n-j,n-1]$, where $j\in \{1, 2, \ldots,  n-m_2\}$.  Therefore 
$C_Q\textbf{f}_0=\textbf{f}_0$ and $C_Q\textbf{f}_j= -[n-j-1,n-j]+[n-1]+[n]$.
In this case we know that $\chi(\Grass_{\textbf{f}_j}(N_2))=1$ for all $j\in[0,n-m_2]$.
Then
\[
\mathcal{C}_{\Lambda}(N_2)=\textbf{x}^{[m_2-1]-[n-1]}\sum_{j=0}^{n-m_2}{\textbf{x}^{-[n-j-1,n-j]+[n-1
,n]}}
\]

For $N_2'$: the dimension vector are given by $\textbf{e}_0=\textbf{0}$ and 
$\textbf{e}_i=[n-(i-1),n]$, for $i\in \{1, 2, \ldots, n-m_2+1 \}$. From this, we get that 
$C_Q\textbf{e}_0=\textbf{e}_0$ and $C_Q\textbf{e}_i= -[n-i,n-i+1]+[n]$ where $i\in \{1, 2, \ldots, 
n-m_2+1 \}$.
Then
\[
\begin{split}
\mathcal{C}_{\Lambda}(N_2')&=\textbf{x}^{[m_2-1]}\sum_{i=0}^{n-m_2+1}{\textbf{x}^{-[n-i,
n-i+1]+[n]}}\\
&=\textbf{x}^{[m_2-1]}(1 + \sum_{i=1}^{n-m_2+1}{\textbf{x}^{-[n-i, n-i+1]+[n]}})\\
&=\textbf{x}^{[m_2-1]}(1 + \sum_{j=0}^{n-m_2}{\textbf{x}^{-[n-j-1, n-j]+[n]}})\\
&=\textbf{x}^{[m_2-1]} + \textbf{x}^{[m_2-1]} \sum_{j=0}^{n-m_2}{\textbf{x}^{-[n-j-1, n-j]+[n]}}\\
&=\textbf{x}^{[m_2-1]} + \textbf{x}^{[m_2-1]-[n-1]} \sum_{j=0}^{n-m_2}{\textbf{x}^{-[n-j-1, n-j]+ 
[n-1]+[n]}}\\
&=\textbf{x}^{[m_2-1]} + \mathcal{C}_{\Lambda}(N_2)\\
&=\mathcal{C}_{\Lambda}(\mathcal{S}_{m_2-1}^-)+ \mathcal{C}_{\Lambda}(N_2)
\end{split}
\]
The second equality is completed.
\end{proof}
\section{The Caldero-Chapoton algebra for an arbitrary triangulation}\label{Section13}

In this section we will extend the main results of the previous section. Let $\tau$ be 
any triangulation of $\Sigma_{n}$ and let $\Lambda(\tau)$ be the  algebra associated to 
$\tau$. We can find a triangulation $T$ of $\widetilde{\Sigma}_n$ such that $G\cdot T=
\tau$. By Lemma \ref{PushG} we have that the push-down functor $\pi_{*}:\Lambda(T)
\module \rightarrow \Lambda(\tau)\module$ is a $G$-precovering. Recall that $G=
\mathbb{Z}_3$ acts on $\widetilde{\Sigma}_n$ by an appropriate rotation. In this 
section we are going to prove that $\pi_{*}$ is a $G$-covering. We are going to use 
this to characterize the $E_{\Lambda(\tau)}$-rigid representations.

We are following the notation of \cite[Section 5]{PS-17}. For $\tau=\{t_1,t_2,\ldots , 
t_n\}$ we write, unless we say something else, the triangulation $T$ according to its 
orbits, namely $T=\{t_{1,1}, t_{1,2}, t_{1,3}, \ldots, t_{n,1}, t_{n,2}, t_{n,3}\}$.
If we denote by $x_{i,j}$ the initial cluster variable associated with the arc $t_{i,j}
$, then the initial cluster will be $\textbf{x}_0=(x_{1,1}, x_{1,2}, x_{1,3}, \ldots, 
x_{n,1}, x_{n,2}, x_{n,3})$. If we associated the variable $z_i$ to the arc $t_i$, then 
we obtain a morphism of algebras $\pi: \mathbb{C}[x_{i,j}^{\pm}]\rightarrow \mathbb{C}
[z_i^{\pm}]$, given by $\pi(x_{i,j})=z_i$ for $i=1,\ldots, n$ and $j=1,2, 3$.
The action of $\mathbb{Z}_3$ on $T$ allow us to define the following function
\[
\pi: \mathbb{N}^{3n}\rightarrow \mathbb{N}^{n}, \ \ \ \pi((a_{1,1}, a_{1,2}, a_{1,3}, 
\ldots, a_{n,1}, a_{n,2}, a_{n,3})^t)_i=a_{i,1}+a_{i,2}+a_{i,3}.
\]
We hope that the reader is not confused with our unfortunately choice of $\pi$ for 
different maps. 
For us the arc $t_n\in \tau$ will denote the pending of the triangulation  and $t_{n,
1}, t_{n,2}, t_{n,3}$ are the three sides of the triangle invariant under the action of 
$\mathbb{Z}_3$ on $\widetilde{\Sigma}_n$. Therefore the matrix $C_{Q(T)}$ has a 
decomposition in blocks of size  $3\times 3$. Recall that by definition $C_{Q(T)}$ is 
skew-symmetric, moreover  $C_{Q(T)}$ is skew-symmetric by blocks and every block is a 
multiple of the identity of size 3, except for the block $n,n$ that corresponds to the 
adjacency of the $3$-cycle of $Q(T)$ with vertices  $t_{n,1}, t_{n,2}, t_{n,3}$.

\begin{lemma}\label{Sec13Lem1}
The push down functor $\pi_{*}:\Lambda(T)\module \rightarrow \Lambda(\tau)\module$ is a 
Galois $G$-covering.
\end{lemma}
\begin{proof}
By Lemma \ref{LemBL1} we only need to prove that $\pi_{*}$ is dense. Well, by 
Proposition  \ref{FinitedimPro} and Proposition \ref{PropGentAlg} we know that $
\Lambda(\tau)$ is  a finite-dimensional  gentle algebra. From Theorem \ref{BRTheorem} 
we have that the strings parametrize the indecomposable modules of $\Lambda(\tau)$. 
Since we work in Krull-Schmidt categories we need to prove that the Galois $G$-covering 
$\pi: \Lambda(T)\rightarrow \Lambda(\tau)$ induces a surjective function between the 
set of all string of those algebras and  that $\pi_*(N(\tilde{W}))\cong N(W)$ where  $
\tilde{W}$ is a string of $\Lambda(T)$ such that $G\cdot \tilde{W}=W$, recall that 
$N(W)$ denotes the string module associated to $W$. The last fact  follows from 
definitions.

Suppose the pending arc of $\tau$ is based at $v_i$. Assume $W=W_2\varepsilon^{k(W)}
W_1$ is a string for $\Lambda(\tau)$ with $W_i$ a string without the letter $
\varepsilon$ for $i=1,2$ and $k(W)\in\{-1, 0,+1\}$. Recall that $\varepsilon$ is the 
loop based at the pending arc of $\tau$. It is clear that if $W_1$ does not  contain 
the letter $\varepsilon$, then $W_1$ can be lifted to a string with letters contained 
in one of the three fundamental region divided by the dashed blue lines, see Figure 
\ref{Figstring}, say that is contained in the region that contains  to $[u_i, u_{i+n
+1}]$. Note that the final letter of $W_1$ must be $a_1$ or $b_1^{-1}$, see Figure 
\ref{Figstring}. Now, if $k(W)=0$, then $W$ itself can be lifted to a word in that 
region. If $k(W)=1$, we choose $\varepsilon_{3,1}$ and $W_2$ can be lifted to a string 
in the third fundamental region containing $[u_i, u_{i+2(n+1)}]$. If $k(W)=-1$, we put 
the letter $\varepsilon_{2,1}$ and it is clear that $W_2$ can be lifted to a string of 
$\Lambda(T)$ with letter of the second fundamental region containing $[u_{i+n+1},u_{i
+2(n+1)}]$. Therefore the string $W$ can be lifted to one string $\tilde{W}$ of $
\Lambda(T)$. Note that $\tilde{W}$ depends on where we lifted the tail point of $W_1$. 
The proof of the lemma is completed.
\end{proof}

\begin{lemma}\label{DecPushD}
The push down functor $\pi_{*}:\Lambda(T)\module \rightarrow \Lambda(\tau)\module$ 
induces a Galois $G$-covering $\pi_{*}:\decrep(\Lambda(T))\rightarrow 
\decrep(\Lambda(\tau))$.
\end{lemma}
\begin{proof}
Let $R=\mathbb{C}^{Q(T)_0}$ be the \emph{vertex span} of $\Lambda(T)$. We can see that 
$R_G=\mathbb{C}^{Q(\tau)_0}$ is the vertex span of $\Lambda(\tau)$. With this notation 
it is clear that a decorated representation $(M,V)$ is a pair where $M\in \Lambda(T)
\module$ and $V\in R\module$. For $V\in R\module$ we can define $\pi_* (V)\in R_G
\module$ as in Remark \ref{DefPush}. We put $\pi_*(M,V)=(\pi_*(M), \pi_*(V))$ and the 
lemma follows from the fact that $\Hom_{\decrep(\Lambda(T))}((M,V), (N,W))=
\Hom_{\Lambda(T)}(M,N)\bigoplus\Hom_{R}(V,W)$.
\end{proof}

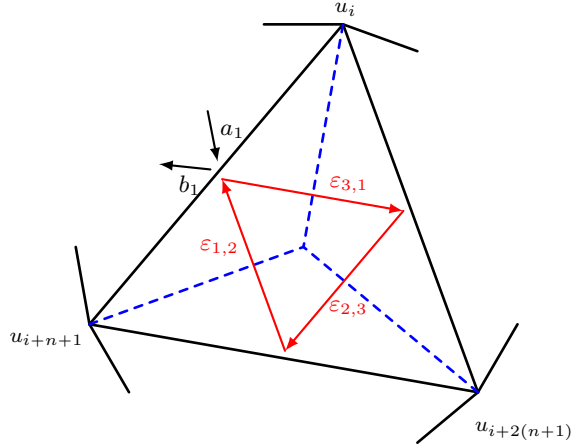
\begin{figure}[ht]
\centering
\begin{tikzpicture}[scale=2,cap=round,>=latex]
\begin{scope}
\node at (0,0) {$\cdot$};
\node at (80:1.6) {$u_i$};
\node at (200:1.8) {$u_{i+n+1}$};
\node at (320:1.9) {$u_{i+2(n+1)}$}; 
  \newdimen\R
  \R=1.5cm
  \coordinate (center) at (0,0);
 \draw (0:\R);
      \draw[line width=1pt] (60:\R)--(80:\R);
      \draw[line width=1pt] (80:\R)--(100:\R);
	  \draw[line width=1pt] (200:\R)--(180:\R);
      \draw[line width=1pt] (200:\R)--(220:\R);
      \draw[line width=1pt] (300:\R)--(320:\R);
	  \draw[line width=1pt] (320:\R)--(340:\R);	
	  \draw[line width=1pt] (320:\R)--(200:\R);
	  \draw[line width=1pt] (200:\R)--(80:\R);
	  \draw[line width=1pt] (320:\R)--(80:\R);	
	  \draw[line width=1pt, blue, dashed] (320:\R)--(0,0);	
	  \draw[line width=1pt, blue, dashed] (80:\R)--(0,0);	
	  \draw[line width=1pt, blue, dashed] (200:\R)--(0,0);
	  \draw[->, line width=0.8pt, red] (140:0.7)--node[pos=0.7, above]{\small{$
	  \varepsilon_{3,1}$}}(20:0.7);
	  \draw[->, line width=0.8pt, red] (20:0.7)--node[pos=0.7, right]{\small{$
	  \varepsilon_{2,3}$}}(260:0.7);
	  \draw[->, line width=0.8pt, red] (260:0.7)--node[pos=0.6, left]{\small{$
	  \varepsilon_{1,2}$}}(140:0.7);
	  \draw[->, line width=0.8pt] (125:1.1)--node[pos=0.4, right]{\small{$a_1$}}
	  (135:0.8);
	  \draw[->, line width=0.8pt] (140:0.8)--node[pos=0.4, below] {\small{$b_1$}}
	  (150:1.1);
\end{scope}
\end{tikzpicture}
\caption{Fundamental regions in $\widetilde{\Sigma}_n$. The sub-index in the red arrows 
indicates which fundamental regions they connect.}	
\label{Figstring}	
\end{figure}

Now, we can extend Lemma \ref{LemmaTau} to other triangulations. Recall that we denote 
by $r^+(j)$ (resp. $r^-(j)$) the arc of $\Sigma_n$ that we obtain by rotating $j$ in 
counterclockwise (resp. clockwise) for an angle of $2\pi/(n+1)$.

\begin{corollary}\label{CoroTau+}
Let $\sigma$ be a triangulation of $\Sigma_n$ and let $j$ be an arc of $\Sigma_n$ not 
in $\sigma$.
\begin{itemize}
\item[(a)] Assume $M(j,\sigma)$ is not projective, then $\tau(M(j,\sigma))=M(r^{+}(j),
\sigma)$.
\item[(b)] Assume $M(j,\sigma)$ is not injective, then $\tau^{-}(M(j,\sigma))=M(r^{-}
(j),\sigma)$.
\end{itemize}
\end{corollary} 
\begin{proof}
This is a consequence of  the $A_n$ case from \cite[Theorem 2.13]{CCS-04} and Theorem 
\ref{BLTheorem}. Let $\widetilde{j}$ be a lifting of $j$ in $\widetilde{\Sigma}_n$ and 
let $\widetilde{\sigma}$ be the lifting of $\sigma$. 

(a). From \cite[Theorem 2.13]{CCS-04} we see that  $\tau(M(\widetilde{j},
\widetilde{\sigma}))=M(r^{+}(\tilde{j}),\widetilde{\sigma})$ and by applying $\pi_*$, 
from  \cite[Theorem 4.7 (1)]{BL-14}, we get what we want, $\tau(M(j,\sigma))=M(r^{+}
(j),\sigma)$. The proof of (b) is similar. 
\end{proof}

The reader can compare the next proposition and Proposition \cite[Proposition 7.15]
{PS-17}.

\begin{proposition}\label{PiErigid}
Let $\tau$ be a triangulation of $\Sigma_n$ and let $\Lambda(\tau)$ be the algebra 
associated to $\tau$. Suppose $T$ is the triangulation of  $\widetilde{\Sigma}_n$ such 
that $G\cdot T =\tau$. If $M$ is an indecomposable representation of $\Lambda\module$, 
then $\pi_*(M)$ is $E_{\Lambda(\tau)}$-rigid if and only if  $E_{\Lambda(T)}(M, g\cdot 
M)=0$ for any $g\in G$.
\end{proposition}

\begin{proof}
The proposition follows from Proposition \ref{PropoTrunc} and the following equalities
\[
\begin{split}
\dimension\Hom_{\Lambda(\tau)}(\tau^{-}(\pi_*(M)),\pi_*(M))&=\dimension
\Hom_{\Lambda(\tau)}(\pi_*(\tau^{-}(M)),\pi_*(M))\\
&=\dimension\bigoplus_{g\in G}\Hom_{\Lambda(T)}(g\cdot\tau^{-}(M),M)\\
&=\sum_{g\in G}{\dimension\Hom_{\Lambda(T)}(\tau^{-}(g\cdot M),M)}.
\end{split}
\]
Since $\pi_*$ is a Galois $G$-covering, Lemma \ref{Sec13Lem1}, the first equality 
follows from Theorem \ref{BLTheorem}, see \cite[Theorem 4.7 (1)]{BL-14}. The second 
line follows from definition of $G$-precovering and the third line is a consequence 
that $g\cdot -$ is an isomorphism of categories. This conclude the proof.
\end{proof}

\begin{lemma}\label{Lem13.1}
Let $N$ be an indecomposable representation of $\Lambda(\tau)$. Then $N$ is 
$E_{\Lambda}$-rigid if and only if $N=M(j, \tau)$ for some arc $j$ of $\Sigma_n$. 
\end{lemma}
\begin{proof}
Let $M$ be a representation of $\Lambda(\tau)$ such that $\pi_*(M)=N$, recall that $
\pi_*$ is dense. From \cite[Corollary 2.12]{CCS-04} we know that there is a bijection 
between the indecomposable representations of $\Lambda(T)$ and the diagonals of $
\widetilde{\Sigma}_n$ not in $T$. Then $M=M(\widetilde{j},T)$ for some arc $
\widetilde{j}$ of $\widetilde{\Sigma}_n$. By Proposition \ref{PiErigid} we know that $N
$ is $E_{\Lambda(T)}$-rigid if and only if $E_{\Lambda(T)}(M, g\cdot M)=0$. We need to 
analyze $\dimension\Hom_{\Lambda(T)}(\tau^{-1}(g\cdot M), M)$. Suppose $\widetilde{j}
=[u_l, u_{l+k}]$ for some $l\in[0, 3n+1]$ , then $g\cdot M=M([u_{l-(n+1)}, u_{l+k-(n
+1)}], T) $ and $\tau^{-1}(g\cdot M)=M([u_{l-n-2}, u_{l+k-n-2}],T)$. This means, in 
particular, that $g\cdot M$ is an arc representation. 
By \cite[Lemma 2.5]{CCS-04},  the Auslander-Reiten formulas  and \cite[Remark 2.15]
{CCS-04}  if $E_{\Lambda(T)}(M, g\cdot M)=0$ for any $g\in \mathbb{Z}_3$, then  we can 
conclude that $\widetilde{j}$ has to be an admissible arc of $\widetilde{\Sigma}_n$, 
therefore $G\cdot \widetilde{j}=j$ is an arc of $\Sigma_n$ and $N=
\pi_*(M(\widetilde{j}, T))=M(j, \tau)$. The proof of the lemma is completed.
\end{proof}

\begin{remark}
Let $\mathcal{F}:\mathcal{A}\rightarrow \mathcal{B}$ be a Galois $G$-precovering. Then 
$\mathcal{F}$ is faithful, see \cite[Lemma 2.6 (2)]{BL-14}.
\end{remark}

\begin{lemma}\label{LemInj}
Let $I_{l}$ be the indecomposable injective at $l\in Q(T)$. Then $\pi_{*}(I_l)\cong 
I_{G\cdot l}$, where $I_{G\cdot l}$ is the indecomposable injective at $G\cdot l$.
\end{lemma}
\begin{proof}
Let $D=\Hom_{\mathbb{C}}(-, \mathbb{C})$ be the standar $\mathbb{C}$-dual functor. 
Remember that $I_l(k)=D\Hom (e_k, e_l)$ and $\pi_{*}(I_l)({G\cdot k})=\bigoplus_{g\in 
\mathbb{Z}_3}{D\Hom(g\cdot e_k, e_l)}$. By definition there exist  an isomorphism $
\pi_{*}^{k,l}:\bigoplus_{g\in \mathbb{Z}_3}{\Hom(g\cdot e_k, e_l)}\rightarrow \Hom(G
\cdot e_k, G\cdot e_l)$. In other words, for any $k\in Q(T)_0$ we get an isomorphism $
\overline{\pi}_{*}^{k,l}:I_{G\cdot l}(G\cdot k)\rightarrow \pi_{*}(I_l)({G\cdot k})$.
 Let $\alpha:k_1\rightarrow k_2\in Q(T)$ be an arrow. It can be shown that  $
 \overline{\pi}_{*}^{k_2,l}\circ (I_{G\cdot l})_{G\cdot \alpha}=\pi_{*}(I_l)_{G\cdot 
 \alpha}\circ \overline{\pi}_{*}^{k_1,l}$.
\end{proof}

\begin{lemma}\label{PiInj}
If $f:M\rightarrow N$ is  injective in $\Lambda(T)\module$, then $\pi_{*}(f):\pi_{*}(M)
\rightarrow \pi_{*}(N)$ is injective in $\Lambda(\tau)\module$. 
\end{lemma}
\begin{proof}
Remember that $f=(f_i)_{i\in Q(T)_{0}}$ is a $3n$-tuple of linear transformations and  
by hypothesis we have that $\dimension \rnk f_i=\dimension  M_i$. The lemma follows 
from Remark \ref{DefPush} and that $\dimension \rnk \pi_{*}(f)_{G\cdot i}=\sum_{g\in
\mathbb{Z}_3}{\dimension\rnk f_{g\cdot i}}$.
\end{proof}

\begin{lemma}\label{LeMIP}
Let $\widetilde{j}$ be an arc of $\widetilde{\Sigma}_n$. For  $M:=M(\widetilde{j})\in 
\Lambda(T)\module$, let 
\[
0\rightarrow M\rightarrow I_{0}\rightarrow I_1
\]
be a minimal injective presentation of  $M$. Then 
\[
0\rightarrow \pi_{*}(M)\rightarrow \pi_{*}(I_{0})\rightarrow \pi_{*}(I_1)
\]
is a minimal injective presentation of $\pi_{*}(M(\alpha))$.
\end{lemma}
\begin{proof}
The lemma follows from the fact that $\pi$ is dense, Lemma \ref{PiInj} and Lemma 
\ref{LemInj}.
\end{proof}

We shall discuss about the Caldero-Chapoton algebra associated to different 
triangulation. Let $T_1$ and $T_2$ be triangulations of $\widetilde{\Sigma}_n$. Denote 
by $\mathcal{A}_i:=\mathcal{A}_{\Lambda(T_i)}$ the Caldero-Chapoton algebra 
corresponding for $i=1,2$. Let $\mathcal{D}_i$ the $\mathbb{C}$-subalgebra of $
\mathcal{A}_i$ generated by $\mathcal{C}_{\Lambda(T_i)}(M(\widetilde{j}, T_i))$ for any  
admissible arc $\widetilde{j}$ of $\widetilde{\Sigma}_n$ for $i=1,2$.
\begin{lemma}\label{AdmLemma}
With the above notation $\mathcal{D}_1$ and $\mathcal{D}_{2}$ are isomorphic as $
\mathbb{C}$-algebras.
\end{lemma}
\begin{proof}
Let $\varphi:\mathcal{A}_1\rightarrow \mathcal{A}_2$ be the corresponding isomorphism 
of cluster algebras. This isomorphism sends a cluster variable to the corresponding 
Laurent Polynomial in the initial seed  associated to $T_2$, i.e $x_{\widetilde{i}}
\mapsto\mathcal{C}_{\Lambda(T_2)}(M(\tilde{i}, T_2))$ for an arc $\widetilde{i}$ in 
$T_1$. Suppose that we can get $T_2$ from $T_1$ by the flip sequence $(s_l, s_{l-1}, 
\ldots, s_1)$. By \cite{DWZ-10} we conclude that $\mathcal{C}_{\Lambda(T_1)}
(M(\widetilde{j}, T_1))=\mathcal{C}_{\Lambda(T_2)}(\mu_{s_1}\mu_{s_2}\cdots \mu_{s_l}
(M(\widetilde{j}, T_2)))$, then
\[\varphi(\mathcal{C}_{\Lambda(T_1)}(M(\widetilde{j}, T_1)))=\mathcal{C}_{\Lambda(T_2)}
(M(\widetilde{j}, T_2)).\] 
That means that $\varphi$ can be restricted to $\mathcal{D}_1$ and we obtain the 
isomorphism desired. The lemma is completed.
\end{proof}

\begin{remark}\label{RemEq1}
Let  $W$ be  a string of $\Sigma_n$. Consider a lifting $\widetilde{W}$ of $W$ on $
\widetilde{\Sigma}_n$.  If $\textbf{f}$ is a dimension vector of some 
sub-representation of $N(\widetilde{W})$, then $\pi(\textbf{y}^{C_{Q(T)}\cdot 
\textbf{f}})=
\textbf{z}^{C_{Q(\tau)}\cdot \pi(\textbf{f})}$.  Indeed, we need to prove that 
\begin{equation}\label{MatrixQ} 
\pi(y_{i,1}^{C_{Q(T)_{i,1}}\cdot \textbf{f}}y_{i,2}^{C_{Q(T)_{i,2}}\cdot \textbf{f}}
y_{i,3}^{C_{Q(T)_{i,3}}\cdot \textbf{f}})=z_{i}^{C_{Q(\tau)_i}\cdot \pi(\textbf{f})}
\end{equation}
for any $i\in[1,n]$, here $C_{Q(T)_{i,j}}$ denote the $(i,j)$-th row of the matrix 
$C_{Q(T)}$.  For $i<n$ the calculation is straightforward and we will concentrate in 
the case when $i=n$. Assume $i=n$, we orient  the arcs $t_{n,1}, t_{n,2}, t_{n,3}$ in 
counter clockwise on $\widetilde{\Sigma}_n$. This orientation determines the $(n,n)$ 
block of  $C_{Q(T)}$. In order to obtain \eqref{MatrixQ} we need that 
\begin{equation}\label{VectorE}
(f_{n,2}-f_{n,1})-(f_{n,3}-f_{n,1})+ (f_{n,3}-f_{n,2})=0.
\end{equation}
The observation is that \eqref{VectorE} is true in case $\textbf{f}$ is the dimension 
vector of an indecomposable representation of $\Lambda(T)$.
\end{remark}

\begin{lemma}\label{ProChar}
Let $\tau$ be a triangulation of $\Sigma_n$ and let $\Lambda(\tau)$ be the algebra 
associated to $\tau$. Let $W$ be a string on $Q(\tau)$ and let $\widetilde{W}$ be a 
lifting of $W$ in $Q(T)$, where $T$ is the triangulation of  $\widetilde{\Sigma}_n$ 
such that $G\cdot T =\tau$. Then for any dimension vector $\textbf{e}$ of $N(W)$ we get
\[
\sum_{\textbf{f}\colon \pi(\textbf{f})=\textbf{e}}{\chi(\Grass_{\textbf{f}}
(N(\widetilde{W})))}=\chi(\Grass_{\textbf{e}}(N(W))).
\]
\end{lemma}
\begin{proof}
 First, we are going to introduce some notation.  We write $j:=j(W)$ for the arc 
 determined by $W$, note that this arc can have self intersections. We will denote 
 $M(j):=N(W)$. Suppose $j$ connects $v_k$ and $v_l$ with $k\leq l$. So, we orient $j$ 
 from $v_k$ to $v_l$. 
Let $x_{p_1}$ be  the first intersection point between $j$ and the pending arc $p(\tau)
$ of $\tau$. Let $x_{p_2}$ be the second intersection point between $j$ and $p(\tau)$.
We divide the arc $j$ in three parts;
\begin{itemize}
\item The \emph{top} part $j_{1,0}=[v_k, x_{p_1}]$.
\item The \emph{center} part $j_{1,1}=[x_{p_1}, x_{p_2}]$.
\item The \emph{buttom} part $j_{0,1}=[x_{p_2}, v_l]$.
\end{itemize}
Let $\upp_j=\{x_i\colon x_i=j_{1,0}\cap i \mbox{ with } i\in \tau  \}$ be the 
\emph{upper} points of $j$. Let  $\bpp_j=\{y_i\colon y_i=j_{0,1}\cap i \mbox{ with } i
\in \tau  \}$ be the \emph{below} points of $j$. For convention if $j$ does not cross 
$p(\tau)$, then $x_{p_1}=v_l$, $x_{p_2}=v_k$ and $\bpp_j=\upp_j$, see Figure 
\ref{FigChar1}.

Let $L\in  \Grass_{\textbf{e}}(M(j))$ be a sub-representation of $M(j)$ with dimension 
vector $\textbf{e}$. We are going to define an action of $\mathbb{C}^*$ on $
\Grass_{\textbf{e}}(M(j))$. For $t\in \mathbb{C}^*$ we define $t\cdot L$ as follows:
\[ 
(t\cdot L)_k=
\left\{
\begin{array}{ll}
 {\scriptstyle \left( \begin{smallmatrix}
ta \\
b \end{smallmatrix} \right)}\cdot\mathbb{C} \mbox{  \ \  if } L_k={\scriptstyle \left( \begin{smallmatrix}
a \\
b \end{smallmatrix} \right)}\cdot\mathbb{C} \mbox{ and  } \dimension M(j)_k=2,\\
L_k \mbox{ \ \   in other wise}.\\
\end{array}
\right.
\]
Indeed, this define an action of $\mathbb{C}^*$ on $\Grass_{\textbf{e}}(M(j))$. By 
Lemma \ref{LemmaBi} we know that $\chi(\Grass_{\textbf{e}}(M(j))^{\mathbb{C}^*})=
\chi(\Grass_{\textbf{e}}(M(j)))$. In this case $\Grass_{\textbf{e}}(M(j))^{\mathbb{C}
^*}$ is a finite set, then the Euler characteristic is its cardinality. 
 Denote by $Q(j)$ the full  \emph{sub-quiver} of $Q(\tau)$ defined by $j$. We consider 
 the lifting $\widetilde{j}$ of $j$ on $\widetilde{\Sigma}_n$. 
  On $\widetilde{\Sigma}_n$ we can also define the corresponding top, center and bottom 
  part of $\widetilde{j}$.

Note that if the arc $j$ does not cross $p(\tau)$, then $\widetilde{j}$ is completely 
contained in one fundamental region of the action and $\pi$ acts as a bijection between 
dimension vectors of sub-representations of $M(\widetilde{j})$ and dimension vector of 
sub-representations of $M(j)$. In other words, there exist an unique $\textbf{f}$ such 
that $\pi(\textbf{f})=\textbf{e}$. Therefore $
\chi(\Grass_{\textbf{f}}(N(\widetilde{W})))=\chi(\Grass_{\textbf{e}}(N(W)))$.

Let  $L\in \Grass_{\textbf{e}}(M(j))^{\mathbb{C}^*}$ be a sub-representation of $M(j)
 $. It is clear that $L$ define a walk in $Q(j)$ and also an unique  subset $F(L)$ of 
 $\bpp_j\cup\upp_j$. Indeed, the action we have defined allows to identify every 
 subspace $L_i$ with points of $F(L)\subset \bpp_j\cup\upp_j$ in the following way. If 
 $L_i$ is generated by $(1,0)^t$, then we take  the corresponding upper point of $j$. 
 If $L_i$ is generated by $(0,1)^t$, then we take the corresponding below point of $j$. 
 In case $L_i$ is 2 dimensional, then we  take both, the upper and below point of $j$.  
 It is clear that $F(L)$ determines an unique vector $\textbf{f}_L$ of 
 $M(\widetilde{j})$ such that $\pi(\textbf{f}_L)=\textbf{e}$. This implies that 
\[
\sum_{\textbf{f}\colon \pi(\textbf{f})=\textbf{e}}{\chi(\Grass_{\textbf{f}}
(N(\widetilde{W})))}\geq\chi(\Grass_{\textbf{e}}(N(W))).
\]
For any vector $\textbf{f}$ of some sub-representation of $N(\widetilde{W})$ with $
\pi(\textbf{f})=\textbf{e}$ we can find a subset $D_{\textbf{f}}$ of $\bpp_j\cup\upp_j$ 
corresponding to a sub-representation of $N(W)$, namely we obtain the image under $
\pi_*$ of the representation given by $\textbf{f}$. We make this by cuting the arc $
\widetilde{j}$ on $\widetilde{\Sigma}_n$ along the boundary of the fundamental regions 
that it crosses and gluing that parts on $\Sigma_n$ according to the orientation we 
fixed on $\widetilde{j}$. This subset corresponds to a sub-representation 
$L_{\textbf{f}}$ of $N(W)$. By the definition of the action we can conclude that 
$L_{\textbf{f}}\in \Grass_{\textbf{e}}(M(j))^{\mathbb{C}^*}$.
This shows the another inequality. Hence the lemma is completed.
\end{proof}

\begin{figure}[h]
\includegraphics[scale=0.8]{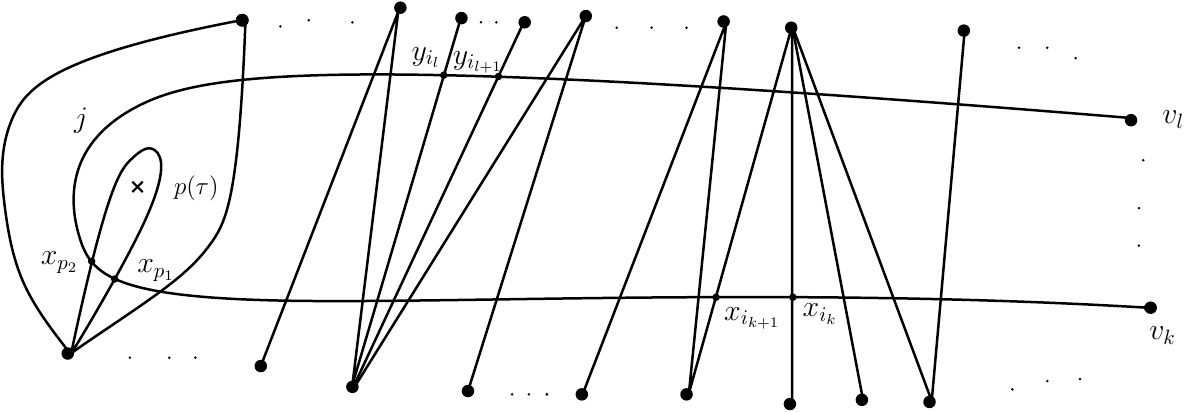}
\caption{An arc $j$ on $\Sigma_n$ with respect to a triangulation $\tau$.}
\label{FigChar1}
\end{figure}

The next proposition follows from Lemma \ref{LeMIP}, Remark \ref{RemEq1} and Lemma 
\ref{ProChar}. The reader can compare this result with the discussion of  \cite[Remark 
7.9]{PS-17}.
\begin{proposition}\label{MorphiP}
Let $\tau$ be a triangulation of $\Sigma_n$ and let $\Lambda(\tau)$ be the algebra 
associated to $\tau$. Assume $T$ is the triangulation of  $\widetilde{\Sigma}_n$ such 
that $G\cdot T =\tau$. Then for any string $W$ of $\Lambda(\tau)$ the following 
equation is true
\[
\pi(\mathcal{C}_{\Lambda(T)}(N(\widetilde{W})))=\mathcal{C}_{\Lambda(\tau)}(N(W)).
\]
\end{proposition}
\begin{proof}
By expanding  $\mathcal{C}_{\Lambda(T)}(N(\widetilde{W}))$ and rearranging its monomials 
as in Lemma \ref{ProChar}  we are able to apply Lemma \ref{LeMIP} and Remark 
\ref{RemEq1} to any monomial. Note that from Lemma \ref{LeMIP} we have that $
\pi(g_{\Lambda(T)}(N(\widetilde{W})))=g_{\Lambda(\tau)}(N(W))$. The proposition is 
completed.
\end{proof}
\section{The Caldero-Chapoton algebra is a generalized cluster algebra}\label{Sect15}

In this section we will state and prove our main result. Before that, we need some 
previous propositions. 

\begin{proposition}
Let $\tau$ be a triangulation of $\Sigma_n$ and let $\Lambda(\tau)$ be the algebra 
associated to $\tau$. Assume $T$ is the triangulation of  $\widetilde{\Sigma}_n$ such 
that $G\cdot T =\tau$ and  $j\notin \tau$. Then the $G_{\textbf{d}}$-orbit $\mathcal{O}
(M(j))$ is open in $\rep_{\textbf{d}}(\Lambda)$. 
\end{proposition}
\begin{proof}
By \cite[1.7 Corollary 3]{DP-95} we need to prove that  for any arc $j$ of $\Sigma_n$ 
we have that $\Ext_{\Lambda(\tau)}(M(j), M(j))=0$. This is clear by the 
Auslander-Reiten formulas since $E(M(j))=0=\dimension\Hom(\tau^{-}(M(j)), M(j))$.
\end{proof}
If we denote by $Z(j)$ the irreducible component containing $M(j)$, then $\mathcal{O}
(M(j))$ is dense in $Z(j)$. Therefore $M(j)$ is generic and all its homological data is 
generic in $Z(j)$. We can take generic versions of the results of the above section as 
in Section \ref{Section6}.

Repeating the arguments of Proposition \ref{PiErigid} and applying what we know for the 
$A_n$ case, for instance see \cite[Remark 2.15]{CCS-04}, we have
\begin{proposition}
Given a triangulation $\sigma$ of $\Sigma_n$ and two arcs $j_1,j_2\in \sigma$ we have 
$E_{\Lambda(\tau)}(Z_{j_1},Z_{j_2})=0$.
\end{proposition}

The next proposition shows that the $E$-rigid representations generate the 
corresponding Caldero-Chapoton algebra.
\begin{proposition}\label{PropoGenerate}
The set 
\[
\{\mathcal{C}_{\Lambda(\tau)}(Z)\colon Z\in \firr(\Lambda), E_{\Lambda(\tau)}(Z)=0\}
\]
generates the Caldero-Chapoton algebra $\mathcal{A}_{\Lambda(\tau)}$ as 
$\mathbb{C}$-algebra.
\end{proposition}
\begin{proof}
As in Proposition \ref{ProGen} we are going to prove that the Caldero-Chapoton 
functions of $E$-rigid representations generate the remaining Caldero-Chapoton 
functions. Let $\widetilde{j}$ be an arc of $\widetilde{\Sigma}_n$  such that it does 
not belong to any triangulation invariant under the action of $\mathbb{Z}_3$. By 
Proposition \ref{PiErigid} from these arcs come all the non-$E$-rigid representations 
of $\Lambda(\tau)$, so what we need to do is to prove the result in this case. In other 
words, we are going to prove that the Caldero-Chapoton function of  $
\pi_*(M(\widetilde{j}))$ can be expressed in terms of the Caldero-Chapoton functions of 
$E$-rigid representations. For $\widetilde{j}$ we construct a quadrilateral in the 
following way; first, we choose an ending point of $\widetilde{j}$, say $u_i$. Then we 
draw the triangle invariant under the $\mathbb{Z}_3$-action incident to $u_i$ with 
sides given by $\widetilde{j}_1$, $\widetilde{j}'$ and $\widetilde{j}_4$. Finally, we 
complete the quadrilateral with the other ending point of $\widetilde{j}$ such that $
\widetilde{j}$ and $\widetilde{j}'$ are the respective  diagonals. We label  the 
remaining sides with $\widetilde{j}_2$ and $\widetilde{j}_3$.  This construction is 
depicted in  Figure \ref{FigPTol}. From \cite{CCh-05,CCS-04}  we have that 
\[
\mathcal{C}_{\Lambda(T)}(M(\widetilde{j}))\mathcal{C}_{\Lambda(T)}(M(\widetilde{j}'))=
\mathcal{C}_{\Lambda(T)}(M(\widetilde{j}_1))\mathcal{C}_{\Lambda(T)}(M(\widetilde{j}
_3))+ \mathcal{C}_{\Lambda(T)}(M(\widetilde{j}_2))\mathcal{C}_{\Lambda(T)}
(M(\widetilde{j}_4)).
\]
Note that  $\widetilde{j}_1$, $\widetilde{j}'$ and $\widetilde{j}_4$ are in the same 
orbit. By applying the algebras homomorphism $\pi$ to the above equation, from 
Proposition \ref{MorphiP}, we obtain 
 \[
\mathcal{C}_{\Lambda(\tau)}(\pi_*(M(\widetilde{j})))\mathcal{C}_{\Lambda(\tau)}
(\pi_*(M(\widetilde{j}')))=\mathcal{C}_{\Lambda(\tau)}(\pi_*(M(\widetilde{j}')))
\mathcal{C}_{\Lambda(\tau)}(\pi_*(M(\widetilde{j}_3)))+ \mathcal{C}_{\Lambda(\tau)}
(\pi_*(M(\widetilde{j}_2)))\mathcal{C}_{\Lambda(\tau)}(\pi_*(M(\widetilde{j}'))).
\]
Since we are in an integral domain, we have the desired relation
\[
\mathcal{C}_{\Lambda(\tau)}(\pi_*(M(\widetilde{j})))=\mathcal{C}_{\Lambda(\tau)}
(\pi_*(M(\widetilde{j}_3)))+\mathcal{C}_{\Lambda(\tau)}(\pi_*(M(\widetilde{j}_2))).
\]
The proposition is completed.
\end{proof}
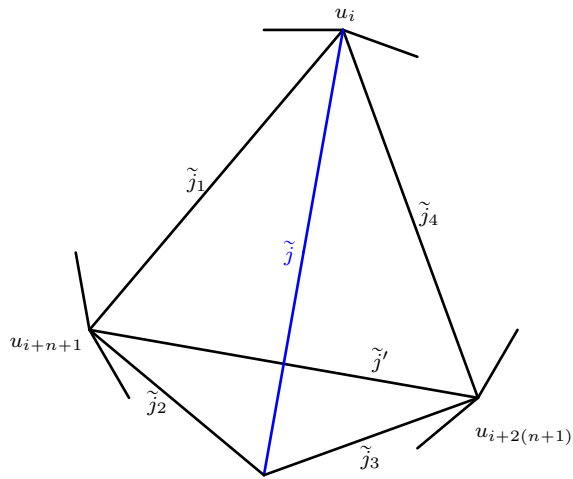
\begin{figure}[ht]
\centering
\begin{tikzpicture}[scale=2,cap=round,>=latex]
\begin{scope}
\node at (0,0) {$\cdot$};
\node at (80:1.6) {$u_i$};
\node at (200:1.8) {$u_{i+n+1}$};
\node at (320:1.9) {$u_{i+2(n+1)}$}; 
  \newdimen\R
  \R=1.5cm
  \coordinate (center) at (0,0);
 \draw (0:\R);
      \draw[line width=1pt] (60:\R)--(80:\R);
      \draw[line width=1pt] (80:\R)--(100:\R);
	  \draw[line width=1pt] (200:\R)--(180:\R);
      \draw[line width=1pt] (200:\R)--(220:\R);
      \draw[line width=1pt] (300:\R)--(320:\R);
	  \draw[line width=1pt] (320:\R)--(340:\R);	
	  \draw[line width=1pt] (320:\R)--node[pos=0.25, above]{$\widetilde{j}'$}(200:\R);
	  \draw[line width=1pt] (200:\R)--node[pos=0.5, left]{$\widetilde{j}_1$}(80:\R);
	  \draw[line width=1pt] (320:\R)--node[pos=0.5, right]{$\widetilde{j}_4$}(80:\R);	
	  \draw[line width=1pt, blue] (80:\R)-- node[pos=0.5, left]{$\widetilde{j}$}(260:
	  \R);
	  \draw[line width=1pt] (200:\R)--node[pos=0.5, left]{$\widetilde{j}_2$}(260:\R);
	  \draw[line width=1pt] (260:\R)--node[pos=0.5, below]{$\widetilde{j}_3$}(320:\R);
\end{scope}
\end{tikzpicture}
\caption{The quadrilateral with $\widetilde{j}$ as diagonal and with two adjacent sides 
of one invariant triangle of $\widetilde{\Sigma}_n$.}
\label{FigPTol}	
\end{figure}

From the above proposition we obtain our main result.

\begin{theorem}\label{MainResult}
For any triangulation $\tau$ of $\Sigma_n$ we have that the Caldero-Chapoton algebra $
\mathcal{A}_{\Lambda(\tau)}$ is isomorphic to the generalized cluster  algebra $
\mathcal{A}(B(\tau_0))$.
\end{theorem}
\begin{proof}
Let $T_1$ and $T_2$ be triangulations of $\widetilde{\Sigma}_n$. Denote by $\mathcal{A}
_i:=\mathcal{A}_{\Lambda(T_i)}$ the Caldero-Chapoton algebra corresponding for $i=1,2$. 
Let $\mathcal{D}_i$ the $\mathbb{C}$-subalgebra of $\mathcal{A}_i$ generated by $
\mathcal{C}_{\Lambda(T_i)}(M(\widetilde{j}, T_i))$ for any  admissible arc $
\widetilde{j}$ of $\widetilde{\Sigma}_n$ for $i=1,2$.
By  Lemma \ref{AdmLemma} we have that $\mathcal{D}_1$ and $\mathcal{D}_2$ are 
isomorphic. The previous Proposition and Proposition \ref{MorphiP} show that the 
Caldero-Chapoton algebra associated to $\tau_i=G\cdot T_i$ is $\pi(\mathcal{D}_i)$ for 
$i=1,2$. We conclude that the Caldero-Chapoton algebras $\mathcal{A}_1$ and $
\mathcal{A}_2$ are isomorphic. Now, from \cite{PS-17} we know that the image of $
\Lambda(\tau_0)$ under $\pi$ is a Chekhov-Shapiro generalized cluster  algebra with 
initial seed $(B(\tau_0), \textbf{d}_{\tau_0})$. Indeed, from \cite[Lemma 5.6]{PS-17} 
and \cite[Lemma 5.7]{PS-17} we know that the exchange polynomial are those of 
Chekhov-Shapiro. The theorem is completed. 
\end{proof}
\section{Example}\label{Section12}

The example in this section is meant to illustrate our main result. It can also be 
considered a complement to \cite[Example 9.4.2]{CLS-15}.
Let $\tau_0$ be the special triangulation of $\Sigma_3$ 
and let $\tau$ be the triangulation of Example \ref{FirExamLam}, see Figure 
\ref{FigExamMR}.

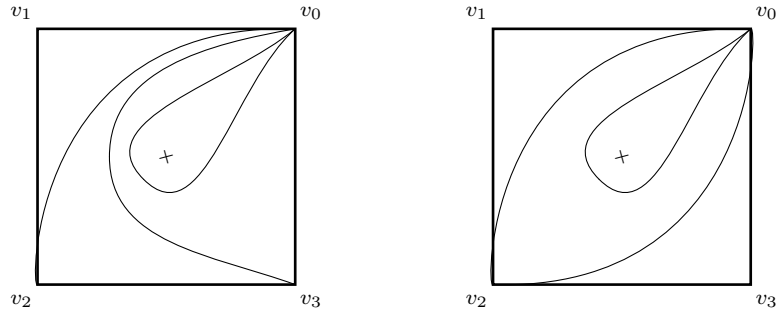
\begin{figure}[ht]
\centering
\begin{tikzpicture}[scale=3,cap=round,>=latex]
\begin{scope}
\node at (0,0) {\rotatebox{60}{$\times$}};
  \newdimen\R
  \R=0.8cm
  \coordinate (center) at (0,0);
 \draw[line width=1.0pt] (45:\R)
     \foreach \x in {45,135,225,315} {-- (\x:\R) }
              -- cycle (315:\R)
              -- cycle (225:\R) 
              -- cycle (135:\R) 
              -- cycle (45:\R);
    \draw (45:\R) to[out=220,in=135] (-0.1,-0.1);
	\draw (-0.1,-0.1) to[out=315,in=222] (45:0.8);
    \draw[line width=0.4pt](45:\R)   .. controls (-0.6, 0.6) and (-0.6,-0.6)  .. 
	    (225:\R);  
	\draw[line width=0.4pt, rotate=-90](135:\R) to[out=280,in=180] (0,-0.25);
	\draw[line width=0.4pt, rotate=-90](0,-0.25) to[out=0,in=250] (45:0.8);    
			\node at (45:0.9) {$v_0$};	
			\node at (135:0.9) {$v_1$};
			\node at (225:0.9) {$v_2$};
			\node at (315:0.9) {$v_3$};
\end{scope}
\begin{scope}[xshift=2cm]	
\node at (0,0) {\rotatebox{60}{$\times$}};
  \newdimen\R
  \R=0.8cm
  \coordinate (center) at (0,0);
 \draw[line width=1.0pt] (45:\R)
     \foreach \x in {45,135,225,315} {-- (\x:\R) }
              -- cycle (315:\R)
              -- cycle (225:\R) 
              -- cycle (135:\R) 
              -- cycle (45:\R);
	    \draw (45:\R) to[out=220,in=135] (-0.1,-0.1);
			\draw (-0.1,-0.1) to[out=315,in=222] (45:0.8);
	    \draw[line width=0.4pt](45:\R)   .. controls (-0.6, 0.6) and (-0.6,-0.6)  .. 
	    (225:\R);	
	    \draw[line width=0.4pt](45:\R) ..controls (0.6,0.6) and (0.6,-0.6) .. (225:\R);
 	\node at (45:0.9) {$v_0$};	
	\node at (135:0.9) {$v_1$};
	\node at (225:0.9) {$v_2$};
	\node at (315:0.9) {$v_3$};	    
\end{scope}
\end{tikzpicture}
\caption{On the left side we can see the special triangulation $\tau_0$ and on the 
right side we have the  triangulation $\tau$ of $\Sigma_3$.}
\label{FigExamMR}
\end{figure}

 It is clear that $\tau$ and $\tau_0$ are related by a flip at one arc. To ease the 
 notation we  set $\Lambda=\Lambda(\tau)$, see Example \ref{FirExamLam}. From Theorem 
 \ref{BRTheorem} we know that the indecomposable $\Lambda$-modules are parametrized by 
 the strings of $\Lambda$. We say that a string $W$ is $E$-\emph{rigid} if its string 
 module $N(W)$ is $E$-rigid. 
There are 12 indecomposable $E$-rigid decorated representations of $\Lambda$ of which 9 
are given by the $E$-rigid strings $1_1$, $1_2$, $\varepsilon$, $a$, $\varepsilon b$, 
$c\varepsilon$, $c\epsilon b$, $b^{-1}\varepsilon b$ and $c\varepsilon c^{-1}$;  and 
the remaining three are the negative simple representations of $\Lambda$. The 
non-$E$-rigid strings are $1_3$, $b$, $c$, $b^{-1}\varepsilon$, $\varepsilon c^{-1}$ and  
$b^{-1}\varepsilon c^{-1}.$

By definition $\mathcal{C}_{\Lambda}(\mathcal{S}_i^{-})=y_i$ for $i=1,2,3$. In Figure 
\ref{FigString} we write the  string module corresponding to every arc of $\Sigma_3$.  
The Caldero-Chapoton functions associated to the 9 $E$-rigid strings of $\Lambda$ are
\begin{figure}[ht]
\centering
\begin{tikzpicture}[scale=1.5,cap=round,>=latex]
\begin{scope}
\node at (0,0) {\rotatebox{60}{$\times$}};
\node at (0,-0.8) {$S_1$};
  \newdimen\R
  \R=0.8cm
  \coordinate (center) at (0,0);
 \draw (45:\R)
     \foreach \x in {45,135,225,315} {-- (\x:\R) }
              -- cycle (315:\R)
              -- cycle (225:\R) 
              -- cycle (135:\R) 
              -- cycle (45:\R);
      \draw (45:\R) to[out=220,in=135] (-0.1,-0.1);
			\draw (-0.1,-0.1) to[out=315,in=222] (45:0.8);
	    \draw[line width=0.4pt](45:\R)   .. controls (-0.6, 0.6) and (-0.6,-0.6)  .. 
	    (225:\R);	
	    \draw[line width=0.4pt](45:\R) ..controls (0.6,0.6) and (0.6,-0.6) .. (225:\R);
			\draw[blue, line width=1pt](135:\R) to[out=280,in=180] (0,-0.25);
			\draw[blue, line width=1pt](0,-0.25) to[out=0,in=250] (45:0.8);
		\node at (0.6,0.7) {$v_0$};
\end{scope}	
\begin{scope}[xshift=2cm]
\node at (0,0) {\rotatebox{60}{$\times$}};
\node at (0,-0.8) {$S_2$};
  \newdimen\R
  \R=0.8cm
  \coordinate (center) at (0,0);
 \draw (45:\R)
     \foreach \x in {45,135,225,315} {  -- (\x:\R) }
              -- cycle (315:\R)
              -- cycle (225:\R) 
              -- cycle (135:\R) 
              -- cycle (45:\R);
	    \draw (45:\R) to[out=220,in=135] (-0.1,-0.1);
			\draw (-0.1,-0.1) to[out=315,in=222] (45:0.8);
	    \draw[line width=0.4pt](45:\R)   .. controls (-0.6, 0.6) and (-0.6,-0.6)  .. 
	    (225:\R);	
	    \draw[line width=0.4pt](45:\R) ..controls (0.6,0.6) and (0.6,-0.6) .. (225:\R);
			\draw[blue, line width=1pt, rotate=-90](135:\R) to[out=280,in=180] 
			(0,-0.25);
			\draw[blue, line width=1pt, rotate=-90](0,-0.25) to[out=0,in=250] (45:0.8);
		\node at (0.6,0.7) {$v_0$};
\end{scope}	
\begin{scope}[xshift=4cm]
\node at (0,0) {\rotatebox{60}{$\times$}};
\node at (0,-0.8) {$N(c\varepsilon b)$};
  \newdimen\R
  \R=0.8cm
  \coordinate (center) at (0,0);
 \draw (45:\R)
     \foreach \x in {45,135,225,315} {  -- (\x:\R) }
              -- cycle (315:\R)
              -- cycle (225:\R) 
              -- cycle (135:\R) 
              -- cycle (45:\R);
	    \draw (45:\R) to[out=220,in=135] (-0.1,-0.1);
			\draw (-0.1,-0.1) to[out=315,in=222] (45:0.8);
	    \draw[line width=0.4pt](45:\R)   .. controls (-0.6, 0.6) and (-0.6,-0.6)  .. 
	    (225:\R);	
	    \draw[line width=0.4pt](45:\R) ..controls (0.6,0.6) and (0.6,-0.6) .. (225:\R);
			\draw[blue, line width=1pt](135:\R) ..controls (45:0.4)    ..  (315:\R);
		\node at (0.6,0.7) {$v_0$};
\end{scope}	
\begin{scope}[xshift=6cm]
\node at (0,0) {\rotatebox{60}{$\times$}};
\node at (0,-0.8) {$N(\varepsilon)$};
  \newdimen\R
  \R=0.8cm
  \coordinate (center) at (0,0);
 \draw (45:\R)
     \foreach \x in {45,135,225,315} {  -- (\x:\R) }
              -- cycle (315:\R)
              -- cycle (225:\R) 
              -- cycle (135:\R) 
              -- cycle (45:\R);
	    \draw (45:\R) to[out=220,in=135] (-0.1,-0.1);
			\draw (-0.1,-0.1) to[out=315,in=222] (45:0.8);
	    \draw[line width=0.4pt](45:\R)   .. controls (-0.6, 0.6) and (-0.6,-0.6)  .. 
	    (225:\R);	
	    \draw[line width=0.4pt](45:\R) ..controls (0.6,0.6) and (0.6,-0.6) .. (225:\R);
			\draw[blue, line width=1pt, rotate=180] (45:\R) to[out=220,in=135] 
			(-0.1,-0.1);
			\draw[blue, line width=1pt, rotate=180] (-0.1,-0.1) to[out=315,in=222] 
			(45:0.8);
		\node at (0.6,0.7) {$v_0$};
\end{scope}	
\begin{scope}[xshift=8cm]
\node at (0,0) {\rotatebox{60}{$\times$}};
\node at (0,-0.8) {$N(a)$};
  \newdimen\R
  \R=0.8cm
  \coordinate (center) at (0,0);
 \draw (45:\R)
     \foreach \x in {45,135,225,315} {  -- (\x:\R) }
              -- cycle (315:\R)
              -- cycle (225:\R) 
              -- cycle (135:\R) 
              -- cycle (45:\R);
	    \draw (45:\R) to[out=220,in=135] (-0.1,-0.1);
			\draw (-0.1,-0.1) to[out=315,in=222] (45:0.8);
	    \draw[line width=0.4pt](45:\R)   .. controls (-0.6, 0.6) and (-0.6,-0.6)  .. 
	    (225:\R);	
	    \draw[line width=0.4pt](45:\R) ..controls (0.6,0.6) and (0.6,-0.6) .. (225:\R);
			\draw[blue, line width=1pt](135:\R) ..controls (225:0.4)    ..  (315:\R);
		\node at (0.6,0.7) {$v_0$};
\end{scope}	
\begin{scope}[yshift=-2cm]
\node at (0,0) {\rotatebox{60}{$\times$}};
\node at (0,-0.8) {$N(c\varepsilon c^{-1})$};
  \newdimen\R
  \R=0.8cm
  \coordinate (center) at (0,0);
 \draw (45:\R)
     \foreach \x in {45,135,225,315} {  -- (\x:\R) }
              -- cycle (315:\R)
              -- cycle (225:\R) 
              -- cycle (135:\R) 
              -- cycle (45:\R);
	    \draw (45:\R) to[out=220,in=135] (-0.1,-0.1);
			\draw (-0.1,-0.1) to[out=315,in=222] (45:0.8);
	    \draw[line width=0.4pt](45:\R)   .. controls (-0.6, 0.6) and (-0.6,-0.6)  .. 
	    (225:\R);	
	    \draw[line width=0.4pt](45:\R) ..controls (0.6,0.6) and (0.6,-0.6) .. (225:\R);
			\draw[blue, line width=1pt, rotate=-90] (45:\R) to[out=220,in=135] 
			(-0.1,-0.1);
			\draw[blue, line width=1pt, rotate=-90] (-0.1,-0.1) to[out=315,in=222] 
			(45:0.8);
		\node at (0.6,0.7) {$v_0$};
\end{scope}	
\begin{scope}[xshift=2cm, yshift=-2cm]
\node at (0,0) {\rotatebox{60}{$\times$}};
\node at (0,-0.8) {$N(\varepsilon b)$};
  \newdimen\R
  \R=0.8cm
  \coordinate (center) at (0,0);
 \draw (45:\R)
     \foreach \x in {45,135,225,315} {  -- (\x:\R) }
              -- cycle (315:\R)
              -- cycle (225:\R) 
              -- cycle (135:\R) 
              -- cycle (45:\R);
	    \draw (45:\R) to[out=220,in=135] (-0.1,-0.1);
			\draw (-0.1,-0.1) to[out=315,in=222] (45:0.8);
	    \draw[line width=0.4pt](45:\R)   .. controls (-0.6, 0.6) and (-0.6,-0.6)  .. 
	    (225:\R);	
	    \draw[line width=0.4pt](45:\R) ..controls (0.6,0.6) and (0.6,-0.6) .. (225:\R);
			\draw[blue, line width=1pt, rotate=90](135:\R) to[out=280,in=180] 
			(0,-0.25);
			\draw[blue, line width=1pt, rotate=90](0,-0.25) to[out=0,in=250] (45:0.8);
		\node at (0.6,0.7) {$v_0$};
\end{scope}	
\begin{scope}[xshift=4cm, yshift=-2cm]
\node at (0,0) {\rotatebox{60}{$\times$}};
\node at (0,-0.8) {$N(c\varepsilon )$};
  \newdimen\R
  \R=0.8cm
  \coordinate (center) at (0,0);
 \draw (45:\R)
     \foreach \x in {45,135,225,315} {  -- (\x:\R) }
              -- cycle (315:\R)
              -- cycle (225:\R) 
              -- cycle (135:\R) 
              -- cycle (45:\R);
	    \draw (45:\R) to[out=220,in=135] (-0.1,-0.1);
			\draw (-0.1,-0.1) to[out=315,in=222] (45:0.8);
	    \draw[line width=0.4pt](45:\R)   .. controls (-0.6, 0.6) and (-0.6,-0.6)  .. 
	    (225:\R);	
	    \draw[line width=0.4pt](45:\R) ..controls (0.6,0.6) and (0.6,-0.6) .. (225:\R);
			\draw[blue, line width=1pt, rotate=180](135:\R) to[out=280,in=180] 
			(0,-0.25);
			\draw[blue, line width=1pt, rotate=180](0,-0.25) to[out=0,in=250] (45:0.8);
		\node at (0.6,0.7) {$v_0$};
\end{scope}	
\begin{scope}[xshift=6cm, yshift=-2cm]
\node at (0,0) {\rotatebox{60}{$\times$}};
\node at (0,-0.8) {$N(b^{-1}\varepsilon b)$};
  \newdimen\R
  \R=0.8cm
  \coordinate (center) at (0,0);
 \draw (45:\R)
     \foreach \x in {45,135,225,315} {  -- (\x:\R) }
              -- cycle (315:\R)
              -- cycle (225:\R) 
              -- cycle (135:\R) 
              -- cycle (45:\R);
	    \draw (45:\R) to[out=220,in=135] (-0.1,-0.1);
			\draw (-0.1,-0.1) to[out=315,in=222] (45:0.8);
	    \draw[line width=0.4pt](45:\R)   .. controls (-0.6, 0.6) and (-0.6,-0.6)  .. 
	    (225:\R);	
	    \draw[line width=0.4pt](45:\R) ..controls (0.6,0.6) and (0.6,-0.6) .. (225:\R);
			\draw[blue, line width=1pt, rotate=90] (45:\R) to[out=220,in=135] 
			(-0.1,-0.1);
			\draw[blue, line width=1pt, rotate=90] (-0.1,-0.1) to[out=315,in=222] 
			(45:0.8);
		\node at (0.6,0.7) {$v_0$};
\end{scope}
	\end{tikzpicture}
\caption{The nine $E$-rigid representations of $\Lambda$ with respect to the 
triangulation $\tau$ on $\Sigma_3$.}
\label{FigString}
\end{figure}
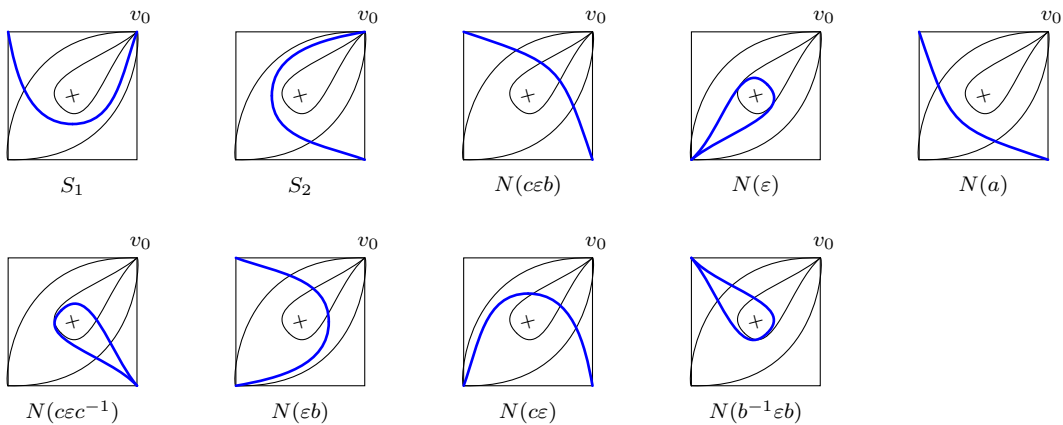 

\begin{flalign*}
&\mathcal{C}_{\Lambda}(N(\varepsilon b))=\frac{y_1^2+ y_1y_2+y_2^2 
+y_2y_3}{y_1y_3}, && \mathcal{C}_{\Lambda}(S_1)=\frac{y_2+y_3}{y_1},
\\
&\mathcal{C}_{\Lambda}(N(c\varepsilon ))=\frac{y_1y_3 + y_1^2+ 
y_1y_2+y_2^2 }{y_2y_3}, &&
\mathcal{C}_{\Lambda}(S_2)=\frac{y_1+y_3}{y_2},
\\
&\mathcal{C}_{\Lambda}(N(c\varepsilon b ))=\frac{y_1y_3 + y_1^2+ 
y_1y_2+y_2^2 + y_2y_3 }{y_1y_2y_3}, && 
\mathcal{C}_{\Lambda}(N(a))=\frac{y_2+y_3+y_1}{y_1y_2}, 
\\
&\mathcal{C}_{\Lambda}(N(c\varepsilon c^{-1} ))=\frac{y_3^2 + 2y_1y_3 + 
y_1^2 + y_2y_3+y_1y_2 +y_2^2}{y_2^2y_3}, && \mathcal{C}_{\Lambda}
(N(\varepsilon))=\frac{y_1^2+ y_1y_2+y_2^2}{y_3},
\\
\end{flalign*}
\begin{flalign*}
&\mathcal{C}_{\Lambda}(N(b^{-1}\varepsilon b))=\frac{y_1^2 + y_1y_2 + 
y_2^2 + 2y_2y_3+y_3^2 +y_1y_3}{y_1^2y_3}.&&
\end{flalign*}

 The Caldero-Chapoton functions associated to the non-$E$-rigid strings 
 of $\Lambda$ 
 are
\begin{flalign*}
&\mathcal{C}_{\Lambda}(N(b))=\frac{y_1+y_2+y_3}{y_1},& 
\mathcal{C}_{\Lambda}(S_3)=y_1+y_2,& \ \ \ \ \ \ \ \ 
\ \  \  \ \ \ \ \ \ \ \ \ \ \ \  \ 
\mathcal{C}_{\Lambda}(N(c))=\frac{y_3+y_1+y_2}{y_2},
\end{flalign*}
\begin{flalign*}
&\mathcal{C}_{\Lambda}(N(b^{-1}\varepsilon ))=\frac{y_1^2+ 
y_1y_2+y_1y_3+y_2^2 +y_2y_3}{y_1y_3},&&
\end{flalign*}
\begin{flalign*}
&\mathcal{C}_{\Lambda}(N(\varepsilon c^{-1} ))=\frac{y_1y_3 + y_1^2+ 
y_1y_2+y_2y_3+y_2^2}{y_2y_3},&&
\end{flalign*}
\begin{flalign*}
&\mathcal{C}_{\Lambda}(N(b^{-1}\varepsilon c^{-1} ))=\frac{ 
y_1y_3+y_1^2+y_1y_2+y_2^2+y_2y_3+ y_3^2+y_1y_3+y_2y_3}{y_1y_2y_3}.&&
\end{flalign*}

\begin{remark}\label{ExPro1}
According to Proposition \ref{PropoGenerate} we have that the Caldero-Chapoton 
functions of indecomposable $E$-rigid representations generate the remaining 
Caldero-Chapoton functions. In this case we get the following relations
\begin{align*} 
\mathcal{C}_{\Lambda}(S_3)&=\mathcal{C}_{\Lambda}(\mathcal{S}_1^{-})+\mathcal{C}
_{\Lambda}(\mathcal{S}_2^{-}),\\ 
\mathcal{C}_{\Lambda}(N(b))&=\mathcal{C}_{\Lambda}(S_1)+1,\\
\mathcal{C}_{\Lambda}(N(c))&=\mathcal{C}_{\Lambda}(S_2)+1,\\
\mathcal{C}_{\Lambda}(N(b^{-1}\varepsilon ))&=\mathcal{C}_{\Lambda}(N(\varepsilon b )) 
+1,\\
\mathcal{C}_{\Lambda}(N(\varepsilon c^{-1} ))&=\mathcal{C}_{\Lambda}(N(c\varepsilon  ))
+1,\\
\mathcal{C}_{\Lambda}(N(b^{-1}\varepsilon c^{-1} ))&=\mathcal{C}_{\Lambda}(N(c
\varepsilon b ))+\mathcal{C}_{\Lambda}(N(a)).
\end{align*}
These relations correspond to the procedure of the proof of Proposition 
\ref{PropoGenerate}. With the notation of \cite[Example 9.4.2]{CLS-15} we can define 
the following isomorphism of the corresponding Caldero-Chapoton algebras $\varphi:
\mathcal{A}_{\Lambda(\tau)}\rightarrow\mathcal{A}_{\Lambda(\tau_0)}$. We set 
$y_1\mapsto x_1$, $y_2\mapsto C_{\Lambda(\tau_0)}(2)=\frac{x_1+x_3}{x_2}$ and 
$y_3\mapsto x_3$. This morphism sends the Caldero-Chapoton function of the arc 
representation associated to one arc of $\tau$ to the Caldero-Chapoton function with 
respect to $\Lambda(\tau_0)$ of the same arc.
\end{remark}

\begin{acknowledgements}
We owe thanks to Andrew T. Carroll, Christof Gei\ss, Jan 
Geuenich, Pierre-Guy Plamondon and Jan Schr\"{o}er for helpful discussions.
The first author was  supported by grants \emph{CONACyT-238754} and 
\emph{PAPIIT-IA102215}. The second author was supported by CONACyT-269197 scholarship.
\end{acknowledgements}

\end{document}